%% file: sample_paper.tex
\documentclass[twoside]{article}

\usepackage[accepted]{aistats2026}
\usepackage{natbib}
\usepackage[utf8]{inputenc} 
\usepackage[T1]{fontenc} 
\usepackage{graphicx}
\usepackage{amsfonts,nccmath,mathtools}
\usepackage{amssymb}
\usepackage{graphicx}
\usepackage{epstopdf}
\usepackage{algorithm}
\usepackage{algorithmic}

\input{defns}

\ifpdf
  \DeclareGraphicsExtensions{.eps,.pdf,.png,.jpg}
\else
  \DeclareGraphicsExtensions{.eps}
\fi

\usepackage{amsopn}

\DeclareMathOperator*{\argmin}{arg\,min}
\newcommand{\1}{\mathbf{1}}
\usepackage{xcolor}
\usepackage{amsthm}

\usepackage[disable]{todonotes}
\usepackage{enumitem}

\usepackage{indentfirst}
\usepackage{bm}

\allowdisplaybreaks
\usepackage{adjustbox}
\usepackage{hyperref} 
\usepackage{cleveref}

\let\argmin\relax

\newtheorem{thm}{Theorem}[section]

\newtheorem{assumption}{Assumption}
\newtheorem{lemma}{Lemma}

\newtheorem{remark}{Remark}

\counterwithin*{step}{section}

\newtheorem{example}{Example}

\crefname{result}{result}{results}
\crefname{result}{step}{steps}
\crefname{assumption}{assumption}{assumptions}
\crefname{lem}{lemma}{lemmas}
\Crefname{lem}{Lemma}{Lemmas}

\begin{document}

\runningtitle{Optimal Local Convergence Rates under Local $\alpha$--P\L}
\runningauthor{Masiha et al.}

%

%

\twocolumn[

\aistatstitle{Optimal Local Convergence Rates of Stochastic First-Order Methods under Local $\alpha$-P\L}

\aistatsauthor{ Saeed Masiha$^{1}$ \And Saber Salehkaleybar$^{2}$ \And  Niao He$^{3}$ \And Negar Kiyavash$^{1}$ \And Patrick Thiran$^{4}$ }

\aistatsaddress{ $^{1}$EPFL School of Management of Technology \\  $^{2}$Leiden University Computer Science institute (LIACS) \\ $^{3}$ETH Department of Computer Science \\ $^{4}$EPFL Department of Computer and Communications Sciences } ]

\begin{abstract}
We study the local convergence rate of stochastic first-order methods under a \emph{local} $\alpha$–Polyak--{\L}ojasiewicz ($\alpha$–P{\L}) condition in a neighborhood of a target connected component $\mathcal M$ of the local minimizer set. The parameter $\alpha\in[1,2]$ is the exponent of the gradient norm in the $\alpha$–P{\L} inequality: $\alpha=2$ recovers the classical P{\L} case, $\alpha=1$ corresponds to H\"older-type error bounds, and intermediate values interpolate between these regimes. Our performance criterion is the number of oracle queries required to output $\hat x$ with $F(\hat x)-l\le\varepsilon$, where $l:=F(y)$ for any $y\in\mathcal M$. We work in a local regime where the algorithm is initialized near $\mathcal M$ and, with high probability, its iterates remain in that neighborhood.
We establish a lower bound $\Omega(\varepsilon^{-2/\alpha})$ for all stochastic first-order methods in this regime, and we obtain a matching upper bound $\mathcal O(\varepsilon^{-2/\alpha})$ for $1\le \alpha<2$ via a SARAH-type variance-reduced method with time-varying batch sizes and step sizes. 
In the convex setting, assuming a local $\alpha$–P{\L} condition on the $\varepsilon$-sublevel set, we further show a complexity lower bound $\widetilde{\Omega}(\varepsilon^{-2/\alpha})$ for reaching an $\varepsilon$-global optimum, matching the $\varepsilon$-dependence of known accelerated stochastic subgradient methods.
\end{abstract}

\section{Introduction}
Stochastic first-order methods are the workhorses of modern large-scale learning \citep{bottou2018optimization}. Classical analyses for such methods typically assume global convexity or impose global growth conditions \citep{karimi2016linear}. However, such assumptions do not reflect the training landscape of deep and structured models, where the objectives are highly nonconvex, yet once the iterates enter a favorable local region, simple gradient-based schemes often converge rapidly. A more realistic assumption is perhaps \emph{local} growth conditions, such as the Polyak--{\L}ojasiewicz (P{\L}) and Kurdyka--{\L}ojasiewicz (K{\L}) inequalities, or their H\"older-type generalizations \citep{attouch2010proximal,drusvyatskiy2016errorbound,li2018kl,jiang2022ebkl}. 
Unlike strong convexity, these conditions do not imply the uniqueness of the minimizer. This distinction is crucial in deep learning, where permutation and rescaling symmetries, as well as over-parameterization, typically results in connected sets of minimizers rather than a single point  \citep{freeman2016topology,simsekli2021geometry,neyshabur2015pathsgd,garipov2018modeconnectivity,draxler2018essentially}.
We formalize this behavior by working with a \emph{local $\alpha$-P{\L}} condition: there exist $\tau>0$, $\alpha\in[1,2]$, and a neighborhood $\mathcal U$ of the target component $\mathcal M$ of the local minimizers of $F$ such that, with $l:=F(y)$ for any $y\in\mathcal M$, $F(x)-l \le \tau\|\nabla F(x)\|^\alpha$ for all $x\in\mathcal U$.\\
In this work, we consider the following setting: Fix a connected component $\mathcal M$ of local minimizers and a neighborhood $\mathcal U\supset\mathcal M$ where $F$ satisfies the local $\alpha$-P{\L} inequality and ask the question:
\vspace{-\baselineskip}
\begin{quote}
\textit{How does the exponent $\alpha$ quantitatively influence the optimal number of oracle calls needed to output $\hat x\in\mathcal U$ with
$\mathbb{E}[F(\hat x)]-l\le \varepsilon$?}
\end{quote}
\vspace{-\baselineskip}
\paragraph{Setting.}
We work with a \emph{batch-smooth} stochastic first-order oracle (SFO), denoted by $\mathsf O^{\tilde L}_{\sigma}$, which (i) returns unbiased gradient estimates with variance bounded by $\sigma^{2}$  (Eq.~\eqref{eq_stoch_1st_order_oracle}, Sec.~\ref{setup_noncvx}), (ii) is $\tilde{L}$-mean-square smooth across query points (Eq.~\eqref{eq-L-avg-smooth}, Sec.~\ref{setup_noncvx}), and (iii) upon $K$ simultaneous queries returns gradients that share the same randomness (Eq.~\eqref{oracle_class}, Sec.~\ref{setup_noncvx}).\\
Let $\mathsf A$ be any stochastic first-order method interacting with $O\in\mathsf O^{\tilde L}_{\sigma}$. 
Assume $\mathsf A$ is initialized sufficiently close to a target connected component $\mathcal M$ of local minimizers and satisfies a \emph{stay-in-neighborhood} property: with high probability, all iterates remain in a fixed neighborhood $\mathcal U$ of $\mathcal M$. This assumption is standard in analyses of local growth/error–bound conditions and can be justified for gradient-type methods under local smoothness and bounded-variance noise \citep{weissmann2025almost,mertikopoulos2020almost}.\\
The \emph{oracle complexity} of $\mathsf A$ is the number of oracle calls required to obtain $\hat x$ with $\mathbb E[F(\hat x)]-l\le \varepsilon$, where $l:=F(y)$ for every $y\in\mathcal M$. We assume that, in a neighborhood of $\mathcal M$, the objective $F$ is locally Lipschitz and locally smooth, and that it satisfies a local $\alpha$–P{\L} condition.

Throughout, we focus on $\alpha\in[1,2]$; the case $\alpha=2$ recovers the classical P{\L} inequality \citep{karimi2016linear}, while $\alpha=1$ corresponds to gradient-dominance with exponent~$1$ \citep{masiha2022stochastic,liu2020improved}.
\paragraph{Motivation for studying local $\alpha$-P{\L} setting.}
The local $\alpha$-P{\L} assumption holds for a wide class of objectives. By the Kurdyka--{\L}ojasiewicz (K{\L}) theory, \emph{every real-analytic function} satisfies a local K{\L} inequality around each critical point; in particular, for any critical point $x_\star$ there exist $\tau>0$, $\alpha\in(1,2]$\footnote{Note that the exponent $\alpha$ is local and may vary from point to point.}, and a neighborhood $\mathcal N$ of $x_\star$ such that $F(x)-F(x_\star)\ \le\ \tau\,\|\nabla F(x)\|^{\alpha}$ for all $x\in\mathcal N$.
This follows from the classical {\L}ojasiewicz gradient inequality and modern K{\L} extensions \citep{bolte2007lojasiewicz,attouch2010proximal,kurdyka1998gradients}. In particular, a nondegenerate local minimum typically has $\alpha=2$, recovering the classical P{\L} regime \citep{attouch2010proximal}. When the Hessian is rank-deficient and the first nonzero term in the Taylor expansion is of order $p>2$, the local P\L\ exponent satisfies $\alpha=\frac{p}{p-1}\in(1,2)$; see Appendix~\ref{append:gradient-dominated functions} for more details and closed-form examples (such as isolated and manifold minimizers, and a deep-linear product loss).

The K{\L} property extends to broad \emph{tame} classes—semialgebraic/\emph{definable}\footnote{see Appendix~\ref{app:kl-examples} for the definition.} objectives—that encompass many standard ML losses \citep{bolte2007clarke,attouch2010proximal}. 
They include least-squares/ridge, logistic or cross-entropy, ReLU networks (often with $\ell_1/\ell_2$ regularization), and polynomial matrix factorization (see Appendix~\ref{app:kl-examples} for more details). In all these cases the empirical risk is definable, hence satisfies a local K{\L} inequality and therefore exhibits local $\alpha$-P{\L} growth near critical points (see related references in Appendix~\ref{app:kl-examples}). Moreover, in \Cref{sec:experiment}, we empirically validate the local $\alpha$–P{\L} property on two training losses for binary classification and dictionary learning.

\paragraph{Contributions.}
We show that the exponent $\alpha$ in the local $\alpha$-P{\L} condition determines the $\varepsilon$–dependence of the oracle/sample complexity. In particular, to output an $\varepsilon$-local optimum $\hat x$ (i.e., $F(\hat x)-l\le\varepsilon$ where $l$ is the target value of function\footnote{In the nonconvex setting, let \(l:=F(y)\) for any \(y\in\mathcal M\), where \(\mathcal M\) is a connected component of local minimizers. In the convex setting, set \(l:=\min_{y\in\mathbb R^d} F(y)\).}) on average, the optimal number of oracle calls is
\(
\Theta\big(\varepsilon^{-2/\alpha}\big)
\) for $1\le \alpha\le 2$.
This tight rate holds both a \emph{nonconvex} setting with a local convergence in a neighborhood of a connected component of minimizers, and in a \emph{convex} setting with a global convergence to an $\varepsilon$-optimal solution.

\textbf{Nonconvex setting (local convergence):} \emph{(i) Lower bound:} Under local Lipschitzness, local smoothness, and local $\alpha$-P\L\ near any connected component $\mathcal M$ of local minimizers (see \Cref{assum_PL_alpha}, Sec.~\ref{sec:local_PL}), any stochastic first-order algorithm with a batch-smooth SFO requires $\Omega\big(\varepsilon^{-2/\alpha}\big)$ oracle calls for $1\le \alpha\le2$ in order to reach an $\varepsilon$-local optimum on average, i.e., finding $\hat{x}$ such that $\mathbb{E}[F(\hat{x})]-l\le \varepsilon$. We derive the lower bound by reducing the search for an $\varepsilon$-local optimum to a sequential hypothesis testing problem with noisy observations. We subsequently establish a connection between the probability of error in the hypothesis testing problem and the complexity of finding an $\varepsilon$-local optimum point.
\emph{(ii) Upper bound:}\footnote{When $\alpha=2$, SGD with a decaying step-size attains an $\varepsilon$-local optimum \emph{in expectation} using $\mathcal{O}(\varepsilon^{-1})$ oracle calls \citep{weissmann2025almost}. Thus, for $\alpha=2$, an $\mathcal{O}(\varepsilon^{-1})$ complexity upper bound is available.} We show that the lower bound is tight by proving that a variance-reduced method (a SARAH-style scheme \citep{nguyen2017sarah}) with varying step-sizes and batch-sizes in \Cref{alg:SARAH} achieves $\mathcal{O}\big(\varepsilon^{-2/\alpha}\big)$ for $1\le \alpha<2$. The analysis establishes a high-probability ``stay-in-neighborhood'' property for the updates of the algorithm and then proves convergence of $F(x_t)$ to the level $l$.\\
Technically, localization for SARAH requires controlling non-martingale cross-terms induced by a biased gradient estimator via a coupled choice of time-varying step sizes and batch sizes.
\textbf{Convex setting (global convergence):} For convex functions, under local $(\alpha,\tau,\varepsilon)$-P\L\ property (see Assumption \ref{assum_local_PL_alpha}), we provide a lower bound $\tilde{\Omega}(\tau^{2/\alpha}\varepsilon^{-2/\alpha})$\footnote{In this paper we use $\tilde{\mathcal{O}}$ and $\tilde{\Omega}$ to hide poly-logarithmic factors.} ($1\le \alpha\le 2$) for first-order algorithms with a stochastic first-order oracle and bounded stochastic gradients\footnote{This is a standard assumption in stochastic convex non-smooth optimization \citep{xu2017stochastic,yang2018rsg}.} in order to reach an $\varepsilon$-global-optimum point. We establish this bound by a reduction to the noisy binary search (NBS) problem \citep{karp2007noisy}.
    When $\alpha\in[1,2]$, this lower bound matches the oracle complexity of accelerated stochastic subgradient methods \cite{xu2017stochastic} in terms of dependence on $\varepsilon$ and $\tau$.

To the best of our knowledge, this is the first work to study the \emph{optimal} local convergence rates of stochastic first-order algorithms under a \emph{local} $\alpha$-P{\L} condition. Related results on upper and lower bounds under (global) P{\L} assumptions are summarized in Appendix~\ref{related_work}.

\paragraph{Organization.}
The rest of the paper is organized as follows: \Cref{sec:local_PL} provides preliminaries and definitions regarding local $\alpha$-P\L{}. \Cref{sec:LB_under_local_PL} gives lower and upper bounds on the oracle complexity of stochastic first-order methods under local $\alpha$-P\L{} for $1\le \alpha<2$. The lower bound for the stochastic first-order methods under convexity and local $(\alpha,\tau,\varepsilon)$-P\L\ property is given in Section~\ref{sec_Lower bound for stochastic convex}. \Cref{sec:experiment} provides the empirical evaluation of local $\alpha$-P\L{}. Full proofs of our main results, together with additional discussion, appear in the appendices.
\section{Preliminaries}\label{sec:local_PL}
For a differentiable function \(F:\mathbb{R}^{d}\to\mathbb{R}\), define
\[
\mathcal{M}_{F} \;:=\; \big\{\,x\in\mathbb R^{d}:\ x \text{ is a local minimizer of } F\,\big\}.
\]
\begin{assumption}[Compactness]\label{assump_comp_local_min}
We assume that the collection of all local minima \(\mathcal{M}_{F}\) is contained within a compact set.
\end{assumption}
We now formalize the local geometry assumption used throughout this work.
\begin{assumption}[Local $\alpha$-P\L]\label{assum_PL_alpha}
    Let $F:\mathbb{R}^{d}\to\mathbb{R}$ be a continuously differentiable function.
Fix $\alpha\in[1,2]$. We say that $F$ satisfies the local $\alpha$-P\L{} condition if there exists a constant
$\tau>0$\footnote{We assume a single constant $\tau>0$ that works uniformly across all connected components of local minimizers. In Appendix~\ref{append_exist_univ_tau}, we show that such a universal $\tau$ exists provided each component $\mathcal M\subset\mathcal{M}_{F}$ admits its own constant $\tau_{\mathcal M}>0$.} such that for every isolated connected component $\mathcal M\subseteq \mathcal{M}_{F}$,
there exists an open neighborhood $\mathcal N(\mathcal M)$ for which
\begin{equation}\label{eq_local_PL}
F(x)-l(\mathcal M) \;\le\; \tau\,\|\nabla F(x)\|^{\alpha}\quad\forall x\in \mathcal N(\mathcal M),
\end{equation}
where $l(\mathcal M):=F(y)$ for any $y\in\mathcal M$\footnote{W.L.O.G, we assume that $F(y)>l(\mathcal{M})$ for all $y\in\mathcal{N}(\mathcal{M})\setminus \mathcal{M}$, as $\mathcal M$ is an isolated
compact connected set of local minima (otherwise we choose $\mathcal{N}(\mathcal{M})$ small enough).}.

\end{assumption}
\paragraph{Relevant ML examples of local \texorpdfstring{$\alpha$}{alpha}-P{\L} functions.}\label{sec:why-alpha-pl}
The local $\alpha$–P{\L} inequality \eqref{eq_local_PL} quantifies how ``flat'' the objective is near a target connected component $\mathcal M$ of local minimizers. 
By the Kurdyka–{\L}ojasiewicz (K{\L}) inequality, many standard ML losses admit such a \emph{local} bound; see Appendix~\ref{app:kl-examples} for a brief summary. 
As a concrete example, for a one-dimensional $d$-layer linear neural network and a single data point $(a,b)$, the squared loss
\(
F(x)\;=\;\big((\prod_{i=1}^{d}x_{i})\,a - b\big)^2
\)
satisfies a local $\alpha$–P{\L} inequality with exponent $\alpha=\tfrac{2d}{2d-1}\in(1,2)$ (see further examples in Appendix~\ref{append:gradient-dominated functions}). We further empirically validate the local $\alpha$–P{\L} inequality in \Cref{sec:experiment} for two training losses used for binary classification and dictionary learning on synthetic data.



\begin{remark}
In \Cref{assum_PL_alpha}, we impose the $\alpha$–P{\L} condition only in neighborhoods of
\emph{isolated} connected components of the set of local minimizers. Concretely, if
\(\mathcal M\) is a connected component of $\mathcal{M}_{F}$ that is isolated, then there exists an open neighborhood \(\mathcal N(\mathcal M)\) containing no other minimizer.
In particular, isolation induces a positive function–value margin (\(F(x)-l>0\) for \(x\in \mathcal N(\mathcal M)\setminus \mathcal M\)), which is key to the  ``stay-in-neighborhood'' arguments for SARAH-type updates (Appendix~\ref{append_proof_high_prob_stab}). Moreover, if \(F\in\mathcal C^2\) and every critical point outside a neighborhood of \(\mathcal M\) is a strict maximum (i.e., \(\nabla^2 F(x)\prec 0\)), then the set of local minimizers is connected and coincides with the set of global minimizers \citep[Proposition~2.1]{masiha2025superquantile} and therefore the connected set of minimizers is isolated. Empirically, connected minimizer sets are frequently observed in over-parameterized models \citep{draxler2018essentially,nguyen2019connected,kuditipudi2019explaining}; in such regimes, the isolation requirement is essentially redundant.
\end{remark}
For a set $\mathcal{S}\subset\mathbb R^{d}$ and $R>0$, define a \emph{tube} with radius $R$ around set $\mathcal{S}$ as $\mathbb B_{2}(\mathcal{S};R)\;:=\;\big\{x\in\mathbb R^{d}:\ \mathrm{dist}(x,\mathcal{S})\le R\big\},$
where $\mathrm{dist}(x,\mathcal{S}):=\inf_{y\in \mathcal{S}}\|x-y\|$.
In the following lemma, using the compactness and isolation of the set of local minimizers of $F$, we show that there exists a tube $\mathbb B_2(\mathcal M;R)\subseteq \mathcal N(\mathcal M)$ for $\mathcal N(\mathcal M)$ defined in \Cref{assum_PL_alpha}.
\begin{lemma}\label{lem_pl_imply_tube_r}
    Suppose that for an isolated connected component $\mathcal M\subseteq\mathcal{M}_{F}$, there exists an open neighborhood $\mathcal N(\mathcal M)$ on which $F$ satisfies the local $\alpha$-P{\L} inequality \eqref{eq_local_PL}, and that $\mathcal M$ is compact. Then there exists $R>0$ with
$\mathbb B_2(\mathcal M;R)\subseteq \mathcal N(\mathcal M)$.
\end{lemma}
The proof is given in Appendix~\ref{append_tube}.


\begin{remark}
Our analysis concerns the \emph{post-transient} phase (after the iterates enter a basin of attraction), where rates are determined by the geometry near minimizers. Empirically, dynamics in modern over-parameterized models are typically observed not to remain at saddles or maxima~\citep{dauphin2014saddle}. Theory supports this picture: (i) gradient descent with random initialization almost surely avoids strict saddles \citep{lee2016grad}; (ii) small perturbations enable polynomial-time escape from saddle neighborhoods \citep{jin2017pgd}; and (iii) the stochasticity of SGD behaves as a Langevin-type diffusion that helps leave saddle regions and reach basins~\citep{mandt2017sgd,raginsky2017nonconvex}. It is therefore natural to impose the $\alpha$–P{\L} condition \emph{locally} around minimizers rather than at arbitrary critical points.
\end{remark}


We assume the following local Lipschitzness and smoothness properties for the objective function $F$.
\begin{assumption}[Local Lipschitzness]\label{lip_assum}
    A function $F:\mathbb{R}^d\to\mathbb{R}$ is \emph{locally Lipschitz} if for every compact set $\mathcal{K}\subset\mathbb{R}^d$, there exists a constant $L_{\mathcal{K}}\ge0$ such that $|F(x)-F(y)| \;\le\; L_{\mathcal{K}}\,\|x-y\|$ for all $x,y\in \mathcal{K}$.
\end{assumption}
\begin{assumption}[Local Smoothness]\label{lip_grad_assum}
    A differentiable function $F:\mathbb{R}^d\to\mathbb{R}$ is \emph{locally smooth} (has a \emph{locally Lipschitz} gradient) if for every compact set $\mathcal{K}\subset\mathbb{R}^d$ there exists $\hat{L}_{\mathcal{K}}\ge0$ such that $\|\nabla F(x)-\nabla F(y)\| \;\le\; \hat{L}_{\mathcal{K}}\,\|x-y\|$ for all $x,y\in \mathcal{K}$.
\end{assumption}
\subsection{Geometric Notation and Constants}\label{sec:geom-constants}
Fix an isolated connected component $\mathcal M\subseteq\mathcal{M}_{F}$ of local minimizers, and let $l:=F(y)$ for any $y\in\mathcal M$. By \Cref{lem_pl_imply_tube_r}, choose $R>0$ such that $\mathbb B_2(\mathcal M;R)\subseteq \mathcal N(\mathcal M)$ (so the local regularity assumptions hold on $\mathbb B_2(\mathcal M;R)$). We use the nested sets
\[
\begin{aligned}
\mathcal U&:=\mathbb B_2(\mathcal M;R/2),\\
\mathcal U_0&:=\{x\in\mathcal U:\ F(x)-l\le s/2\},\\
\mathcal R&:=\mathbb B_2(\mathcal M;3R/4)\setminus \mathbb B_2(\mathcal M;R/2).
\end{aligned}
\]
Since $\mathcal M$ is compact and isolated, the closure of $\mathcal R$ is compact and disjoint from $\mathcal M$, hence continuity of $F$ implies the positive barrier $s_0:=\inf_{x\in\mathcal R}(F(x)-l)>0$. We fix any $s\in(0,s_0)$; as long as the iterates satisfy $F(x_t)-l<s$, no iterate can enter $\mathcal R$, and therefore they remain in $\mathcal U$.

\section{Upper and lower bounds on local convergence rate under local P\L}\label{sec:LB_under_local_PL}
Let $\mathcal U$ be a sufficiently small neighborhood of the target isolated connected component $\mathcal M$ of local minimizers of $F$ (i.e., $\mathcal{U}\subseteq \mathcal{N}(\mathcal{M})$) so that the local $\alpha$-P\L\ holds, and let $l := F(y)$ for any $y \in \mathcal M$. For $\varepsilon>0$, a (possibly random) output $\hat x$ is an \emph{$\varepsilon$-local optimum relative to $\mathcal M$} if
\begin{align}\label{eq_eps_local_opt}
\mathbb E\!\left[\,F(\hat x)-l \,\middle|\, \{\hat x \in \mathcal U\} \right] \le \varepsilon.
\end{align}
We call $F(\hat x)-l$ the \emph{local optimality gap} relative to $\mathcal M$. 

Let \(\mathcal E\) be the event that (i) the algorithm is initialized in \(\mathcal U\) within a neighborhood of \(\mathcal M\), and (ii) all subsequent iterates remain in \(\mathcal U\).
Our performance metric is the \emph{expected number of oracle queries} required to produce an output \(\hat x\), evaluated \emph{conditional on} \(\mathcal E\).
We later show that for \Cref{alg:SARAH}, if the initialization is sufficiently close to \(\mathcal M\), then \(\mathcal E\) holds with high probability and the returned \(\hat x\) is an \(\varepsilon\)-local optimum relative to \(\mathcal M\).

\subsection{Problem setting}\label{setup_noncvx}
\noindent\textbf{Function class.}
For $\alpha\in[1,2]$, define
\begin{align}\label{func_class}
    \resizebox{0.9\linewidth}{!}{$\mathcal{F}_{\alpha}=\left\{\begin{array}{l}
        \!F:\mathbb{R}^{d}\to \mathbb{R}
    \end{array}\Bigg|\begin{array}{cc}
        F\text{\rm{ is locally Lipschitz,}}\\
         F\text{\rm{ is locally smooth,}}  \\
         F \text{\rm{ is local $\alpha$-P\L,}}\\
         \mathcal{M}_{F} \text{\rm{ is compact}}    \end{array}
    \right\}.$}
\end{align}
Equivalently, $F\in\mathcal{F}_{\alpha}$ if and only if \Cref{assump_comp_local_min,lip_assum,lip_grad_assum,assum_PL_alpha} hold.

\noindent
\noindent\textbf{Batch-smooth stochastic first-order oracle.}
We access a stochastic oracle specified by a measurable map $\gv:\mathbb{R}^d\times\mathcal Z\to\mathbb{R}^d$ and a distribution $P_Z$ on $\mathcal Z$ such that, for all $x,y\in\mathbb{R}^d$, $\mathbb E_{Z\sim P_Z}\!\big[\gv(x,Z)\big] = \nabla F(x),$
\begin{align}
&\mathbb E_{Z\sim P_Z}\!\big[\|\gv(x,Z)-\nabla F(x)\|^2\big] \le \sigma^2, 
\label{eq_stoch_1st_order_oracle}
\\
&\mathbb E_{Z\sim P_Z}\!\big[\|\gv(x,Z)-\gv(y,Z)\|^2\big] \le \tilde L^{\,2}\,\|x-y\|^2.
\label{eq-L-avg-smooth}
\end{align}
The $\tilde{L}$-average smoothness \eqref{eq-L-avg-smooth} is a standard assumption in variance-reduction analyses \citep{fang2018spider,lei2017non,cutkosky2019momentum}. The family of such oracles is denoted by $\mathsf O^{\tilde L}_{\sigma}$.

A \emph{batch} query at points $x^{(1)},\dots,x^{(K)}$ draws a single $Z\sim P_Z$ and returns shared-noise gradients; for every $O\in\mathsf O^{\tilde L}_{\sigma}$,
\begin{equation}\label{oracle_class}
O\big(x^{(1)},\ldots,x^{(K)}\big)
=
\big(\gv(x^{(1)},Z),\ldots,\gv(x^{(K)},Z)\big),
\end{equation}
where $\gv$ satisfies \eqref{eq_stoch_1st_order_oracle}–\eqref{eq-L-avg-smooth}. 

\noindent
\textbf{First-order optimization algorithm.}
A stochastic first-order algorithm $\mathsf A$ initialized at $x_0$ proceeds in rounds $t=1,2,\ldots$. At round $t$, it (possibly adaptively) chooses a batch size $K_t$ and a set of query points $x_t := \big(x_t^{(1)},\ldots,x_t^{(K_t)}\big)\in(\mathbb{R}^d)^{K_t},$
then receives noisy gradients from the batch oracle $O$ in \eqref{oracle_class}. Formally, for $t\ge 1$,
\[
\resizebox{\linewidth}{!}{$x_t \;=\; \mathsf A_t\!\left(
  O\big(x_1^{(1)},\ldots,x_1^{(K_1)}\big),\ \ldots,\ 
  O\big(x_{t-1}^{(1)},\ldots,x_{t-1}^{(K_{t-1})}\big)
\right),$}
\]
where $\mathsf A_t$ is a measurable policy that may depend on $\{x_{0},x_1,\ldots,x_{t-1}\}$ and the algorithm’s internal randomness. We write $\mathcal A_T$ for the class of algorithms that perform at most $T$ oracle rounds.
\begin{assumption}[$(1-\delta)$-stay-in-neighborhood]\label{propert_local_cvg_algor}
Let $F\in\mathcal F_\alpha$ and let $\mathcal M_{F}$ be the set of its local minimizers.
Fix an isolated connected component $\mathcal M\subset \mathcal M_{F}$, a horizon $T\in\mathbb N$, and a $\delta\in(0,1)$.
An algorithm $\mathsf A\in\mathcal A_T$ satisfies ``$(1-\delta)$-stay-in-neighborhood'' assumption on $\mathcal M$ 
if there exist neighborhoods\footnote{A convenient special case is the ``tubular'' choice $\mathcal{U}_0=\mathbb B_2(\mathcal M;r_0)$ and
$\mathcal{U}=\mathbb B_2(\mathcal M;r_1)$ with $0<r_0\le r_1$.} $\mathcal{U}_0$ and $\mathcal{U}$ where $\mathcal{U}_0\subseteq \mathcal{U}$ and $\mathcal M\subset \mathcal{U}_0\subseteq \mathcal{U}$ such that, for any initialization $x_0\in \mathcal{U}_0$, the algorithm’s iterates satisfy
\[
\mathbb P\!\left(x_t\in \mathcal{U}\ \text{for all }t=1,\ldots,T\right)\ \ge\ 1-\delta,
\]
where the probability is taken over the internal randomness of the algorithm and the oracle.
\end{assumption}
\Cref{propert_local_cvg_algor} is natural: empirically, training dynamics in modern ML often enter a basin and thereafter behave \emph{locally}; see, e.g., \cite{li2018visualizing,garipov2018modeconnectivity,draxler2018essentially}. 
In our lower-bound analysis, we restrict our attention to algorithms that satisfy \Cref{propert_local_cvg_algor} on $(\mathcal{U}_0,\mathcal{U})$; equivalently,
for fixed $(\mathcal M,T,\delta)$, we define
\begin{align}\label{eq_def_alg_stay_in_negh}
    &\mathcal{A}_T(\mathcal{U}_0, \mathcal{U},\delta)
:=\nonumber\\
&\resizebox{0.9\linewidth}{!}{$\Bigl\{\mathsf{A}\in \mathcal{A_T}: \forall x_0\in \mathcal{U}_0, 
\mathbb P\bigl(x_t\in \mathcal{U}, \forall t\le T\bigr)\ge 1-\delta\Bigr\}.$}
\end{align}
In the upper-bound analysis, we show that a SARAH-type method (see \Cref{alg:SARAH}) indeed satisfies this assumption.


\subsection{Complexity lower bound}\label{sec:3.2}
The main result of this section is stated in the following theorem.


\begin{thm}\label{lower_bound_non-convex_PL_L-avg-smooth}
Fix $\alpha\in(1,2]$. There exist (i) an instance $(F,O)$ with $F\in\mathcal F_\alpha$ and a batch-smooth SFO $O\in\mathsf O^{\tilde L}_{\sigma}$, (ii) a connected component $\mathcal M\subseteq\mathcal M_F$ of local minimizers with level $l:=F(y)$ for $y\in\mathcal M$, and (iii) constants $\tau>0$, $R>0$, together with an open neighborhood $\mathcal N(\mathcal M)\supseteq \mathbb B_2(\mathcal M;R)$ on which the local $\alpha$–P{\L} condition holds with parameter $\tau$, such that the following holds.

For any first–order algorithm $\mathsf A$ that (i) is initialized at $x_0\in\mathcal N(\mathcal M)$ and (ii) satisfies the stay–in–neighborhood property up to the iteration horizon $T$ with $\mathcal U_0\subseteq \mathcal U\subseteq \mathcal N(\mathcal M)$ (i.e., $\mathsf A\in\mathcal A_T(\mathcal U_0,\mathcal U,\delta)$), the number of oracle queries $m$ required to output $\hat x$ with $\mathbb{E}\!\left[F(\hat x)-l\ \middle|\ \mathcal E_{T}(\mathcal U)\right]\ \le\ \varepsilon$, must satisfy $m=\Omega(\tau^{2/\alpha}\sigma^{2}\epsilon^{-2/\alpha})$ for $\varepsilon\le R^{\alpha}\tau^{1/(\alpha-1)}(\tfrac{\alpha-1}{\alpha})^{\frac{\alpha}{\alpha-1}}$ where $\mathcal E_{T}(\mathcal U):=\{x_k\in\mathcal U\ \text{for all iterates }k\le T\}$.
\end{thm}
\begin{proof}[Sketch of proof]
We follow a standard information–theoretic reduction for minimax lower bounds \citep{raginsky2009information}. The idea is to cast optimization as a noisy binary hypothesis test \citep[Ch.~2]{10.5555/1522486} between two objective functions that are statistically hard to distinguish.\\
We construct two one–dimensional $\mathcal{C}^1$ functions that (i) satisfy local Lipschitzness, local smoothness, and a local $\alpha$–P{\L} condition; (ii) have unique minimizers separated by a distance $\rho$; and (iii) coincide outside a small neighborhood. Under a noisy gradient oracle with variance $\sigma^2$, the two functions induce statistically close distributions: choosing
\(\rho \asymp \sigma^{\alpha-1}/{m^{(\alpha-1)/2}}\)
ensures the Kullback–Leibler divergence between the distributions of $m$ oracle responses is $\mathcal{O}(1)$. Therefore, by Fano’s inequality, any algorithm has a constant probability of outputting a point at least $\rho/2$ distance away from the true minimizer. Using the local error bound implied by the local $\alpha$–P{\L} condition, this distance lower bound translates into a function–value suboptimality (the local optimality gap defined in \eqref{eq_eps_local_opt}) of order \(\frac{\tau\sigma^\alpha}{m^{\alpha/2}}\).
Hence the expected local optimality gap is $\Omega\big(\tau\sigma^\alpha m^{-\alpha/2}\big)$, yielding the claimed complexity lower bound.
\end{proof}

\begin{remark}\label{remark_alpha=1_lower_bound}
The case $\alpha=1$ is not included in \Cref{lower_bound_non-convex_PL_L-avg-smooth}, as it follows from known results.
In particular, \cite{foster2019complexity} proved that, under convexity and smoothness, any stochastic first-order method needs  $\Omega(\varepsilon^{-2})$ oracle calls to find an $\varepsilon$-stationary point in the sense that  $\mathbb{E}[\|\nabla F(x)\|] \le \varepsilon$. 
 In \Cref{proof_remark_alpha=1_lower_bound}, we showed that their hard instance function belongs to our class $\mathcal{F}_{\alpha=1}$: its set of stationary points coincides with the 
set of global minimizers, and it satisfies the local $1$-P\L{}, Lipschitz, and smoothness properties. Moreover, the stochastic gradients  in their construction can be realized by an oracle $O \in \mathsf{O}_{\sigma}^{\tilde{L}}$.  
 Therefore, when $\alpha=1$, their lower bound of $\Omega(\varepsilon^{-2})$ holds in the setting considered in this section.
\end{remark}

\subsection{Complexity upper bound}\label{sec_upper_oracle_complexity}

We introduce \Cref{alg:SARAH}, a SARAH–type variance–reduced method with \emph{time-varying}  batch-sizes and step-sizes. Define the mini-batch estimator
\(
\gv_{\mathcal J}(x)\;:=\;\frac{1}{|\mathcal J|}\sum_{j\in\mathcal J}\gv(x,\xi_j),
\)
where \(\{\xi_j\}_{j\in\mathcal J}\) are i.i.d.\ samples and \(\mathbb{E}[\gv(x,\xi)]=\nabla F(x)\).
The SARAH gradient update is given in line~4 of \Cref{alg:SARAH}. When \Cref{alg:SARAH} is initialized in a small neighborhood of a fixed connected component $\mathcal M$ of local minimizers, the method attains the optimal sample complexity $\Theta(\varepsilon^{-2/\alpha})$ for obtaining an $\varepsilon$-local minimum point relative to $\mathcal{M}$, matching the lower bound in \Cref{lower_bound_non-convex_PL_L-avg-smooth}.

A key technical step is to show that, under a local $\alpha$-P{\L} condition and a batch-smooth oracle, the iterates of \Cref{alg:SARAH} remain with high probability within a sufficiently small neighborhood of the target component $\mathcal M$. Even in the case that $\mathcal{M}=\{x_{\mathrm{loc}}\}$ is a singleton local minimizer, this is a challenging task to keep updates of \Cref{alg:SARAH} $\{x_{t}\}_{t\ge1}$ close to $x_{\mathrm{loc}}$ \emph{without} convexity. Our time-varying batch-sizes and step-sizes ensure small per-iteration steps and control the distance between two consecutive updates; together with the local $\alpha$-P{\L} geometry, this prevents escape from the neighborhood.


Conditioned on the event that the entire trajectory remains in a neighborhood of $\mathcal M$, we establish \emph{function-value} convergence to the level of $\mathcal M$:
\(\lim_{t\to\infty}\mathbb{E}[F(x_t)] = l,\)
together with non-asymptotic bounds on $\mathbb{E}[F(x_t)]-l$ along the run. Since a local $\alpha$-P{\L} condition does not enforce convexity and $\mathcal M$ may be a manifold of minimizers, one should not expect $x_t$ to converge to a single point in $\mathcal M$. In this case, function-value convergence is a natural notion of convergence.\footnote{The function-value convergence is a symmetry-invariant target and does not depend on choosing a representative minimizer. See \citep{liu2023revisiting} for more details.}
Our main theorem, stated next, shows that \Cref{alg:SARAH} indeed achieves the tight upper bound under local $\alpha$-P{\L}.
\begin{algorithm}[t]
\caption{SARAH}\label{alg:SARAH}
\textbf{Input:} Maximum number of iterations $T$, batch sizes $\{n^{t}_{g}\}_{t=1}^{T}$, and the period length $S$, initial point $x_{0}$.
\begin{algorithmic}[1]
\STATE $t\gets 0$
\WHILE{$t\le T-1$}
\STATE Sample index set $\mathcal{J}_{t}$ with $|\mathcal{J}_{t}|=n^{t}_{g}$.
\STATE \[
\resizebox{\linewidth}{!}{$\vv_{t} \gets \begin{cases}
\gv_{\mathcal{J}_{t}}(x_{t}),\quad &\text{mod}(t,S)=0\\
\gv_{\mathcal{J}_{t}}(x_{t})-\gv_{\mathcal{J}_{t}}(x_{t-1})+\vv_{t-1},\quad&\text{else}
\end{cases}$}
\]
\STATE $x_{t+1}\gets x_{t}-\eta_{t}\vv_{t}$
\STATE $t \gets t+1$
\ENDWHILE
\RETURN{$x_{T}$}
\end{algorithmic}
\label{algorithm2}
\end{algorithm}
\begin{thm}\label{thm:cvg_SARAH}
Fix $\delta\in(0,1)$ and let $\mathcal M$ be an isolated connected component of local minima of $F$ with level $l=F(y)$ for all $y\in\mathcal M$. 
Assume $F\in\mathcal F_\alpha$ and choose $R>0$ so that $\mathbb B_{2}(\mathcal M;R)\subseteq\mathcal N(\mathcal M)$ (cf.\ \Cref{assum_PL_alpha}). 
Let the oracle be batch-smooth, $O\in\mathsf O^{\tilde L}_{\sigma}$, and run \Cref{alg:SARAH} from $x_0$ with step sizes $\eta_t=\eta_0 (t+1)^{-\frac{\alpha}{2}-(2-\alpha)\mathsf{x}}$ and batch-sizes $n_g^t=(t+1)^{1-2\mathsf{x}}$, for an arbitrary small value $\mathsf{x}>0$.
Denote the iterates by $\{x_t\}_{t\ge1}$ and set $\hat{x}:=x_T$. The following statements hold. 
\begin{itemize}
    \item There exists $s>0$ such that, with \(
\mathcal U:=\bigl\{x:\mathrm{dist}(x,\mathcal M)\le R/2\bigr\},\)
\(\mathcal U_0:=\bigl\{x:\mathrm{dist}(x,\mathcal M)<R/2,\ F(x)-l\le s/2\bigr\},\) and $\eta_{0}=\mathcal{O}(\min\{\sqrt{s},s^{-1}\delta,R\sqrt{\delta}\})$, if $x_0\in\mathcal U_0$, the event \(\mathcal{E}_T(\mathcal{U}):=\{\ x_t\in\mathcal U\ \text{for all }t=1,\dots,T\ \}\) occurs with probability at least $1-\delta$.
\item Let $N:=\mathbb E\!\left[\sum_{t=1}^{T} n_g^t\mid \mathcal{E}_T(\mathcal{U})\right]$ denote the expected total number of oracle queries used up to iteration $T$ given $\mathcal{E}_T(\mathcal{U})$ defined in the first statement. 
Then for $\varepsilon\le s/2$, $\mathbb E\!\left[F(\hat{x})-l\mid \mathcal{E}_T(\mathcal{U})\right]\le \varepsilon$, with \(N\ =\ \mathcal O\big(\varepsilon^{-2/\alpha}\big)\).
\end{itemize}
\end{thm}
\begin{remark}[On tuning and batch sizes]
The $\alpha$-dependent schedules in \Cref{thm:cvg_SARAH} are used to obtain the rate-optimal $\varepsilon$-dependence; if $\alpha$ is unknown, one may use a generic polynomially decaying step size (yielding a suboptimal worst-case rate under the same local assumptions). The increasing mini-batch sizes are likewise used to match the optimal oracle complexity; designing a localized method with comparable guarantees using single-sample (recursive-momentum) gradients is an interesting direction for future work.
\end{remark}
\begin{proof}[Proof sketch]
We show that (i) the iterates \emph{stay} near the target set of minimizers $\mathcal M$ with high probability, and (ii) conditioning on this event, the function value gap $D_t:=F(x_t)-l$ decays at the optimal rate under the local $\alpha$-P{\L}.\\
\noindent\textit{(i) High–probability stay-in-neighborhood of $\mathcal M$.}
Let $R>0$, the neighborhoods $\mathcal U_0\subset\mathcal U$, the forbidden annulus $\mathcal R$, and the barrier levels $0<s<s_0$ be as in \Cref{sec:geom-constants}. Starting from $x_0\in\mathcal U_0$, we ensure (a) that updates are small so the iterates cannot jump across $\mathcal R$, and (b) that the function gaps stay below $s$ so the barrier prevents entrance into $\mathcal R$.
To ensure the event $\mathcal{E}_T(\mathcal{U})=\{x_{t}\in\mathcal{U}\text{ for all $t=1,...,T$}\}$ happens with high probability, we control the growth of the following two events by appropriate choice of our time-varying step/batch sizes. 

\emph{(1) Small–error event $\mathcal E_{\mathrm{err}}$.} 
Following Appendix~\ref{append_proof_high_prob_stab}, we collect all stochastic terms in the descent recursion into an accumulator $\mathsf{R}_n$ (a ``noise budget''). Under our decaying step sizes and growing batch sizes, its expected increments are summable, which yields $\mathbb P(\mathcal E_{\mathrm{err}})\ge 1-\delta/2$ for $\mathcal E_{\mathrm{err}}:=\{\mathsf{R}_k<s,\ \forall k\le T\}$. In turn, $\mathcal E_{\mathrm{err}}$ implies $D_t<s$ for all $t\le T$.

\emph{(2) Small–steps event $\mathcal E_{\mathrm{step}}$.}
Because SARAH estimator (line 4 in \Cref{alg:SARAH}) is bounded on the neighborhood and step-sizes are tuned, each update is small: $\|x_{t}-x_{t-1}\|\le R/4$ for $t\ge1$. Thus no iterate can escape from $\mathbb{B}_{2}(\mathcal{M};3R/4)$ in a single step. We ensure this holds with probability at least $1-\delta/2$.

By a union bound, both events occur simultaneously with probability at least $1-\delta$.
On $\mathcal E_{\mathrm{err}}\cap\mathcal E_{\mathrm{step}}$, the preceding argument prevents the iterates from entering $\mathcal{R}$, hence $\mathcal{E}_T(\mathcal{U})$ holds with probability at least $1-\delta$.


\noindent\textit{(ii) Optimal complexity inside the neighborhood.}
Conditioned on staying in $\mathcal U$ (which holds with probability $\ge 1-\delta$), the same local smoothness and $\alpha$-P{\L} geometry imply the following recursion inequality for $\delta_{t}:=\mathbb{E}[F(x_{t})\mid \mathcal{E}_T(\mathcal{U})]-l$:
\begin{align}\label{eq_0009090}
    \resizebox{0.9\linewidth}{!}{$\delta_{t+1}\le \delta_{t}-\frac{\eta_{t}}{2\tau^{\frac{2}{\alpha}}}\delta_{t}^{\frac{2}{\alpha}}+\frac{\eta_{t}}{2}\mathbb{E}[\|\vv_{t}-\nabla F(x_{t})\|^{2}\mid \mathcal{E}_T(\mathcal{U})].$}
\end{align}    
The first term is a geometric drift induced by local $\alpha$-P{\L}, and the last term is the (decaying) noise level of the SARAH estimator; solving this nonlinear recursion amounts to balancing these two effects (see \Cref{lemma:rec_eq} in Appendix~\ref{append:G3}).
Using the time–varying batch sizes in \Cref{thm:cvg_SARAH} and the SARAH-update of gradients, we ensure that
$
\mathbb{E}[\|\vv_t-\nabla F(x_t)\|^{2}\mid \mathcal{E}_T(\mathcal{U})]\;=\;\mathcal{O}\!\left(\tfrac{1}{t}\right).
$
Together with the step sizes from \Cref{thm:cvg_SARAH}, the recursion in \eqref{eq_0009090} yields $\delta_t \;=\; \mathcal{O}\!\left(t^{-{\alpha}/{2}+\alpha\mathsf{x}}\right)$.
Therefore, to achieve $\delta_t=\mathcal{O}(\varepsilon)$, the total number of stochastic gradient calls (i.e., the sum of all mini–batch sizes) satisfies, on average, $\sum_{t=0}^{n-1} n_g^{\,t}\;=\;\mathcal{O}\!\bigl(\varepsilon^{-2/\alpha}\bigr)$. Consequently, the method attains the \emph{optimal} oracle complexity $\Theta\!\bigl(\varepsilon^{-2/\alpha}\bigr)$, matching the lower bound in \Cref{lower_bound_non-convex_PL_L-avg-smooth}.
\end{proof}
\begin{remark}
Our high–probability stability argument follows the localization scheme used for SGD in \citep{weissmann2025almost,mertikopoulos2020almost} which proves that, once the iterate enters a small basin around a target component $\mathcal M$, it stays there with high probability. However, the proof for \Cref{alg:SARAH} differs in two essential ways: (i) \textit{Different descent dynamics.} For SGD with an \emph{unbiased} gradient estimator $\gv_t$,
the descent inequality features a variance term scaled by $\eta_t^{2}$, of the form $\eta_t^{2}\,\E\!\left[\|\gv_t-\nabla F(x_t)\|^{2}\right]$.
For SARAH, the recursion involves the variance of the \emph{biased} estimator $\vv_t$ with a
\emph{linear} step-size factor, namely $\eta_t\,\E\!\left[\|\vv_t-\nabla F(x_t)\|^{2}\right]$.
The linear (rather than quadratic) scaling makes the stochastic error more difficult to control; noise decays more slowly and requires tighter coupling between step-size and batch-size (see \Cref{lemma_recursion,lemma:rec_eq} in Appendix~\ref{append:G3}). (ii) \textit{Non–martingale cross term.} In the SGD analysis, the inner-product ``cross term'' $\langle \nabla F(x_t),\, \gv_t - \nabla F(x_t)\rangle$ forms a martingale difference with zero conditional mean; its weighted sum is controlled directly from this zero-mean property.
In contrast, SARAH’s estimator $\vv_t$ is \emph{biased}, hence $\langle \nabla F(x_t),\, \vv_t - \nabla F(x_t) \rangle$ is not a martingale difference. Our analysis instead controls the \emph{weighted accumulation} of cross terms
\(
\sum_{t=0}^{T-1} \eta_t\,\big\langle \nabla F(x_t),\, \vv_t - \nabla F(x_t) \big\rangle,
\)
by using the SARAH variance–reduction recursion with carefully chosen,
time–varying step sizes $\{\eta_t\}$ and batch sizes $\{n_g^t\}$. This
lets us bound simultaneously the bias–variance aggregate
\(
\sum_{t=0}^{T-1} \eta_t\,\|\vv_t - \nabla F(x_t)\|^2
\quad\text{and}\quad
\sum_{t=0}^{T-1} \eta_t\,\big\langle \nabla F(x_t),\, \vv_t - \nabla F(x_t) \big\rangle,\)
so that both remain uniformly small with high probability on the ``small–error'' event.
This replaces the martingale step used in SGD and is the key technical difference in our localization analysis; see \Cref{lemm_13_append} in Appendix~\ref{append_proof_high_prob_stab} for the precise high–probability bound.
\end{remark}

\section{Lower bound on global convergence rate under local P\L\ and convexity}\label{sec_Lower bound for stochastic convex}
\vspace{-0.1cm}
In this section, we consider the problem of finding an $\varepsilon$-global-optimum point when the objective function $F:\mathcal{X}\to\mathbb{R}$ is convex and satisfies the local $(\alpha,\tau,\varepsilon)$-P\L\ property (refer to Assumption~\ref{assum_local_PL_alpha}).  Our goal is to find a point $\hat{x}\in \mathcal{X}$ such that $F(\hat{x})-\min_{x\in \mathcal{X}}F(x)\le \varepsilon$
with probability at least $1-\delta$, with access to $F$ through a stochastic first-order oracle with bounded stochastic gradients.

We first summarize the setting we use to establish the complexity lower bound.

\noindent\textbf{Function class.}
\begin{assumption}[Local $(\alpha,\tau,\varepsilon)$-P\L]\label{assum_local_PL_alpha}
Function $F:\mathcal{X}\to \mathbb{R}$ (where $\mathcal{X}\subseteq \mathbb{R}^{d}$) satisfies the local $(\alpha,\tau,\varepsilon)$-P\L\ property when for all $x\in \mathcal{X}\cap\mathcal{S}_{\varepsilon}$, we have
\begin{equation*}
    F(x)-\min_{x\in \mathcal{X}}F(x)\leq \tau \|\nabla F(x)\|^{\alpha},
\end{equation*}
where $\mathcal{S}_{\varepsilon}:=\{x:~F(x)-\min_{x\in \mathcal{X}}F(x)\le \varepsilon\}$, $\tau>0$, and $\alpha\in [1,2]$ are two constants.
\end{assumption}
$\mathcal{F}^{\mathcal{X}}_{\alpha,\tau,\varepsilon}$  includes all convex functions that satisfy Assumptions \ref{assum_local_PL_alpha}, i.e.,
{\small
\begin{align}\label{func_class_convex_PL}
    \mathcal{F}^{\mathcal{X}}_{\alpha,\tau,\varepsilon}=\left\{\begin{array}{l}
        F:\mathcal{X}\to \mathbb{R}\\
        \mathcal{X}\subset \mathbb{R}^{d}
    \end{array}\Bigg|\begin{array}{cc}
         F\text{\rm{ is convex,}}\\
         F \text{\rm{ satisfies $(\alpha,\tau,\varepsilon)$-P\L}}
    \end{array}
    \right\}.
\end{align}
}

\noindent
\textbf{Domain class.} Denote by $\mathbb{S}_{R}$, the class of  convex, closed, and bounded sets in $\mathbb{R}^{d}$ whose diameter $\text{diam}(\mathcal{X})\le R$ for every $\mathcal{X}\in \mathbb{S}_{R}$.

\noindent
\textbf{Stochastic first-order oracle with bounded stochastic gradients.}
We denote a family of stochastic first-order oracles satisfying the following properties by $\mathsf{O}^{G}$: (i) satisfying the conditions in  \eqref{eq_stoch_1st_order_oracle},
and (ii) bounded stochastic gradients, i.e., $\|\gv(x,z)\|\le G$
 for every $x\in \mathcal{X}$ and $z\in\mathcal{Z}$ where $G>0$ is some constant.

We now provide a tight lower bound for the probability-based minimax oracle complexity of the function class $\mathcal{F}^{\mathcal{X}}_{\alpha,\tau,\varepsilon}$ and the family of oracles $\mathsf{O}^{G}$ for stochastic projected first-order methods.
\begin{thm}\label{lower_bound_PL_convex_NBS}
For the family of domain sets $\mathbb{S}_{R}$, there exists a function $F$ in the function class $\mathcal{F}^{\mathcal{X}}_{\alpha,\tau,\varepsilon}$, an oracle $O$ in the family of oracles $\mathsf{O}^{G}$, and $\alpha\in(1,2]$ and $\varepsilon\le \min\{((\alpha-1)/{\alpha})^{\alpha}\tau,1\}$, such that the number of oracle queries for any stochastic first-order algorithm $\mathcal{A}\in\mathcal{A}_{m}$ required to output an estimate $\hat{x}$ satisfying $F(\hat{x})-F^{*}\le \varepsilon$ with probability at least $1-\delta$ is
\begin{align}\label{eq_lb_cvx_local_pl}
   \Omega\left(G^{2}\tau^{\frac{2}{\alpha}}\varepsilon^{-\frac{2}{\alpha}}\cdot\log\left(\frac{2\alpha R}{(\alpha-1)\varepsilon^{\frac{\alpha-1}{\alpha}}\tau^{\frac{1}{\alpha}}}\right)\right).
\end{align}
\end{thm}
\begin{proof}[Sketch of proof]
We reduce from noisy binary search problem (NBS): given a sorted set of $N$ keys and a comparison oracle that is correct with probability $1/2+p$, any (adaptive) strategy needs $\Omega(p^{-2}\log N)$ queries \citep{feige1994computing,karp2007noisy}.
Let $X=[0,R]$ and partition it into $N$ equal intervals $[a_j,a_{j+1})$. The unknown target interval index $j^{*}$
plays the role of the hidden position in NBS.
From the NBS comparison outcomes we define a stochastic gradient oracle $g(x,Z)$ that is (i) unbiased for the gradient of a convex function $F$,
and (ii) uniformly bounded $\|g(x,Z)\|\le G$. The construction ensures that on the $\varepsilon$-sublevel set $S_\varepsilon$, $F(x)-F^{*} \ \le\ \tau\,\|\nabla F(x)\|^{\alpha}$
and that any $\varepsilon$-minimizer must lie in the unique target interval $[a_{j^{*}},a_{j^{*}+1})$,
so solving the optimization problem to accuracy $\varepsilon$ identifies $j^{*}$. We set $p\ \asymp\ \tfrac{\varepsilon^{1/\alpha}}{G\,\tau^{1/\alpha}}$ and $N\ \asymp\ \tfrac{2\alpha}{\alpha-1}\cdot \tfrac{R}{\varepsilon^{(\alpha-1)/\alpha}\,\tau^{1/\alpha}},$
so the gradient signal on $S_\varepsilon$ matches the comparison reliability $p$, and $\varepsilon$-accuracy is equivalent to interval identification.
If a stochastic first-order algorithm uses $T$ oracle calls and outputs $\hat x$ with $F(\hat x)-F^{*}\le\varepsilon$, then it can
solve the NBS instance with at most a constant-factor more queries and the same success probability. Hence substituting our choices of \(p\) and \(N\) into the lower bound $\Omega(p^{-2}\log N)$ gives the result (for $\alpha\in(1,2]$ and suitably small $\varepsilon$ so the construction lies in the stated class).
\end{proof}
\begin{remark}
    Note that every convex function satisfies the local $(\alpha=1,\tau,\varepsilon)$-P\L\ property. It is well-known that for  bounded domain convex functions, stochastic first-order methods achieve a tight lower bound of $\Omega(\varepsilon^{-2})$ with access to $\mathsf{O}^{G}$ oracle \citep{nemirovskij1983problem,agarwal2009information}. Therefore, similarly to Remark \ref{remark_alpha=1_lower_bound}, we did not include the case $\alpha=1$ in the statement of Theorem \ref{lower_bound_PL_convex_NBS}.
\end{remark}

\begin{remark}[Complexity upper bound]
In \cite[Theorem 1]{xu2017stochastic}, the authors showed that for function $F\in \mathcal{F}^{\mathcal{X}}_{\alpha,\tau,\varepsilon}$ and oracle class $\mathsf{O}^{G}$, a constrained version of the Accelerated Stochastic Subgradient Method (see Algorithm 1 in \citep{xu2017stochastic}) guarantees that $F(x_{T})-\min_{x\in \mathcal{X}}F(x)\le \varepsilon$ with probability $1-\delta$, for some $\delta>0$, and $T=\mathcal{O}\left({G^{2}\tau^{{2}/{\alpha}}\cdot\log(1/\delta)\cdot\log\left({\varepsilon^{-(\alpha-1)/{\alpha}}\tau^{-{1}/{\alpha}}}\right)}/{\varepsilon^{{2}/{\alpha}}}\right)$ which matches with our lower bound in \eqref{eq_lb_cvx_local_pl} in terms of dependency on $\varepsilon$, $\tau$, and $G$.    
\end{remark}
\section{Empirical evaluation of local \texorpdfstring{$\alpha$–P\L{}}{alpha–PL}}\label{sec:experiment}
In this section, we numerically study the local $\alpha$–P\L{} condition in two learning tasks: (i) classification with a ReLU network, and (ii) dictionary learning. We also consider real-data versions of these experiments in Appendix~\ref{append_experiment} and additional learning tasks (e.g., polynomial and deep linear matrix factorization) in Appendix~\ref{app:mf_experiments}; in some setups, the estimated envelope at $\alpha=2$ explodes under tighter localization, suggesting that $\tau(\alpha)$ may be unbounded at $\alpha=2$. Let $\hat{x}$ be the last iterate returned by full-batch gradient descent. In our experiments, we consider the runs for which $\|\nabla F(\hat{x})\|$ is small, so $\hat{x}$ is an \emph{almost critical point}.  We consider a small ball $\mathbb{B}_2(\hat{x};R)$ and compute, for $\alpha\in[1,2]$,
\(
\mathsf{R}_\alpha(x)\;:=(F(x)-F(\hat{x}))\|\nabla F(x)\|^{-\alpha}.
\)
 Let $\tau(\alpha)\;:=\;\sup_{x\in \mathbb{B}_2(\hat{x};R)} \mathsf{R}_\alpha(x)$.
In both tasks (classification and dictionary learning), we observe that $\tau(\alpha)$ is finite on $\mathbb{B}_2(\hat{x};R)$ for all $\alpha\in[1,2]$ (see Figure \ref{fig:placeholder}).
Since $\tau(\alpha)$ is evaluated on a fixed-radius neighborhood, this does not by itself prove that $\tau(\alpha)\to\infty$ as $\alpha\to2$; however, we do observe a sharp growth of $\tau(\alpha)$ as $\alpha$ approaches $2$ in both experiments.
Therefore, for any $\alpha\in[1,2]$, by picking $\tau=\tau(\alpha)$ in \Cref{assum_PL_alpha} and $\mathcal{N}(\hat{x})=\mathbb{B}_2(\hat{x};R)$, the local $\alpha$–P\L{} condition is satisfied.
\begin{figure}
    \centering
    \includegraphics[width=0.75\linewidth]{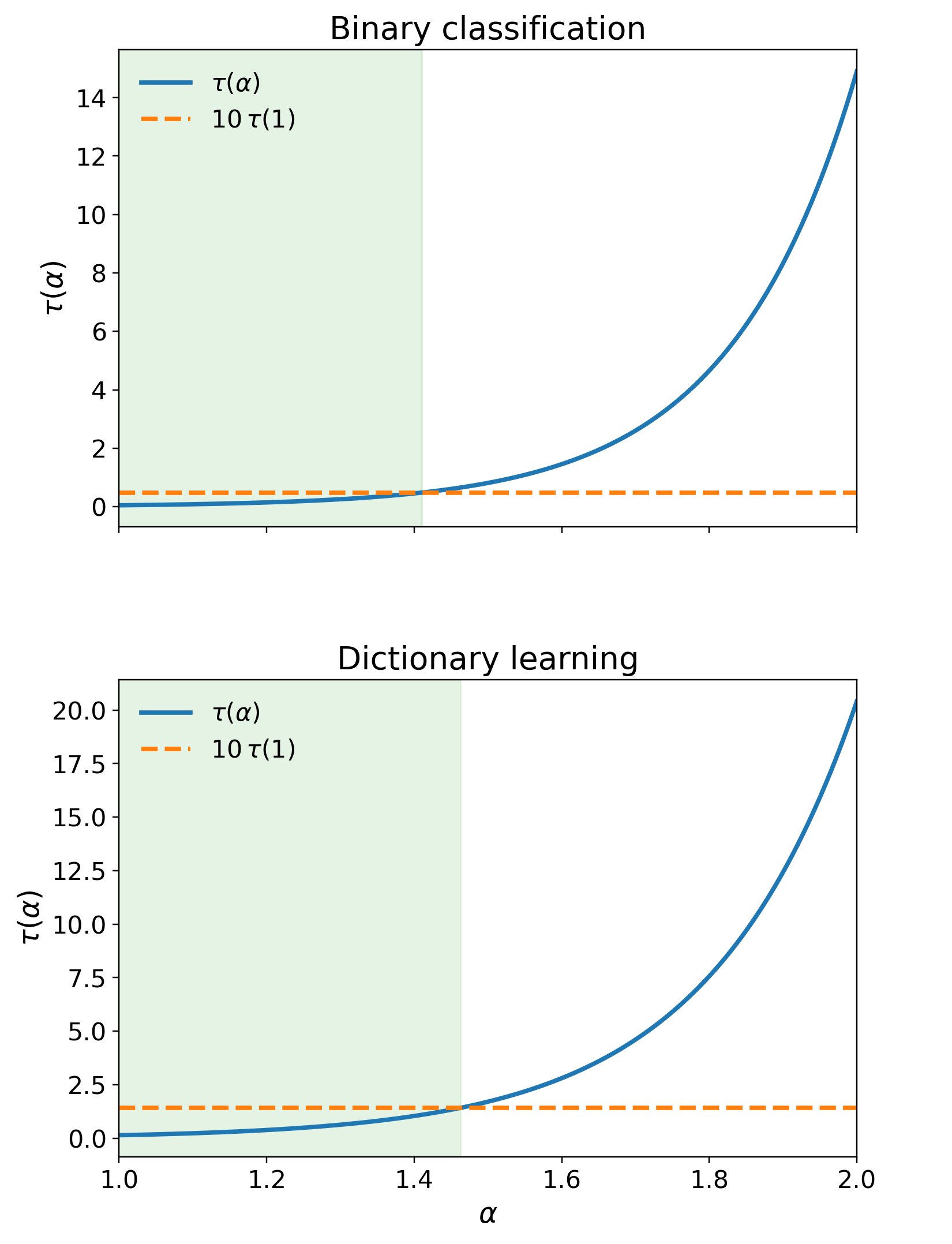}
    \caption{$\tau(\alpha)$ versus $\alpha\in[1,2]$ on the ball $\mathbb{B}_2(\hat{x};\,R=0.1)$. The green region highlights $\{\alpha:\tau(\alpha)\le 10\tau(1)\}$, illustrating a near-flat regime around $\alpha=1$ and the sharp growth as $\alpha\uparrow 2$.}
    \label{fig:placeholder}
\end{figure}
The complexity upper bound scales polynomially with \(\tau(\alpha)\) and depends on \(\varepsilon\) as \(\varepsilon^{-2/\alpha}\). Therefore, as \(\alpha\) increases, the term depending on \(\tau(\alpha)\) increases while the \(\varepsilon^{-2/\alpha}\) term decreases. Since \(\tau(\alpha)\) is finite over \(\alpha\in[1,2]\), for sufficiently small \(\varepsilon\), the optimal pair \((\tau(\alpha),\alpha)\) for the tightest upper bound lies near \(\alpha=2\). If \(\varepsilon\) is not sufficiently small, however, the optimal \(\alpha\) moves towards \(\alpha=1\). These curves therefore provide guidance on how, for a given \(\varepsilon\), to set \(\alpha\) in the learning rate of Algorithm~\ref{alg:SARAH} to achieve the tightest upper bound on the expected number of oracle queries. Further experimental details for these two tasks are provided in Appendix~\ref{append_experiment}, and additional empirical results are reported in Appendix~\ref{app:mf_experiments}.

\paragraph{Limitations and future work.}
Our guarantees are local around an isolated connected component of minimizers and assume initialization in a neighborhood where the local $\alpha$--P{\L} condition holds. The constants (e.g., the neighborhood radius and $\tau$) are problem-dependent, and achieving the optimal rate relies on schedules that depend on (or estimate) $\alpha$ and may use increasing mini-batches. Extending these results to more global guarantees and to batch-free variance-reduced updates are natural directions for future work.

\phantomsection\label{page:refs_start}
\bibliographystyle{abbrvnat}
\bibliography{references}

\section*{Checklist}

The checklist follows the references. For each question, choose your answer from the three possible options: Yes, No, Not Applicable.  You are encouraged to include a justification to your answer, either by referencing the appropriate section of your paper or providing a brief inline description (1-2 sentences). 
Please do not modify the questions.  Note that the Checklist section does not count towards the page limit. Not including the checklist in the first submission won't result in desk rejection, although in such case we will ask you to upload it during the author response period and include it in camera ready (if accepted).

 \begin{enumerate}

 \item For all models and algorithms presented, check if you include:
 \begin{enumerate}
   \item A clear description of the mathematical setting, assumptions, algorithm, and/or model. [\textbf{Yes}/No/Not Applicable]
   \item An analysis of the properties and complexity (time, space, sample size) of any algorithm. [\textbf{Yes}/No/Not Applicable]
   \item (Optional) Anonymized source code, with specification of all dependencies, including external libraries. [Yes/No/Not Applicable]
 \end{enumerate}

 \item For any theoretical claim, check if you include:
 \begin{enumerate}
   \item Statements of the full set of assumptions of all theoretical results. [\textbf{Yes}/No/Not Applicable]
   \item Complete proofs of all theoretical results. [\textbf{Yes}/No/Not Applicable]
   \item Clear explanations of any assumptions. [\textbf{Yes}/No/Not Applicable]     
 \end{enumerate}

 \item For all figures and tables that present empirical results, check if you include:
 \begin{enumerate}
   \item The code, data, and instructions needed to reproduce the main experimental results (either in the supplemental material or as a URL). [\textbf{Yes}/No/Not Applicable]
   \item All the training details (e.g., data splits, hyperparameters, how they were chosen). [\textbf{Yes}/No/Not Applicable]
         \item A clear definition of the specific measure or statistics and error bars (e.g., with respect to the random seed after running experiments multiple times). [\textbf{Yes}/No/Not Applicable]
         \item A description of the computing infrastructure used. (e.g., type of GPUs, internal cluster, or cloud provider). [\textbf{Yes}/No/Not Applicable]
 \end{enumerate}

 \item If you are using existing assets (e.g., code, data, models) or curating/releasing new assets, check if you include:
 \begin{enumerate}
   \item Citations of the creator If your work uses existing assets. [Yes/No/\textbf{Not Applicable}]
   \item The license information of the assets, if applicable. [Yes/No/\textbf{Not Applicable}]
   \item New assets either in the supplemental material or as a URL, if applicable. [Yes/No/\textbf{Not Applicable}]
   \item Information about consent from data providers/curators. [Yes/No/\textbf{Not Applicable}]
   \item Discussion of sensible content if applicable, e.g., personally identifiable information or offensive content. [Yes/No/\textbf{Not Applicable}]
 \end{enumerate}

 \item If you used crowdsourcing or conducted research with human subjects, check if you include:
 \begin{enumerate}
   \item The full text of instructions given to participants and screenshots. [Yes/No/\textbf{Not Applicable}]
   \item Descriptions of potential participant risks, with links to Institutional Review Board (IRB) approvals if applicable. [Yes/No/\textbf{Not Applicable}]
   \item The estimated hourly wage paid to participants and the total amount spent on participant compensation. [Yes/No/\textbf{Not Applicable}]
 \end{enumerate}

 \end{enumerate}
 
 \clearpage
 \newpage
 \onecolumn
\appendix
\aistatstitle{Optimal Local Convergence Rates of Stochastic First-Order Methods under Local $\alpha$-P\L:
Supplementary Materials}
\subsection*{Notations}
We adopt the following notation in the sequel. Calligraphic letters (e.g., $\mathcal{S}$) denote sets. Lowercase bold letters (e.g., $x$) denote vectors. $\|\cdot\|$ denotes the $\ell_{2}$-norm of a vector. We use $KL(\mu\|\nu):=\int \log\left(\frac{d\mu}{d\nu}(x)\right)\mu(dx)$ to denote the Kullback–Leibler (KL) divergence between two probability measures $\mu$ and $\nu$. The diameter of the subset $\mathcal{X}$ of $\mathbb{R}^{d}$ is defined by $\text{diam}(\mathcal{X}):=\sup_{x,y\in\mathcal{X}}\|x-y\|$. 
$\mathcal{C}$ denotes the class of continuous real functions. That is, $f$ is in differentiability class $\mathcal{C}^{k}$ if and only if there exists a $k$-th derivative of $f$ which is continuous.
Given functions $f, g: \mathcal{A} \to [0, \infty)$ where $\mathcal{A}$ could be any set, we use non-asymptotic big-O notation: $f = \mathcal{O}(g)$ if there exists a constant $c<\infty$ such that $f(a) \le c\cdot g(a)$ for all $a\in\mathcal{A}$ and $f = \Omega(g)$ if there is a constant $c > 0$ such that $f(a) \ge c \cdot g(a)$. We write $f = \tilde{\mathcal{O}}(g)$ as a shorthand for $f = \mathcal{O}(g\cdot\max\{1, (\log(g))^{k}\})$ for some integer $k>0$ and  $\tilde{\Omega}$ is similarly defined. The $d$-dimensional ball with radius $R$ around the center $x_{0}$ with respect to $\ell_{2}$-norm is denoted by $\mathbb{B}_{2}(x_{0};R):=\left\{x\in\mathbb{R}^{d}:\,\|x-x_{0}\|\le R\right\}$. For a subset $\mathcal{A}\subseteq \mathbb{R}^{d}$, its closure is denoted $\overline{\mathcal{A}}$
(and sometimes $\operatorname{cl}(\mathcal{A})$) and is the smallest closed set
containing $\mathcal{A}$. We use $\1(\cdot)$ for the indicator function: $\1(E)=1$ if event $E$ occurs and $0$ otherwise. We use $\sigma(\mathcal{S})$ to denote the sigma-field generated by a collection $\mathcal{S}$, i.e., the smallest sigma-field containing $\mathcal{C}$. For notational convenience, throughout all proofs in appendices, we write $L_{1}$ and $L_{2}$ for the
local regularity constants of $F$ on the compact neighborhood $\mathcal N(\mathcal M)$ around the connected component $\mathcal{M}$ of local minimizers of $F$:
for all $x,y\in\mathcal N(\mathcal M)$,
\begin{align*}
    &|F(x)-F(y)| \le L_{1}\,\|x-y\|\\
&\|\nabla F(x)-\nabla F(y)\| \le L_{2}\,\|x-y\|.
\end{align*}

\paragraph{Organization of the Appendix.} Appendix~\ref{related_work} reviews P{\L}-type conditions and oracle-complexity upper/lower bounds in both convex and nonconvex settings. 
Appendix~\ref{append_exist_univ_tau} motivates using a single (uniform) constant \(\tau\) in \Cref{assum_PL_alpha}. 
Appendix~\ref{append:gradient-dominated functions} explains why \(\alpha\in(1,2)\) captures flatness and why many ML losses satisfy the K{\L} inequality, hence admit a local \(\alpha\)–P{\L} property; Appendix~\ref{append_d} further shows that at nondegenerate minima one recovers the classical local \((\alpha=2)\)–P{\L}. 
Proofs of the main results from \Cref{sec:local_PL} are deferred to Appendix~\ref{append_sec:local PL}. The proof of \Cref{lower_bound_non-convex_PL_L-avg-smooth} appears in Appendix~\ref{proof_sec:3.2}, and the proof of \Cref{thm:cvg_SARAH} appears in Appendix~\ref{append_upper_bound_result}. 
The results in \Cref{sec_Lower bound for stochastic convex} are proved in Appendix~\ref{append:proofs_sec:4}. 
Appendix~\ref{append_comparison_relate_work} compares \Cref{lower_bound_PL_convex_NBS} with the lower bound of \citet{foster2019complexity}. 
Appendix~\ref{append_experiment} provides details for \Cref{sec:experiment}, and Appendix~\ref{app:mf_experiments} reports additional empirical evaluations.


\section{Related work}\label{related_work}
\textbf{P\L\ property and its applications:}
The $(\alpha=2)$-P\L\ property (commonly called PŁ condition) was initially introduced by \citep{polyak1963gradient}.
\cite{karimi2016linear} showed that the $(\alpha=2)$-P\L\ condition is less restrictive than several known global optimality conditions in the literature of machine learning~\citep{liu2014asynchronous,necoara2019linear,zhang2013linear}. 
 The $(\alpha=2)$-P\L\ property is satisfied (sometimes locally rather than globally, and also under distributional assumptions) for the population risk in some learning models including neural networks with one hidden layer \citep{li2017convergence}, ResNets with linear activation \citep{hardt2016identity}, generalized linear models and robust regression \citep{foster2018uniform}.\\
 The global P\L{} inequality is often too restrictive for modern neural networks: even when optimization behaves well in practice, verifying a uniform global PL constant typically relies on substantial---and often impractical---over-parameterization, as global convergence guarantees are proved in very wide-network regimes \citep{allenZhu2019convergence,du2019gd,zou2018sgd}. \cite{islamov2024loss} make this point explicit and construct settings where classical PL-type (and related ``aiming'') conditions fail, motivating weaker, geometry-aware alternatives; they propose landscape conditions that avoid heavy over-parameterization and tolerate saddle regions. This evidence justifies focusing on \emph{local} \(\alpha\)-PL behavior near basins of attraction---where training actually settles---rather than assuming a global PL landscape that rarely holds in realistic deep models.\\
 \citet{aich2025lplr} show that neighborhoods around initialization that are \emph{locally quasi-convex} with a \emph{stable neural tangent kernel (NTK)} in fact satisfy a \emph{local PL} inequality; they term such neighborhoods \emph{Locally Polyak--\L{}ojasiewicz Regions (LPLRs)} and prove that gradient descent attains \emph{linear} convergence once iterates enter an LPLR. Their analysis provides a concrete mechanism (NTK stability) by which local PL---the special case \(\alpha=2\) of our local \(\alpha\)-PL framework---emerges at finite number of hidden neurons per layer (width). Our results are complementary: while \citet{aich2025lplr} establish deterministic linear rates for GD under local PL, we derive \emph{optimal local convergence rates for stochastic first-order methods} under the more general local \(\alpha\)-PL condition (\(\alpha\in[1,2]\)), thereby covering noisy/mini-batch regimes.

\textbf{Complexity lower bounds:} In the convex setting, several complexity lower bounds have been derived by establishing a connection between stochastic optimization and hypothesis testing.
For instance, \cite{shapiro2005complexity} reduced a 
class of one-dimensional linear optimization problems to a binary hypothesis testing problem. 
Later on, this approach was used to derive the minimax oracle complexity of stochastic convex optimization in several works~\citep{agarwal2009information,raginsky2009information}. As an example, \cite{agarwal2009information} obtained a lower bound of $\Omega(\varepsilon^{-2})$ for the minimax oracle complexity of stochastic first-order methods in order to achieve an $\varepsilon$-global-optimum point of a bounded-domain Lipschitz convex function. This bound is derived through a reduction to a Bernoulli vector parameter estimation problem. For the same function class in \citep{agarwal2009information},  \cite{raginsky2009information} derived a complexity lower bound of $\Omega(\varepsilon^{-2})$ by a reduction to hypothesis testing with feedback, where the oracle provides noisy gradients by adding Gaussian noise to the true gradients\footnote{Note that \cite{agarwal2009information} considered noisy first-order oracles which do not allow additive noise due to a coin-tossing construction.}.
If the function is smooth (instead of Lipschitz) and convex, and the initial optimality gap is bounded (instead of the domain being bounded), a lower bound of $\Omega(\varepsilon^{-2})$ exists for the oracle complexity of stochastic first-order methods, according to Foster et al.'s complexity analysis \citep{foster2019complexity}. This bound is derived through a reduction to a noisy binary search problem. 

\noindent
In the non-convex setting, under $(\alpha=2)$-P\L\ and $L$-smoothness, \cite{yue2022lower} established a lower bound of $\Omega(L\tau \log(\varepsilon^{-1}))$ on the deterministic first-order methods to achieve an $\varepsilon$-global-optimum point\footnote{In this lower bound, the dependencies on $L$, $\tau$, and $\varepsilon$ are the same as the ones in gradient descent's iteration complexity.}. The main idea is based on a ``zero-chain" function\footnote{For a zero-chain function having a sufficiently high dimension, $d-T$ entries of update vector will never reach their optimal values after the execution of any first-order algorithm for a given $T$ number of iterations.} proposed as a hard instance, which is composed of the worst convex function designed by \citep{nesterov2003introductory} and a coordinate-wise function that makes the function non-convex.
More recently, \cite{yu2023optimal} obtained  lower bounds on the oracle complexity of zeroth-order methods for non-convex smooth and $\alpha$-P\L\ functions with an additive noise oracle. This lower bound is tight in terms of the dependence on $\varepsilon$ for dimensions less than six.

For our lower bound in the non-convex setting (Theorem \ref{lower_bound_non-convex_PL_L-avg-smooth}), akin to \cite{raginsky2009information} we use a reduction to hypothesis testing with an additive Gaussian noise oracle. We benefit from a set of mutual information bounds to establish a tight lower bound on the complexity of stochastic first-order optimization algorithms for smooth and gradient-dominated functions. What distinguishes Theorem \ref{lower_bound_non-convex_PL_L-avg-smooth} from \citep[Theorem 2]{raginsky2009information} is the construction of hard instances that satisfy local $\alpha$-P\L. These instances allow us to derive the optimal dependence on the precision $\varepsilon>0$ in the complexity lower bound.

In the convex setting, under local $(\alpha,\tau,\varepsilon)$-P\L\ property, we use a reduction to the noisy binary search problem in order to obtain a tight lower bound for first-order algorithms. 
In Appendix \ref{append_comparison_relate_work}, we discuss in more detail how our approach for deriving the lower bound in Theorem \ref{lower_bound_PL_convex_NBS} compares to \citep{foster2019complexity}. 

\textbf{Complexity upper bounds:}
In the non-convex unconstrained optimization setting, \cite{khaled2020better} showed that under $(\alpha=2)$-P\L\ condition, stochastic gradient descent (SGD) with time-varying step-size reaches an $\varepsilon$-global-optimum point with an oracle complexity of $\mathcal{O}(1/\varepsilon)$. Furthermore, it was shown that this dependency of the oracle complexity on $\varepsilon$ is optimal for SGD \citep{nguyen2019tight}. Recently, \cite{fontaine2021convergence} obtained an oracle complexity $\mathcal{O}(\varepsilon^{-4/\alpha+1})$ for SGD under smoothness and $\alpha$-P\L\ property  for $1\le \alpha\le 2$. \citet{fatkhullin2022sharp} establish an oracle complexity of $\mathcal{O}(\varepsilon^{-2/\alpha})$ for the variance-reduced method \textsc{PAGER} under a batch-smooth stochastic first-order oracle. Their analysis assumes that the entire trajectories of SGD and \textsc{PAGER} remain within the region where the objective satisfies the local-$\alpha$ P{\L} condition. In contrast, we \emph{show} that the trajectory of \Cref{alg:SARAH} stays inside the local neighborhood on which the function satisfies local $\alpha$-P{\L}. For convex functions, when $\alpha$-P\L\ holds on an $\varepsilon$-sub-level set of a global minimizer (see Assumption \ref{assum_local_PL_alpha}), stochastic first-order algorithms achieve an $\varepsilon$-global-optimum point with $\tilde{\mathcal{O}}(\varepsilon^{-2/\alpha})$ samples of stochastic gradients \citep{xu2017stochastic,yang2018rsg}. In Theorem \ref{lower_bound_PL_convex_NBS}, we show that the dependency of number of queries $\tilde{\mathcal{O}}(\varepsilon^{-2/\alpha})$ on $\varepsilon$ is tight.
\section{Universality  of the constant \texorpdfstring{$\tau$}{tau} in \Cref{assum_PL_alpha}}\label{append_exist_univ_tau}
In this part, we first show that $\alpha$-K\L\ property implies $\alpha$-P\L\ one as follows:
\begin{lemma}[K{\L} $\Rightarrow$ local $\alpha$--P{\L} near a compact isolated component]\label{lemm_KL_to_PL_corrected}
Let $F:\mathbb{R}^d\to\mathbb{R}$ be $\mathcal{C}^1$. Let $\mathcal{M}$ be a \emph{compact, isolated, connected} component of local minimizers of $F$, and set $l:=F(y)$ for any $y\in\mathcal{M}$. 
Assume that $F$ satisfies a local K{\L}\ inequality at each $x_\ast\in\mathcal{M}$: there exist $r_{x_\ast}>0$, $\alpha_{x_\ast}\in[1,2]$, and $\tau_{x_\ast}>0$ such that
\[
F(x)-l \;\le\; \tau_{x_\ast}\,\|\nabla F(x)\|^{\alpha_{x_\ast}}
\qquad\text{for all }x\in \mathbb{B}_2(x_\ast;r_{x_\ast})\text{ with }F(x)>l.
\]
Then there exist $r>0$, $\alpha\in[1,2]$, and $\tau>0$ such that, for every
\[
x \in \mathbb{B}_r(\mathcal{M}) \;:=\; \bigcup_{x^\ast\in \mathcal{M}} \mathbb{B}_2(x^\ast;r),
\]
we have $F(x)>l$ for $x\notin\mathcal{M}$ and the uniform local $\alpha$--P{\L} bound
\[
F(x)-l \;\le\; \tau\,\|\nabla F(x)\|^{\alpha}.
\]
\end{lemma}

\begin{proof}
By the assumption, for each $x^\ast\in\mathcal{M}$ there exist $r_{x^\ast}>0$, $\alpha_{x^\ast}\in[1,2]$, and $\tau_{x^\ast}>0$ with
\begin{equation}\label{eq:local_alpha_PL_each_point}
F(x)-l \;\le\; \tau_{x^\ast}\,\|\nabla F(x)\|^{\alpha_{x^\ast}}
\qquad \forall x\in\mathbb{B}_2(x^\ast;r_{x^\ast})\text{ with }F(x)>l.
\end{equation}
Since $\mathcal{M}$ is an isolated compact set of local minima, by shrinking $r_{x^\ast}$ if needed, we may assume that
\[
F(x)>l \quad \text{for all } x\in \mathbb{B}_2(x^\ast;r_{x^\ast})\setminus \mathcal{M},
\]
and also, by continuity of $F$ at $\mathcal{M}$, that
\[
F(x)-l \;\le\; 1 \quad \text{for all } x\in \mathbb{B}_2(x^\ast;r_{x^\ast}).
\]
The compactness of $\mathcal{M}$ yields that there is a \emph{finite} subcover $\mathcal{A}=\{y^\ast_1,\dots,y^\ast_m\}\subset\mathcal{M}$ such that
\[
\mathcal{U}\;:=\;\bigcup_{y^\ast\in\mathcal{A}} \mathbb{B}_2(y^\ast;r_{y^\ast}) \;\supset\; \mathcal{M}.
\]
Define
\[
\alpha \;:=\; \min_{y^\ast\in\mathcal{A}} \alpha_{y^\ast} \in [1,2], 
\qquad
\tau \;:=\; \max_{y^\ast\in\mathcal{A}} \tau_{y^\ast}^{\alpha/\alpha_{y^\ast}} \;>\;0.
\]
Fix any $x\in\mathcal{U}$ and choose $y^\ast\in\mathcal{A}$ with $x\in\mathbb{B}_2(y^\ast;r_{y^\ast})$. From \eqref{eq:local_alpha_PL_each_point} and $F(x)-l\le 1$ we get
\[
\|\nabla F(x)\| \;\ge\; \tau_{y^\ast}^{-1/\alpha_{y^\ast}}\,\big(F(x)-l\big)^{1/\alpha_{y^\ast}}
\;\ge\; \tau^{-1/\alpha}\,\big(F(x)-l\big)^{1/\alpha},
\]
where we used (i) $\alpha\le \alpha_{y^\ast}$ so $(F(x)-l)^{1/\alpha_{y^\ast}}\ge (F(x)-l)^{1/\alpha}$ (since $F(x)-l\le 1$), and (ii) $\tau\ge \tau_{y^\ast}^{\alpha/\alpha_{y^\ast}}$ so $\tau^{-1/\alpha}\le \tau_{y^\ast}^{-1/\alpha_{y^\ast}}$.
Rearranging the terms gives the desired uniform bound
\[
F(x)-l \;\le\; \tau\,\|\nabla F(x)\|^{\alpha}, \qquad \forall x\in\mathcal{U}.
\]
Finally, since $\mathcal{U}$ is open and contains $\mathcal{M}$, choose $r>0$ with $\mathbb{B}_r(\mathcal{M})\subset \mathcal{U}$. As $\mathcal{M}$ is an isolated component of local minimizers of $F$, $F(x)>l$ for $x \in \mathbb{B}_r(\mathcal{M})\setminus \mathcal{M}$. 
\end{proof}

\begin{lemma}[Uniformization of the local $\alpha$--P{\L} constant across components]
\label{lem:uniform_tau_over_components}
Let $F:\mathbb{R}^d\to\mathbb{R}$ be $\mathcal{C}^1$ and fix $\alpha\in[1,2]$. 
Let $\mathcal{M}_F$ denote the set of local minimizers of $F$, and suppose $\mathcal{M}_F$ decomposes into \emph{isolated} connected components. 
For each connected component $\mathcal{M}\subset \mathcal{M}_F$, set $l_{\mathcal{M}}:=F(y)$ for all $y\in\mathcal{M}$, and assume that there exist an open neighborhood $\mathcal{U}_{\mathcal{M}}$ of $\mathcal{M}$ and a constant $\tau_{\mathcal{M}}>0$ such that
\begin{align}\label{eq_0101012}
    F(x)-l_{\mathcal{M}} \;\le\; \tau_{\mathcal{M}}\,\|\nabla F(x)\|^{\alpha}
\qquad\text{and}\qquad
F(x)>l_{\mathcal{M}}\ \text{for all }x\in \mathcal{U}_{\mathcal{M}}\setminus \mathcal{M}.
\end{align}

Then the following hold.

\smallskip
\noindent(a)
If $\mathcal{M}_F$ has finitely many connected components $\{\mathcal{M}_1,\dots,\mathcal{M}_m\}$, then let
\[
\tau \;:=\; \max_{1\le j\le m}\ \tau_{\mathcal{M}_j},
\qquad
\mathcal{U}\;:=\;\bigcup_{j=1}^m \mathcal{U}_{\mathcal{M}_j},
\]
and for all $x\in\mathcal{U}$ let $\mathcal{M}$ be the (unique) component with $x\in\mathcal{U}_{\mathcal{M}}$. Then we have that
\[
F(x)-l_{\mathcal{M}} \;\le\; \tau\,\|\nabla F(x)\|^{\alpha}
\quad\text{and}\quad
F(x)>l_{\mathcal{M}}\ \text{if }x\notin\mathcal{M}.
\]

\noindent(b)
More generally, let $K\subset\mathbb{R}^d$ be compact. 
Then only finitely many components have neighborhoods intersecting $K$; denote this finite index set by
\[
\mathsf{I}(K)\;:=\;\{\mathcal{M}\ \text{component}:\ \mathcal{U}_{\mathcal{M}}\cap K\neq\emptyset\}.
\]
Setting
\[
\tau(K)\;:=\;\max_{\mathcal{M}\in \mathsf{I}(K)}\ \tau_{\mathcal{M}},
\qquad
\mathcal{U}(K)\;:=\;\bigcup_{\mathcal{M}\in \mathsf{I}(K)} \mathcal{U}_{\mathcal{M}},
\]
we have that, for all $x\in\mathcal{U}(K)$ let $\mathcal{M}$ be the (unique) component with $x\in\mathcal{U}_{\mathcal{M}}$,
\[
F(x)-l_{\mathcal{M}} \;\le\; \tau(K)\,\|\nabla F(x)\|^{\alpha}
\quad\text{and}\quad
F(x)>l_{\mathcal{M}}\ \text{if }x\notin\mathcal{M}.
\]
\end{lemma}

\begin{proof}
Since each component $\mathcal{M}$ is isolated, by shrinking the neighborhoods $\{\mathcal{U}_{\mathcal{M}}\}$ if necessary we may assume they are pairwise disjoint and hence that \eqref{eq_0101012} holds on each $\mathcal{U}_{\mathcal{M}}$.

\emph{(a) Finite-component case.}
Let $\tau:=\max_j \tau_{\mathcal{M}_j}$. 
Fix $x\in\mathcal{U}$ and let $\mathcal{M}$ be the unique component with $x\in\mathcal{U}_{\mathcal{M}}$. 
Then \eqref{eq_0101012} gives $F(x)-l_{\mathcal{M}} \le \tau_{\mathcal{M}}\|\nabla F(x)\|^\alpha \le \tau\|\nabla F(x)\|^\alpha$, and the strict inequality of $\mathcal{M}$ also follows from \eqref{eq_0101012}. 

\emph{(b) Compact-region case.}
Because the neighborhoods $\{\mathcal{U}_{\mathcal{M}}\}$ are pairwise disjoint and each contains its (closed) component $\mathcal{M}$, only finitely many of them can intersect a compact set $K$ (otherwise we would obtain an infinite family of pairwise disjoint open sets with points in $K$, contradicting compactness via a standard finite-subcover argument). 
Thus $\mathsf{I}(K)$ is finite, and the same max argument as in part (a) with $\tau(K):=\max_{\mathcal{M}\in\mathsf{I}(K)}\tau_{\mathcal{M}}$ yields the claim on $\mathcal{U}(K)=\bigcup_{\mathcal{M}\in\mathsf{I}(K)}\mathcal{U}_{\mathcal{M}}$.
\end{proof}


\section{\texorpdfstring{$\alpha$}{alpha}-P\L\ functions}\label{append:gradient-dominated functions}
In this section, we start in Section~\ref{append_why_alpha_PL} by motivating the local $\alpha$-P\L\ for the regime $\alpha\in(1,2)$. Then we introduce three closed-form examples that satisfy the $\alpha$-P\L\ property for $1<\alpha<2$. Next we provide in Section~\ref{app:network-revenue-management-example} an example of $(\alpha=1)$-P\L\ function in the network revenue management problem. Finally, we discuss the implications of K\L\ theory in machine learning in Section~\ref{app:kl-examples}.

\subsection{Local alpha P\L\ regime with $\alpha\in(1,2)$}\label{append_why_alpha_PL}
In this section, we explain why exponents \(\alpha\in(1,2)\) capture local flatness: 
if the first nonzero variation of \(F\) at a minimizer occurs at order \(p>2\) (i.e., all derivatives up to order \(p-1\) vanish), then \(F\) exhibits local \(\alpha\)–P{\L} growth with \(\alpha=\tfrac{p}{p-1}\in(1,2)\).
We also present explicit examples of functions that satisfy a local \(\alpha\)–P{\L} inequality for \(1<\alpha<2\).

\emph{Definition (order-$p$ contact at a minimizer).}
Let $F:\mathbb{R}^d\to\mathbb{R}$ be $\mathcal{C}^{p}$ near a local minimizer $x^\star$ with value $l:=F(x^\star)$. 
We say that $F$ has \emph{order-$p$ contact} at $x^\star$ (with $p>2$) if all derivatives up to order $p-1$ vanish at $x^\star$,
\[
\nabla F(x^\star)=0,\quad \nabla^2 F(x^\star)=0,\ \ldots,\ D^{p-1}F(x^\star)=0,
\]
and the first nonzero term in the Taylor expansion of $F$ is of the order $p$ with a positive coefficient in the sense that, for $h$ small,
\[
F(x^\star+h)=l+c_p\|h\|^p+o(\|h\|^p)\quad\text{for some }c_p>0.
\]
(For manifolds of minimizers $\mathcal{M}$, the same definition applies \emph{in the normal directions} to $\mathcal{M}$; tangential directions remain flat.)

\emph{Consequences.} 
For sufficiently small $\|h\|$,
\begin{align*}
    c_1\|h\|^{p} \le F(x^\star+h)-l \le c_2\|h\|^{p},\\\
c_3\|h\|^{p-1} \le \|\nabla F(x^\star+h)\| \le c_4\|h\|^{p-1}.
\end{align*}
Eliminating $\|h\|$ yields the local error bound
\begin{align*}
    F(x)-l \ \le\ \tau\,\|\nabla F(x)\|^{\alpha}\qquad\text{ with}\quad
\alpha=\frac{p}{p-1}\in(1,2),\ \ \tau=\frac{c_2}{c_3^{\,\alpha}}.
\end{align*}
Thus $\alpha\in(1,2)$ corresponds to \emph{degenerate curvature} of $F$ at local minimizers (the Hessian vanishes or is rank-deficient). 
For a minimizer manifold $\mathcal{M}$, writing $x=(u,v)$ with $u$ normal to $\mathcal{M}$ and $v$ tangential to $\mathcal{M}$, if $F(u,v)=\phi(u)$ with $\phi(u)=\Theta(\|u\|^p)$, the same relation i.e. $\alpha=p/(p-1)$ holds in a small tube around $\mathcal{M}$.

\paragraph{Explicit models (with $p>2$ and $\alpha=p/(p-1) \in (1,2)$).}
\begin{example}[Isolated minimizer]
    Consider $f(x) = c \cdot\|x\|^{q}$, where $q > 2$ and $c > 0$. $f(x)$ is $\alpha$-P\L\ with $\alpha=q/(q-1)$ and $\tau=q^{q/(1-q)}\cdot C^{1/(1-q)}$.
\end{example}
\begin{example}[Manifold of minimizers]
    $F(u,v)=\|u\|^{p}$ on $\mathbb{R}^r\times\mathbb{R}^{d-r}$ has minimizers $\mathcal{M}=\{(0,v) \;|\; v \in \mathbb{R}^{d-r}\}$; the same local $\alpha$–P{\L} bound holds in a tubular neighborhood of $\mathcal{M}$ (normal directions behave as the isolated case).
\end{example}
\begin{example}[Product loss]
    In this example, $F:\mathbb{R}^{d}\to\mathbb{R}$ is defined as $F(x)=(x_{1}x_{2}\cdots x_{d})^{2}$. It appears for instance as the squared loss of a one-dimensional $d$-layer linear neural network model on a 2-dimensional data point $(z,y)=(1,0)$ where $z=1$ is the data sample (feature) and $y=0$ the label so that the squared loss on $(x_{1}x_{2}\cdots x_{d}z-y)^{2}$. $F(x)$ satisfies the global $(\alpha=2d/(2d-1))$-P\L\ property because:
\begin{align}
    F(\mathbf{x})-F^{*}\le \frac{1}{(4d)^{\frac{d}{2d-1}}}\|\nabla F(\mathbf{x})\|^{\frac{2d}{2d-1}}.
\end{align}
Indeed,
\begin{align}
    \MoveEqLeft[4]\|\nabla F(\mathbf{x})\|^{2}=4(x_{1}x_{2}\cdots x_{d})^{4}\cdot \left[\frac{1}{x^{2}_{1}}+\ldots+\frac{1}{x^{2}_{d}}\right]\nonumber\\
    &\overset{(a)}{\ge} 4(x_{1}x_{2}\cdots x_{d})^{4}\cdot\frac{d}{\left(x_{1}\ldots x_{d}\right)^{2/d}}\nonumber\\
    &=4d (F(x)-F^{*})^{2-1/d},
\end{align}
where (a) follows because the harmonic mean is always upper bounded by the geometric mean (HM-GM inequality).
\end{example}

\paragraph{Why does $\alpha\in(1,2)$ PL appear in over-parametrized models.}
Empirically and structurally, The loss function landscape of trained network exhibits many \emph{flat} directions in the neighborhood of solutions. 
Flatness means that the Hessian is rank-deficient and that the first nonzero variation normal to the minimizer set can be of order $p>2$. 
As explained above, this yields the local growth law $F(x)-l\lesssim \|\nabla F(x)\|^{\alpha}$ with $\alpha=p/(p-1)\in(1,2)$. 
The following observations support this picture.\\
\emph{(i) Symmetries create flat directions.} 
Common architectures admit parameter symmetries (e.g., permutations of hidden units; positive-homogeneous layer rescaling for ReLU), producing connected families of equivalent solutions and rank-deficient Hessians at minima \citep{dinh2017sharp,neyshabur2015pathsgd}. \emph{(ii) Empirical curvature.}
Hessian spectra measured at trained solutions show many near-zero eigenvalues with a few outliers, indicating extended flat directions \citep{sagun2017empirical,ghorbani2019investigation}.
\emph{(iii) connectivity of solution set.}
Independently trained solutions are connected by simple low-loss curves, revealing wide, flat basins \citep{garipov2018modeconnectivity,draxler2018essentially}.

\subsection{\texorpdfstring{$(\alpha=1)$}{alpha}-P\L\ property in network revenue management problem}
\label{app:network-revenue-management-example}

Another class of examples where an \texorpdfstring{$(\alpha=1)$}{alpha}-P\L\ property appears is found in supply chain and revenue management. As shown in \cite{chen2022efficient}, problems in these domains can often be formulated as
\[
\min_{x\in\mathcal{X}}F(x):=\mathbb{E}[\phi(x\land \xi)],
\]
where $\mathcal{X}$ is a convex and compact subset of $\mathbb{R}^{d}$, $\xi \in \mathbb{R}^{d}$ is a non-negative random vector, the symbol $\land$ denotes component-wise minimum, and $\phi(\cdot)$ is a convex function. As a result, $F(\cdot)$ becomes non-convex. However, such a problem often admits a convex reformulation
\[
\min_{y\in\mathcal{Y}} G(y):=F(g^{-1}(y)),
\]
where $g(x) = x \land \xi$ and $g^{-1}(\uv):=(g_{1}^{-1}(u_{1}),\ldots,g_{d}^{-1}(u_{d}))$ with $g_{i}^{-1}(u_{i}):=\min_{x\in\mathcal{X}}\{x_{i}:g_{i}(x_{i})\ge u_{i}\}$. Therefore, the function $G(\cdot)$ is convex. Suppose $g : \mathcal{X} \to \mathcal{Y}$ is a bijective differentiable map with $\nabla g(x)\succeq \lambda I$, $\lambda \ge 0$ for all $x\in\mathcal{X}$. Then function $F(\cdot)$ satisfies $(\alpha=1)$-P\L\ property. Indeed,  for any $x\in\mathcal{X}$ with $g(x) = y$,
\begin{align*}
    \MoveEqLeft[4]F(x)-F^{*}=G(y)-G^{*}\\
    &\le \langle\nabla G(y),y-y^{*}\rangle\\
    &\le \|\nabla G(y)\|\cdot \|y-y^{*}\|\\
    &=\|\nabla g^{-1}(y)\nabla F(x)\|\cdot \|y-y^{*}\|\nonumber\\
    &\le \frac{D_{\mathcal{Y}}}{\lambda}\|\nabla F(x)\|,
\end{align*}
where $D_{\mathcal{Y}}$ is the diameter of the set $\mathcal{Y}$. Therefore, $F(\cdot)$ is a $(\alpha=1)$-P\L\ function.

\subsection{Implications of K\L/$\alpha$-P{\L}}\label{app:kl-examples}

\paragraph{Kurdyka--\L{}ojasiewicz (K\L{}) inequality.}
Let $F:\mathbb{R}^d\to(-\infty,+\infty]$ be proper and lower semicontinuous, and assume $F$ is $\mathcal{C}^1$
on a neighborhood of a compact set $\mathcal C\subset\mathbb{R}^d$. Suppose $\mathcal C\subset\mathrm{crit}\,F$ where $\mathrm{crit}\,F:=\{x\in\mathbb{R}^{d}:\,\nabla F(x)=0\}$
and that $F$ is constant on $\mathcal C$; denote this constant by $F_{\mathcal C}$.
We say that \emph{$F$ satisfies the K\L{} inequality on $\mathcal C$} if there exist $\eta>0$, a concave
function $\phi:[0,\eta)\to\mathbb{R}_+$ with $\phi(0)=0$, $\phi\in \mathcal{C}^1((0,\eta))$, and $\phi'(s)>0$ for
$s\in(0,\eta)$, and a neighborhood $\mathcal N$ of $\mathcal C$ such that, for all $x\in\mathcal N$ with
$F_{\mathcal C}<F(x)<F_{\mathcal C}+\eta$,
\[
\phi'\!\big(F(x)-F_{\mathcal C}\big)\,\|\nabla F(x)\|\ \ge\ 1.
\]
A common parametrization uses the \emph{K\L{} exponent} $\theta\in[0,1)$: there exist $c>0$ and
$\eta>0$ such that $\phi(s)=c\,s^{1-\theta}$ on $[0,\eta)$, in which case for all $x\in\mathcal{N}$ with $F_{\mathcal C}<F(x)<F_{\mathcal C}+\eta$,
\begin{align}\label{eq:KL-theta}
    \|\nabla F(x)\|\ \ge\ \kappa\,\big(F(x)-F_{\mathcal C}\big)^{\theta}\qquad\text{for some }\kappa>0.
\end{align}
\emph{(For nonsmooth $F$, replace $\|\nabla F(x)\|$ by $\mathrm{dist}(0,\partial F(x))$.)}
Real-analytic functions and, more broadly, definable/tame functions satisfy the K{\L} inequality around each compact set of critical points \citep{kurdyka1998gradients,bolte2007lojasiewicz,attouch2010proximal}.

\paragraph{Local $\alpha$-P{\L} condition.}
Given a target set of minimizers $\mathcal M$, the function $F$ obeys a \emph{local $\alpha$-P{\L}} inequality on a neighborhood $\mathcal U$ of $\mathcal M$ if there exist $\tau>0$ and $\alpha\in[1,2]$ such that
\begin{align}\label{eq:local-PL-alpha}
&F(x)-F^{*} \ \le\ \tau\,\|\nabla F(x)\|^{\alpha}\qquad \forall x\in\mathcal U,\nonumber\\
&\quad\text{where } F^{*}:=F(y)\ \text{for any } y\in\mathcal M.
\end{align}
When $\alpha=2$ we recover the classical P{\L} inequality. Matching equations \eqref{eq:KL-theta} and \eqref{eq:local-PL-alpha} with each other require to take $\alpha=1/\theta$: if $F$ has K{\L} exponent $\theta\in(0,1]$ near $\mathcal M$, then \eqref{eq:local-PL-alpha} holds with $\alpha=1/\theta$ (up to rescaling of constants); in particular, nondegenerate minima (a nondegenerate minimum is a local minimizer $x$ such that $\nabla^{2}F(x)\succ 0$) have $\theta=1/2$ and hence $\alpha=2$ (see more details in Appendix~\ref{append_d}).

\paragraph{o-minimal structures, definable sets, and definable functions.}
An \emph{o-minimal structure} $\mathcal{S}$ on the ordered field $(\mathbb{R},+,\cdot,\le)$ is a sequence $\{\mathcal{S}_n\}_{n\in\mathbb{N}}$ where each $\mathcal{S}_n$ is a collection of subsets of $\mathbb{R}^n$ (called the \emph{definable sets of arity $n$}) satisfying:
\begin{itemize}
  \item(Closure) Each $\mathcal{S}_n$ is closed under finite unions, finite intersections, and complements; if $A\in\mathcal{S}_n$ and $B\in\mathcal{S}_m$, then $A\times B\in\mathcal{S}_{n+m}$; and if $A\in\mathcal{S}_{n+m}$, then any linear image or $k$-dimensional coordinate projection of $A$ belongs to the appropriate $\mathcal{S}_k$.
  \item(Basic sets) All algebraic sets are definable: for every polynomial $p:\mathbb{R}^n\!\to\!\mathbb{R}$, the zero set $\{x\in\mathbb{R}^n:p(x)=0\}$ lies in $\mathcal{S}_n$.
  \item(o-minimality) Every set in $\mathcal{S}_1$ is a finite union of points and open intervals. In particular, there are no “wild” (e.g., fractal or infinitely oscillatory) definable subsets of $\mathbb{R}$.
\end{itemize}
A set $A\subset\mathbb{R}^n$ is \emph{definable (in $\mathcal{S}$)} if $A\in\mathcal{S}_n$.

A function $f:\mathbb{R}^d\to\mathbb{R}^m$ is \emph{definable (in $\mathcal{S}$)} if its \emph{graph}
\[
\operatorname{gph} f \;:=\; \{(x,y)\in\mathbb{R}^{d}\times\mathbb{R}^{m} : y=f(x)\}
\]
is a definable set; equivalently, $\operatorname{gph} f\in\mathcal{S}_{d+m}$. 
Thus, when we say ``$f$ is definable'', we are not placing $f$ itself inside a set collection; rather, we certify definability via its graph being one of the sets admitted by the structure.

\paragraph{Semialgebraic sets and functions (the structure $\mathbb{R}_{\mathrm{alg}}$).}
A set $S\subset\mathbb{R}^n$ is \emph{semialgebraic} if it can be constructed from finitely many polynomial equalities and inequalities using finitely many unions and intersections. 
Typical examples include intervals, halfspaces, Euclidean balls/ellipsoids, polynomial sublevel sets, and finite unions of these. 
The family of all semialgebraic sets $\{\mathcal{S}_n\}_{n\ge 1}$ forms an o\!-minimal structure, denoted $\mathbb{R}_{\mathrm{alg}}$.
A function $f:\mathbb{R}^d\to\mathbb{R}^m$ is \emph{semialgebraic} precisely when $\operatorname{gph} f$ is a semialgebraic subset of $\mathbb{R}^{d+m}$.



Two canonical o-minimal structures are
\begin{enumerate}
\item $\mathbb{R}_{\mathrm{alg}}$, whose definable sets/functions are exactly the semialgebraic ones,
\item $\mathbb{R}_{\mathrm{an},\exp}$, an o-minimal expansion of the real field obtained by adjoining (i) all \emph{restricted real-analytic} functions (i.e., real-analytic on a box and extended by $0$ outside it) and (ii) the exponential function. In particular, any set/function definable using these primitives (together with the field operations and order) is definable in $\mathbb{R}_{\mathrm{an},\exp}$.
\end{enumerate}
Many ML loss functions used in practice are definable: polynomial losses/constraints; piecewise-linear maps such as ReLU and max-pooling (built by finitely many linear pieces); compositions of these with affine layers; and standard regularizers like $\ell_1/\ell_2$. Definability is stable under the operations used to assemble models (addition, composition, taking products, projections), which is why these classes are called \emph{tame} \citep{dries1998tame,vdDriesMacintyreMarker1994,wilkie1996}.
In the rest of this section, we detail why some of the most popular ML loss functions are tame. 


\begin{itemize}\itemsep3pt
\item \textbf{Least squares / ridge / polynomial models.} Polynomial maps and norms are semialgebraic, and the semialgebraic class is closed under finite unions/intersections, products, addition/multiplication, and composition. Hence, empirical risks built from squares and polynomial penalties are semialgebraic (thus definable and tame) and, by the K{\L} theory for definable/subanalytic functions, satisfy a local K{\L} inequality \citep{bolte2007lojasiewicz,attouch2010proximal,dries1998tame}.
\item \textbf{Logistic / cross-entropy.}
We consider two standard cases.

\emph{(a) Binary logistic regression.}
Given data $(x_i,y_i)\in\mathbb{R}^d\times\{\pm1\}$, the empirical logistic risk
\[
L(w)=\frac{1}{n}\sum_{i=1}^n \log\!\big(1+\exp(-y_i\,x_i^\top w)\big) + \lambda\,R(w)
\]
is \emph{real-analytic} on $\mathbb{R}^d$ when $R$ is analytic (e.g., $R(w)=\tfrac12\|w\|_2^2$), because $t\mapsto\log(1+\exp t)$ is analytic on $\mathbb{R}$, and compositions/sums of analytic functions are analytic. 
By the classical {\L}ojasiewicz/K{\L} gradient inequality, every real-analytic function satisfies a local K{\L} inequality around each critical point \citep{bolte2007lojasiewicz,kurdyka1998gradients}. 
Hence $L$ enjoys a local K{\L} inequality, which is equivalent to a local $\alpha$–P{\L} inequality for some $\alpha\in(1,2]$ near its critical points.

\emph{(b) Multiclass softmax cross-entropy.}
With logits $z_i=W x_i\in\mathbb{R}^K$ and labels $y_i\in\{1,\dots,K\}$, the loss
\[
L(W)=\frac{1}{n}\sum_{i=1}^n \Bigg(\log\Big(\sum_{k=1}^K e^{z_{ik}}\Big) - z_{i,y_i}\Bigg) + \lambda\,R(W)
\]
is real-analytic. Here $z_{ik}$ is the $k$-th entry of vector $z_{i}$. The map $z\mapsto \log\!\sum_k e^{z_k}$ is analytic on $\mathbb{R}^K$, and $z_i=W x_i$ is linear in $W$. 
Therefore, $L$ satisfies a local K{\L} (hence local $\alpha$–P{\L}) inequality around its critical points \citep{bolte2007lojasiewicz,kurdyka1998gradients}.

\emph{(c) Nonsmooth regularizers.}
If $R$ is nonsmooth but \emph{definable} (e.g., $R(w)=\|w\|_1$, which is semialgebraic), then $L$ remains definable in the o-minimal structure $\R_{\mathrm{an},\exp}$ (which contains restricted analytic functions and exponential functions). 
Definable functions satisfy the K{\L} property in the (sub)differential sense \citep{bolte2007lojasiewicz,kurdyka1998gradients}, so the same conclusion holds with $\|\nabla L\|$ replaced by $\mathrm{dist}(0,\partial L)$.

\item \textbf{Deep nets with ReLU with $\ell_1/\ell_2$ regularization.}
Absolute value and finite maxima of polynomials are semialgebraic, hence so are ReLU and max.\\
Semialgebraic functions/sets are closed under finite sums, products, and \emph{composition} with affine maps. Therefore each layer map
\[
x \mapsto \text{ReLU}(Ax+b),
\]
is semialgebraic, and so is their finite composition. This also implies that the network is piecewise linear (polyhedral), see \citep{montufar2014number,hanin2019complexity}.\\
The penalties $\|w\|_1$ (piecewise linear) and $\tfrac12\|w\|_2^2$ (polynomial) are semialgebraic. For squared loss (and other polynomial losses), the empirical risk
\[
\frac{1}{n}\sum_i \tfrac12 \|f_\theta(x_i)-y_i\|^2 \;+\; \lambda_1\|w\|_1+\lambda_2\tfrac12\|w\|_2^2
\]
is semialgebraic because it is a finite sum of semialgebraic functions. More generally, if the data loss uses $\exp$/$\log$ (e.g., cross-entropy), the whole objective is \emph{definable} in the o-minimal structure $\R_{\mathrm{an},\exp}$ since semialgebraic functions (the network) and restricted analytic/$\exp$/$\log$ are definable and closed under composition. Hence these objectives admit a local K{\L} (therefore local $\alpha$–P{\L}) inequality.
\item \textbf{Matrix factorization.} Objectives of the form \(\tfrac12\|UV^\top - M\|_F^2\) (with polynomial regularizers) are polynomials in the entries of \(U\) and \(V\), and hence are semialgebraic.
Semialgebraic functions are definable and satisfy a local K{\L} inequality; therefore these problems admit a local \(\alpha\)–P{\L} bound near their critical points.
This is precisely the setting used in K{\L}-based convergence analyses (e.g., PALM) \citep{bolte2014proximal,attouch2010proximal}.
\end{itemize}

\section{Why is $\alpha$ equal to $2$ at nondegenerate minima.}\label{append_d}
Let $x_\star$ be a nondegenerate local minimum of $F$, i.e., such that $\nabla F(x_\star)=0$ and $\nabla^2 F(x_\star)\succ 0$.
By continuity of the Hessian, there exist $\mu>0$ and a neighborhood $U$ of $x_\star$ such that
$\nabla^2 F(x)\succeq \mu I$ for all $x\in U$; therefore $F$ is (locally) $\mu$-strongly convex on $U$.
Local strong convexity implies the P\L{} inequality with exponent $2$:
\begin{equation}\label{eq:localPL2}
  F(x)-F(x_\star)\ \le\ \frac{1}{2\mu}\,\|\nabla F(x)\|^2\qquad \text{for all }x\in U,
\end{equation}
which is exactly the $\alpha$-P\L{} condition with $\alpha=2$ \citep[see, e.g.,][]{karimi2016linear}.

\section{Proofs of \Cref{sec:local_PL}}\label{append_sec:local PL}
For notational convenience, throughout all proofs in appendices, we write $L_{1}$ and $L_{2}$ for the
local regularity constants of $F$ on the compact neighborhood $\mathcal N(\mathcal M)$:
for all $x,y\in\mathcal N(\mathcal M)$,
\begin{align}
    &|F(x)-F(y)| \le L_{1}\,\|x-y\|\label{eq_local_lip}\\
&\|\nabla F(x)-\nabla F(y)\| \le L_{2}\,\|x-y\|.\label{eq_local_smooth}
\end{align}
\subsection{Existence of tube $\mathbb{B}_{2}(\mathcal{M};R)$ inside $N(\mathcal M)$.}\label{append_tube}
Let $\mathcal M\subset\mathbb{R}^d$ be a (nonempty) compact set and let $\mathcal N(\mathcal M)$ be any open set containing $\mathcal M$.
In \Cref{lem:tube}, we show that there exists $R>0$ such that
\[
\mathbb B_2(\mathcal M;R)\ \subset\ \mathcal N(\mathcal M) .
\]
In particular, if $F$ satisfies the local $\alpha$-P{\L} inequality on the open neighborhood $\mathcal N(\mathcal M)$, then it also holds on the metric tube $\mathbb B_2(\mathcal M;R)$ for some $R>0$.

\begin{lemma}\label{lem:tube}
Let $\mathcal{S}\subset\mathbb{R}^d$ be compact and $\mathcal{U}\subset\mathbb{R}^d$ be an open set containing $\mathcal{S}$. Define the distance to the complement
\[
d^\ast\ :=\ \inf_{x\in \mathcal{S}}\ \mathrm{dist}\bigl(x,\mathbb{R}^d\setminus \mathcal{U}\bigr)
\ =\ \inf_{x\in \mathcal{S}}\ \inf_{y\notin \mathcal{U}}\ \|x-y\|.
\]
Then $d^\ast>0$ and, for any $0<R<d^\ast$, $\mathbb B_2(S;R)\subset \mathcal{U}$.
\end{lemma}

\begin{proof}
The function $x\mapsto \mathrm{dist}\bigl(x,\mathbb{R}^d\setminus \mathcal{U}\bigr)$ is continuous and strictly positive on $\mathcal{S}$, because $\mathcal{S}\subset \mathcal{U}$ and $\mathcal{U}$ is open. By the extreme value theorem, the compact set $\mathcal{S}$ attains a positive minimum $d^\ast>0$. If $0<R<d^\ast$ and $z\in \mathbb B_2(\mathcal{S};R)$, then $\mathrm{dist}\bigl(z,\mathbb{R}^d\setminus \mathcal{U}\bigr)\ge d^\ast-R>0$, and therefore $z\in U$.
\end{proof}

\begin{remark}
Suppose that for a connected component $\mathcal M\subseteq\mathcal M$ there exists an open neighborhood $\mathcal N(\mathcal M)$ on which $F$ satisfies the local $\alpha$-P{\L} inequality \eqref{eq_local_PL}, and that $\mathcal M$ is compact. Then there exists $R>0$ with
$\mathbb B_2(\mathcal M;R)\subseteq \mathcal N(\mathcal M)$, so Assumption~\ref{assum_PL_alpha} holds with this $R$.
\end{remark}


\section{Proofs of \Cref{sec:LB_under_local_PL}}\label{append_sec:3}
\subsection{Proof of Remark \ref{remark_alpha=1_lower_bound}}\label{proof_remark_alpha=1_lower_bound}
In this part, we show that the hard instance of function in \cite[Theorem 4]{foster2019complexity} satisfies \Cref{assum_PL_alpha} for $x \in \mathbb{B}_{2}^{d}(0;R)$. 
Moreover, the set of stationary points of this function coincides with its set of global minimizers. In addition, the stochastic gradients of their worst-case function can be produced by an oracle $O\in \mathsf{O}_{\sigma}^{\tilde{L}}$. Let $m$ be the number of iterations of a given stochastic first-order algorithm. \cite[Theorem 4]{foster2019complexity} uses the following hard instance of function:
\begin{align}
\tilde{F}(x)=\frac{\sigma}{m}\sum_{i=1}^{m} \langle x,z_{i}\rangle+\frac{b}{2}\|x\|^{2},
\end{align}
where $\{z_{1},\ldots,z_{m}\}$ are orthonormal vectors in $\mathbb{R}^{d}$ ($d\ge m$) and $b={2\sigma}/(R\sqrt{m})$. $\tilde{F}$ attains its minimum at $x^{*}=-{\sigma}/(bm)\sum_{i=1}^{m}z_{i}$ which has norm $\|x^{*}\|={\sigma}/(b\sqrt{m})={R}/{2}< R$. $\tilde{F}$ satisfies $L_{1}(R)$-Lipschitzness \eqref{eq_local_lip} and $L_{2}(R)$-smoothness \eqref{eq_local_smooth} with
\[
L_{1}(R)\le \sigma+b R,\quad L_{2}(R)\le b.
\]
The stochastic gradient is as follows:
\[
\tilde{\gv}(x,z)=\sigma z+bx,
\]
where $z$ is a random variable with the uniform distribution over $\{z_{1},\ldots,z_{m}\}$. Note that $\mathbb{E}[\tilde{\gv}(x,z)]=\nabla \tilde{F}(x)$, 
\begin{align}
    \mathbb{E}[\|\tilde{\gv}(x,z)-\nabla \tilde{F}(x)\|^{2}]&=\frac{1}{m}\sum_{i=1}^{m}\mathbb{E}[\|\sigma z_{i}-\frac{\sigma}{m}\sum_{j=1}^{m}z_{j}\|^{2}]\nonumber\\
&=\sigma^{2}\left(1-\frac{1}{m}\right)\le \sigma^{2},
\end{align}
and
\[
\mathbb{E}[\|\tilde{\gv}(x,z)-\tilde{\gv}(y,z)\|^{2}]=b^{2}\|x-y\|^{2}.
\]
Therefore, the stochastic gradient $\tilde{\gv}$ satisfies the properties of $\mathsf{O}_{\sigma}^{\tilde{L}}$. 
Since $\tilde{F}$ is convex, we have
\begin{align}
    \MoveEqLeft[4]\tilde{F}(x)-\tilde{F}^{*}\le \langle \nabla \tilde{F}(x),x^{*}-x \rangle\nonumber\\
&\le \sup_{y\in \mathbb{B}_{2}^{d}(0;R)}\|x^{*}-y\|\cdot \|\nabla \tilde{F}(x)\|\le 2R \|\nabla \tilde{F}(x)\|.
\end{align}
Thus $\tilde{F}$ satisfies \Cref{assum_PL_alpha} with $\tau\ge 2R$ and then $\tilde{F}\in \mathcal{F}^{\alpha}_{L_{1},L_{2}}$. Note that 
\begin{align}
    \|\nabla \tilde{F}(x)\|^{2}&=b^{2}\|x\|^{2}+\frac{\sigma^{2}}{m}+\frac{2b\sigma}{m}\sum_{i=1}^{m}\langle x,z_{i}\rangle\nonumber\\
&=2b (\tilde{F}(x)-\tilde{F}^{*}) 
\end{align}
where $\tilde{F}^{*}=-{\sigma^{2}}/(2bm)$. \cite[Theorem 4]{foster2019complexity} proves that $\mathbb{E}[\|\nabla \tilde{F}(\hat{x})\|^{2}]\ge {\sigma^{2}}/(8m)$ where $\hat{x}$ is the output of any randomized algorithm whose input is $S=\{z_{1},\ldots,z_{m/2-1}\}$. Then
\[
\mathbb{E} [\tilde{F}(\hat{x})]-\tilde{F}^{*}=\frac{1}{2b}\mathbb{E}[\|\nabla \tilde{F}(\hat{x})\|^{2}]\ge \frac{R\sqrt{m}}{2\sigma}\cdot\frac{\sigma^{2}}{8m}=\frac{\sigma^{2}R}{16\sqrt{m}}.
\]
Therefore, when $\alpha=1$, their lower bound of $\Omega(\varepsilon^{-2})$ holds in the same setting as Theorem~\ref{lower_bound_non-convex_PL_L-avg-smooth}.

\subsection{Proof of \Cref{lower_bound_non-convex_PL_L-avg-smooth}}\label{proof_sec:3.2}

\begin{proof}[Proof of Theorem \ref{lower_bound_non-convex_PL_L-avg-smooth}]
Let $\mathcal{F}^{\text{uni}}_{\alpha}$ be a subset of $\mathcal{F}_{\alpha}$ such that every $f\in \mathcal{F}^{\text{uni}}_{\alpha}$ has a unique local minimizer. For two functions $f_{0}$ and $f_{1}$ in $\mathcal{F}^{\text{uni}}_{\alpha}$, let us define $\delta(f_{0},f_{1}):=\|x_{f_{1}}^{*}-x_{f_{0}}^{*}\|$ where $x_{f_{i}}^{*}=\arg\min_{x\in \mathbb{R}^{d}}f_{i}(x)$ for $i\in\{0,1\}$. For a fixed algorithm $\mathsf{A}\in\mathcal{A}_{T}$, let $x_{T}$ be the output of the $T$-th iteration of $\mathsf{A}$ and $\hat{F}_{T}$ be a function in $\mathcal{F}^{\text{uni}}_{\alpha}$ whose minimizer is $x_{T}$. 

\noindent 
\citep[Proposition 2.2]{rebjock2024fast} implies then that for $F\in \mathcal{F}^{\text{uni}}_{\alpha}$, there exists a neighborhood of $x_{F}^{*}$,  $\tilde{\mathcal{N}}(x_{F}^{*})\subseteq\mathcal{N}(x_{F}^{*})$ such that $\lambda\cdot\|x-x_{F}^{*}\|^{\alpha/(\alpha-1)}\le F(x)-F(x_{F}^{*})$ for all $x\in \tilde{\mathcal{N}}(x_{F}^{*})$, where $\lambda=\left((\alpha-1)/{\alpha}\right)^{{\alpha}/(\alpha-1)}\tau^{-{1}/(\alpha-1)}$ and $x^{*}_{F}=\arg\min_{x\in\mathbb{R}^{d}}F(x)$.\\ 
We assume that for every algorithm $\mathsf{A}\in \mathcal{A}_{T}$, there are two subsets $\mathcal{U}_{0}^{F}$ and $\mathcal{U}^{F}$ of $\tilde{\mathcal{N}}(x^{*}_{F})$ such that if $x_{0}\in \mathcal{U}_{0}^{F}$ we have $\{x_{t}\}_{t=1}^{T}$ lie in $\mathcal{U}^{F}$ with high probability $1-\delta$, i.e., $\mathsf{A}\in\mathcal{A}_T(\mathcal{U}_0^{F}, \mathcal{U}^{F},\delta)$. We define the following event:
\[
\mathcal{E}_{m}(\mathcal{U}^{F}):=\left\{\{x_{t}\}_{t=1}^{m}\in \mathcal{U}^{F}\right\}.
\]
Conditioned on $\mathcal{E}_{m}(\mathcal{U}^{F})$, for $0<\rho<1/2$, we obtain
\begin{align}
\MoveEqLeft[4]\sup_{F\in\mathcal{F}^{\text{uni}}_{\alpha}}\inf_{\mathsf{A}\in\mathcal{A}_m(\mathcal{U}_0^{F}, \mathcal{U}^{F},\delta)}\mathbb{E}[F(x_{m})\mid\mathcal{E}_{m}(\mathcal{U}^{F})]-F(x_{F}^{*})\nonumber\\
&\ge \lambda\cdot \sup_{F\in\mathcal{F}^{\text{uni}}_{\alpha}}\inf_{\mathsf{A}\in\mathcal{A}_m(\mathcal{U}_0^{F}, \mathcal{U}^{F},\delta)}\mathbb{E}\left[\|x_{m}-x^{*}_{F}\|^{\frac{\alpha}{\alpha-1}}\mid \mathcal{E}_{m}(\mathcal{U}^{F})\right]\nonumber\\
&\overset{(a)}{\ge} \lambda\cdot \left(\sup_{F\in\mathcal{F}^{\text{uni}}_{\alpha}}\inf_{\mathsf{A}\in\mathcal{A}_m(\mathcal{U}_0^{F}, \mathcal{U}^{F},\delta)}\mathbb{E}\left[\|x_{m}-x^{*}_{F}\|\mid \mathcal{E}_{m}(\mathcal{U}^{F})\right]\right)^{\frac{\alpha}{\alpha-1}}\nonumber\\
  &\overset{(b)}{\ge} \lambda\cdot\left(\frac{\rho}{2}\cdot\sup_{F\in\mathcal{F}^{\text{uni}}_{\alpha}}\inf_{\mathsf{A}\in\mathcal{A}_m(\mathcal{U}_0^{F}, \mathcal{U}^{F},\delta)}\mathbb{P}\left[\delta(\hat{F}_{m},F)>\frac{\rho}{2}\mid \mathcal{E}_{m}(\mathcal{U}^{F})\right]\right)^{\frac{\alpha}{\alpha-1}},\label{B4}
\end{align}
where (a) comes from Jensen's inequality, and (b) from Markov's inequality and $\delta(\hat{F}_{m},F)=\|x_{m}-x^{*}_{F}\|$ as $\hat{F}_{m}$ is a function in $\mathcal{F}^{\text{uni}}_{\alpha}$ whose minimizer is $x_{m}$.


\noindent
In order to give a lower bound on \eqref{B4}, we use Fano's inequality given in the following lemma.
\begin{lemma}\label{2_hypo_fano_ineq}\cite[Theorem 2.5]{10.5555/1522486}
    Let $\mathcal{F}$ be a non-parametric class of functions, $\delta(\cdot,\cdot):\mathcal{F}\times \mathcal{F}\to \mathbb{R}$ be a semi-distance\footnote{$\delta(\cdot,\cdot)$ is a semi-distance if it satisfies the symmetry property and the triangle inequality but not the separation property (i.e., for every $f,g\in\mathcal{F}$, $\delta(f,g)=0\,\Leftrightarrow\, f=g$).}, and $\{P_{f}:f\in \mathcal{F}\}$ be a family of probability distribution indexed by $f\in \mathcal{F}$. Assume that there are $f_{0},f_{1}\in\mathcal{F}$ such that $\delta(f_{0},f_{1})\ge \rho>0$ and $KL(P_{f_{0}}\|P_{f_{1}})\le \gamma$ for some $\gamma>0$. Then,
    \begin{align}
\sup_{f\in\mathcal{F}}\inf_{\hat{f}}P_{f}\left(\left\{\delta(\hat{f},f)>\frac{\rho}{2}\right\}\right)\ge \max\left\{\frac{e^{-\gamma}}{4},\frac{1-\sqrt{\gamma/2}}{2}\right\},
    \end{align}
    where $\hat{f}$ is an estimator of $f$ from samples generated by $P_{f}$.
\end{lemma}
 In order to apply Lemma \ref{2_hypo_fano_ineq}, we need to specify $f_{0},f_{1}\in \mathcal{F}^{\text{uni}}_{\alpha}$ and corresponding $P_{f_{0}}, P_{f_{1}}$ such that $\delta(f_{0},f_{1})\ge \rho$ and $KL(P_{f_{0}}\|P_{f_{1}})\le \gamma$. 
 
\noindent
 \textbf{Construction of $f_{0},f_{1}$:}  We construct two continuously differentiable 1-dimensional functions $f_{0},f_{1}:\mathbb{R}\to \mathbb{R}$ as follows:
 \begin{align}
        &f_{0}(x)=\begin{cases}
         C|x|^{\frac{\alpha}{\alpha-1}}\quad &-R\le x\le R\\
         C\frac{\alpha}{\alpha-1}R^{\frac{1}{\alpha-1}}x+D\quad &R<x\\
         -C\frac{\alpha}{\alpha-1}R^{\frac{1}{\alpha-1}}x+D\quad&x<-R
     \end{cases},\label{B006}\\
     &f_{1}(x)=\begin{cases}
     2^{\frac{1}{\alpha-1}}C(|x-\rho|^{\frac{\alpha}{\alpha-1}}+|\rho|^{\frac{\alpha}{\alpha-1}})\quad &0\le x\le 2\rho\\
     f_{0}(x)\quad &2\rho\le x\\
     -\frac{\alpha}{\alpha-1}2^{\frac{1}{\alpha-1}}C\rho^{\frac{1}{\alpha-1}}x+2^{\frac{\alpha}{\alpha-1}}C\rho^{\frac{\alpha}{\alpha-1}}\quad &x\le 0, \label{B007}
     \end{cases}
 \end{align}
 where $C>0$ is a constant and $D=-(\alpha-1)^{-1}CR^{{\alpha}/(\alpha-1)}$.
 We choose $C$ as follows:
\begin{align}\label{choice_C}
    C=\tau^{-\frac{1}{\alpha-1}}\left(\frac{\alpha-1}{\alpha}\right)^{\frac{\alpha}{\alpha-1}}.
\end{align}
Based on \Cref{lem_pl_imply_tube_r}, for every function $F\in\mathcal{F}_{\alpha}^{\text{uni}}$, $\mathcal{U}_{0}^{F},\mathcal{U}^{F}\subseteq\mathbb{B}_{2}(x_{F}^{*};R)$. Let the neighborhoods of the minimizers of $f_{0}$ and $f_{1}$ be defined as 
$\tilde{\mathcal{N}}_{f_{0}}(0) = [-R, R]$ and $\tilde{\mathcal{N}}_{f_{1}}(\rho) = [0, R]$, respectively. Then $\mathcal{U}_{0}^{f_{i}},\mathcal{U}^{f_{i}}$ attributed to $f_{i}$, are subset of $\tilde{\mathcal{N}}_{f_{i}}(x_{f_{i}}^{*})$ for $i\in\{0,1\}$. We also pick $\rho$ such that $2\rho<R$.\\
In Lemma \ref{Lemm_PL_alpha_L_lip_f01} (refer to Appendix \ref{append_lower_bound_non-convex_PL_L-avg-smooth}),  we prove that $f_{0},f_{1}\in \mathcal{F}^{\text{uni}}_{\alpha}$ with the following constants for smoothness and Lipschitzness of both functions in their own neighborhoods:
\begin{equation}\label{B8}
\begin{aligned}
     &L_{2}(R)\ge \frac{(\alpha-1)^{\frac{2-\alpha}{\alpha-1}}}{(\alpha\tau)^{\frac{1}{\alpha-1}}}R^{\frac{2-\alpha}{\alpha-1}},\\
     &L_{1}(R)\ge \left(\frac{\alpha-1}{\alpha}\right)^{\frac{1}{\alpha-1}}\left(\frac{R}{\tau}\right)^{\frac{1}{\alpha-1}}.
\end{aligned}
\end{equation}
\noindent
 \textbf{Specification of the oracle:} We first specify the oracle $O^{*}$, needed to define $P_{f_{0}}$ and $P_{f_{1}}$, and which simply adds a standard normal noise to the gradient values. Let $f\in \mathcal{F}^{\text{uni}}_{\alpha}$. Then
 \begin{align}
     O^{*}(x)=(f'(x)+Z),
 \end{align}
 where $Z$ is a zero-mean normal noise with variance $\sigma^{2}$. Therefore, $O^{*}\in \mathsf{O}^{\tilde{L}}_{\sigma}$ as $f'(x,Z):=f'(x)+Z$ is unbiased, $\mathbb{E}[|f'(x,Z)-\mathbb{E}[f'(x,Z)]|^{2}]= \sigma^{2}$, and this oracle is $\tilde{L}$-average smooth with $\tilde{L}=L_{2}(R)$,
\begin{align*}
     \mathbb{E}[|f'(x,Z)-f'(y,Z)|^{2}]&=|f'(x)-f'(y)|^{2}\nonumber\\
 &\le L_{2}(R)^{2}|x-y|^{2}.
\end{align*}
\noindent 
\textbf{Specification of $P_{f_{0}}$ and $P_{f_{1}}$:} 
 For $i\in\{0,1\}$, $P^{m}_{f_{i}}$ denotes the distribution of $\{X_{t},f'_{i}(X_{t},Z_{t})\}_{t=1}^{m}$ where $X_{t}$ denotes the output of stochastic first-order algorithm $\mathsf{A}$ at iteration $t$.
 \begin{lemma}\label{claim_KL}
     Let $P^{m}_{f_{i}}$ be the distribution of $\{X_{t},f'_{i}(X_{t},Z_{t})\}_{t=1}^{m}$ for $i=\{0,1\}$ and $f_{0},f_{1}$ are defined in \eqref{B006} and \eqref{B007}, respectively. Then for $0<\rho\le 1/2$, we have
     \[
     KL(P^{m}_{f_{0}}\|P^{m}_{f_{1}})=\mathcal{O}\left(\frac{C^{2}m}{\sigma^{2}}\left(\frac{\alpha}{\alpha-1}\right)^{2}\rho^{\frac{2}{\alpha-1}}\right).
     \]
 \end{lemma}
The proof of Lemma~\ref{claim_KL} is given in Appendix~\ref{append_lower_bound_non-convex_PL_L-avg-smooth}. Lemma~\ref{claim_KL} shows that one can pick $\gamma=1/2$ if $\rho=\Theta\left(m^{-(\alpha-1)/{2}}\left({\sigma}/{C}\right)^{\alpha-1}\left((\alpha-1)/{\alpha}\right)^{\alpha-1}\right)$. We set therefore $\gamma$ and $\rho$ to these values in Lemma~\ref{claim_KL} so that $ KL(P^{m}_{f_{0}}\|P^{m}_{f_{1}})\le 1/2$. Hence, given $\delta(f_{0},f_{1})\ge \rho$, 
Lemma~\ref{2_hypo_fano_ineq} implies that
\begin{align}\label{B09}
\sup_{F\in\mathcal{F}^{\text{uni}}_{\alpha}}\inf_{\mathsf{A}\in\mathcal{A}_m(\mathcal{U}_0^{F}, \mathcal{U}^{F},\delta)}\mathbb{P}\left[\delta(\hat{F}_{m},F)>\frac{\rho}{2}\right]\ge \frac{1}{4}.
\end{align}

We return to \eqref{B4}, and finish the proof by plugging \eqref{B09} in \eqref{B4} to get
\begin{align}
\MoveEqLeft[4]\sup_{F\in\mathcal{F}^{\text{uni}}_{\alpha}}\inf_{\mathsf{A}\in\mathcal{A}_m(\mathcal{U}_0^{F}, \mathcal{U}^{F},\delta)}\mathbb{E}[F(x_{m})\mid\mathcal{E}_{m}(\mathcal{U}^{F})]-F(x_{F}^{*})\nonumber\\
  &\ge \lambda\left(\frac{\rho}{2}\cdot\sup_{F\in\mathcal{F}^{\text{uni}}_{\alpha}}\inf_{\mathsf{A}\in\mathcal{A}_m(\mathcal{U}_0^{F}, \mathcal{U}^{F},\delta)}\mathbb{P}\left[\delta(\hat{F}_{m},F)>\frac{\rho}{2}\right]\right)^{\frac{\alpha}{\alpha-1}}\nonumber\\
  &\overset{(c)}{\ge} \lambda\left[\Omega\left(\frac{1}{m^{\frac{\alpha-1}{2}}}\left(\frac{\sigma}{C}\right)^{\alpha-1}\left(\frac{\alpha-1}{\alpha}\right)^{\alpha-1}\right)\right]^{\frac{\alpha}{\alpha-1}}\nonumber\\
  &=\Omega\left(\frac{\lambda\sigma^{\alpha}\left(\frac{\alpha-1}{\alpha}\right)^{\alpha}}{C^{\alpha}m^{\frac{\alpha}{2}}}\right)\overset{(d)}{=}\Omega\left(\frac{\tau\sigma^{\alpha}}{m^{\frac{\alpha}{2}}}\right)\label{B11}
\end{align}
where (c) follows from \eqref{B09} and $\rho=\Theta\left(m^{-(\alpha-1)/{2}}\left({\sigma}/{C}\right)^{\alpha-1}\left((\alpha-1)/{\alpha}\right)^{\alpha-1}\right)$. Equation (d) results from the choices of $\lambda=\left((\alpha-1)/{\alpha}\right)^{{\alpha}/(\alpha-1)}\tau^{-{1}/(\alpha-1)}$ and $C$ in Equation~\eqref{choice_C}.
From \eqref{B11}, $\ms_{\varepsilon}(\mathcal{F}^{\mathcal{X}}_{\alpha,\tau,L},O^{*})=\Omega\left({\tau^{{2}/{\alpha}}\sigma^{2}}/{\varepsilon^{{2}/{\alpha}}}\right)$, which concludes the proof.
\end{proof}
\subsection{Proof of Lemma \ref{claim_KL}}\label{append_lower_bound_non-convex_PL_L-avg-smooth}
Let $X_{i}, Y_{i}$ denote the updating point and the gradient sample observed at iteration $i$ of the stochastic first-order algorithm $\mathsf{A}$, respectively.
Note that 
 \begin{align}\label{eq_B.0010}
 P_{f_{i}}^{m}(Y_{t}|X_{t}=x)=\mathbb{P}(f'_{i}(X_{t},Z_{t})|X_{t}=x)=\mathcal{N}(f'_{i}(x),\sigma^{2}).
 \end{align}
Let us define $X^{m}:=\{X_{i}\}_{i=1}^{m},Y^{m}:=\{Y_{i}\}_{i=1}^{m}$. Then, we have:
    \begin{align}
    KL(P_{f_{1}}^{m}\|P_{f_{0}}^{m})&=\mathbb{E}_{P_{f_{1}}^{m}}\left[\log \frac{P_{f_{1}}^{m}(X^{m},Y^{m})}{P_{f_{0}}^{m}(X^{m},Y^{m})}\right]\nonumber\\
        &=\mathbb{E}_{P_{f_{1}}^{m}}\left[\log \frac{\prod_{t=1}^{m}P_{f_{1}}^{m}(Y_{t}|X_{t})\cdot P(X_{t}|X^{t-1},Y^{t-1})}{\prod_{t=1}^{m}P_{f_{0}}^{m}(Y_{t}|X_{t})\cdot P(X_{t}|X^{t-1},Y^{t-1})}\right]\label{B.11}\\
        &=\mathbb{E}_{P_{f_{1}}^{m}}\left[\log \frac{\prod_{t=1}^{m}P_{f_{1}}^{m}(Y_{t}|X_{t})}{\prod_{t=1}^{m}P_{f_{0}}^{m}(Y_{t}|X_{t})}\right]\nonumber\\
        &=\sum_{t=1}^{m}\mathbb{E}_{X_{t}\sim P_{f_{1}}^{m}}\left[\mathbb{E}_{P_{f_{1}}^{m}}\left[\log \frac{P_{f_{1}}^{m}(Y_{t}|X_{t})}{P_{f_{0}}^{m}(Y_{t}|X_{t})}\Bigg|X_{t}\right]\right]\nonumber\\
        &\le m \cdot \max_{x\in\mathcal{N}_{f_{1}}(\rho)}\mathbb{E}_{P_{f_{1}}^{m}}\left[\log \frac{P_{f_{1}}^{m}(Y_{t}|X_{t})}{P_{f_{0}}^{m}(Y_{t}|X_{t})}\Bigg|X_{t}=x\right]\label{B.111}\\
        &=\frac{m}{2\sigma^{2}}\left(\max_{x\in\mathcal{N}_{f_{1}}(\rho)}|f'_{0}(x)-f'_{1}(x)|^{2}\right)\label{B.12}\\
        &=\frac{C^{2}m}{2\sigma^{2}}\left(\frac{\alpha}{\alpha-1}\right)^{2}\cdot\left[\max_{x\in[0,2\rho]}\left(2^{\frac{1}{\alpha-1}}|x-\rho|^{\frac{1}{\alpha-1}}\text{sgn}(x-\rho)-x^{\frac{1}{\alpha-1}}\right)^{2}\right]\label{B112}\\
        &=\mathcal{O}\left(\frac{C^{2}m}{\sigma^{2}}\left(\frac{\alpha}{\alpha-1}\right)^{2}\rho^{\frac{2}{\alpha-1}}\right),\label{B12}
    \end{align}
    where \eqref{B.11} comes from the fact that given $(X^{t-1},Y^{t-1})$, stochastic first-order algorithm's updated point $X_{t}$ is independent of the choice of the objective function. Equation \eqref{B.111} follows by $\text{supp}(P_{X_{t}})=\mathcal{U}_{0}^{f_{1}}\subseteq\mathcal{N}_{f_{1}}(\rho)$ when $X^{m}\sim P_{f_{1}}^{m}$, Equation \eqref{B.12} follows from \eqref{eq_B.0010}, and \eqref{B112} from the construction of $f_{0}$ (refer to \eqref{B006}) and of $f_{1}$ (refer to \eqref{B007}). In \eqref{B12}, we use the fact that $x=0$ achieves the maximum value in \eqref{B112}.
\begin{lemma}\label{Lemm_PL_alpha_L_lip_f01}
    Functions $f_{0}$ and $f_{1}$, defined in \eqref{B006} and \eqref{B007}, are elements of $\mathcal{F}^{\alpha,\text{uni}}_{L_{1},L_{2}}$ with $L_{2}\ge C{\alpha}{(\alpha-1)^{-2}}R^{(2-\alpha)/(\alpha-1)}$ and $ \tau\ge C^{1-\alpha}\left((\alpha-1)/{\alpha}\right)^{\alpha}$.
\end{lemma}
\begin{proof}[Proof of Lemma \ref{Lemm_PL_alpha_L_lip_f01}]
Recall
    \begin{align}
     &f_{0}(x)=C|x|^{\frac{\alpha}{\alpha-1}},\\
     &f_{1}(x)=\begin{cases}
     2^{\frac{1}{\alpha-1}}C(|x-\rho|^{\frac{\alpha}{\alpha-1}}+|\rho|^{\frac{\alpha}{\alpha-1}}) &0\le x\le 2\rho\\
     f_{0}(x) &2\rho\le x\\
     -\frac{\alpha}{\alpha-1}2^{\frac{1}{\alpha-1}}C\rho^{\frac{1}{\alpha-1}}x+2^{\frac{\alpha}{\alpha-1}}C\rho^{\frac{\alpha}{\alpha-1}} &x\le 0
     \end{cases}.
 \end{align}
 Note that each of $f_{0}$ and $f_{1}$ has a unique minimizer. Specifically, $x^{*}_{f_{0}}=\argmin_{x}f_{0}(x)=0$ and $x^{*}_{f_{1}}=\argmin_{x}f_{1}(x)=\rho$.

\noindent
\textbf{$L_{1}$-local smoothness of $f_{0}$ and $f_{1}$:} we have
\begin{align}
      |f'_{0}(x)|=
         C\frac{\alpha}{\alpha-1}|x|^{\frac{1}{\alpha-1}},
\end{align}
and
\begin{align}
  |f''_{1}(x)|=\begin{cases}
         2^{\frac{1}{\alpha-1}}C\frac{\alpha}{\alpha-1}|x-\rho|^{\frac{1}{\alpha-1}}\quad &0< x< 2\rho\\
         |f'_{0}(x)|\quad &2\rho\le x\\
         \frac{\alpha}{\alpha-1}2^{\frac{1}{\alpha-1}}C\rho^{\frac{1}{\alpha-1}}\quad &x\le 0
     \end{cases}.
\end{align}
Then $|f'_{0}(x)|\le \frac{\alpha C}{\alpha-1}R^{\frac{1}{\alpha-1}}$ for all $x\in \mathcal{N}_{f_{0}}(0)=[-R,R]$ and $|f'_{1}(x)|\le |f'_{0}(x)|\le L_{1}(R)=\frac{\alpha C}{\alpha-1}R^{\frac{1}{\alpha-1}}$ for all $x\in\mathcal{N}_{f_{1}}(\rho)=[0,R]$.

\textbf{$L_{2}$-local smoothness of $f_{0}$ and $f_{1}$:} 
\begin{align}
      |f''_{0}(x)|=
         C\frac{\alpha}{(\alpha-1)^{2}}|x|^{\frac{2-\alpha}{\alpha-1}}.
\end{align}
Let $L_{2}(R)= CR^{(2-\alpha)/(\alpha-1)}{\alpha}/{(\alpha-1)^{2}}$. 
Hence $f_{0}$ is $L_{2}$-locally smooth in $[-R,R]$.

\begin{align}
  |f''_{1}(x)|=\begin{cases}
         2^{\frac{1}{\alpha-1}}C\frac{\alpha}{(\alpha-1)^{2}}|x-\rho|^{\frac{2-\alpha}{\alpha-1}}\quad &0< x< 2\rho\\
         |f''_{0}(x)|\quad &2\rho< x\\
         0\quad &x\le 0
     \end{cases}.
\end{align}

Hence $f_{1}$ is smooth with constant $L_{2}(R)$. Then both $f_{0}$ and $f_{1}$ are $L_{2}(R)$-locally smooth in $[0,R]$.

\textbf{Local $\alpha$-P\L:} For $f_{0}$ with the neighborhood $[-R,R]$ around $x_{f_{0}}^{*}=0$, we have
\begin{align}
    \MoveEqLeft[4]f_{0}(x)-\min_{y\in [-R,R]}f_{0}(y)=C|x|^{\frac{\alpha}{\alpha-1}}\nonumber\\
    &\le\frac{(\alpha-1)^{\alpha}}{\alpha^{\alpha}C^{\alpha}} \cdot\left|C\frac{\alpha}{\alpha-1}|x|^{\frac{1}{\alpha-1}}\right|^{\alpha}\nonumber\\
    &=\frac{(\alpha-1)^{\alpha}}{\alpha^{\alpha}C^{\alpha}}\cdot |f'_{0}(x)|^{\alpha}.
\end{align}
For $f_{1}$ with the neighborhood $[0,R]$ around $x_{f_{1}}^{*}=\rho$, we have
\begin{align}
    \MoveEqLeft[4]f_{1}(x)-\min_{y\in [0,R]}f_{1}(y)=2^{\frac{1}{\alpha-1}}C|x-\rho|^{\frac{\alpha}{\alpha-1}}\nonumber\\
    &\le \frac{(\alpha-1)^{\alpha}}{C^{\alpha}\alpha^{\alpha}2^{\frac{\alpha}{\alpha-1}}}\cdot\left|
2^{\frac{1}{\alpha-1}}\frac{\alpha}{\alpha-1}\cdot C|x-\rho|^{\frac{1}{\alpha-1}}\right|^{\alpha}\nonumber\\
&=\frac{(\alpha-1)^{\alpha}}{C^{\alpha}\alpha^{\alpha}2^{\frac{\alpha}{\alpha-1}}}\cdot |f'_{1}(x)|^{\alpha}.
\end{align}
\end{proof}

\section{Proofs of \Cref{thm:cvg_SARAH}}\label{append_upper_bound_result}
Throughout the proof of \Cref{thm:cvg_SARAH}, we use the following parameterization of the time–varying step size and batch size as
\begin{equation}\label{eq_step_size_batch_size_SARAH}
    \eta_{k}=\frac{\eta_{0}}{(k+1)^{\eta}},\qquad
    n_{g}^{\,k}=(k+1)^{\beta},\quad k\ge 0,
\end{equation}
where $\eta>0$ and $\beta>0$ are some exponent constants and $\eta_{0}>0$ is the initial stepsize.
\subsection{Variance of the gradient estimator of SARAH}\label{append:vr_update}
Let us define $\gv_{\mathcal{J}_{t}}(x_{t}):=\frac{1}{|\mathcal{J}_{t}|}\sum_{j\in\mathcal{J}_{t}}\gv(x_{t},\xi_{j})$.
By adapting the variance reduction method in \citep{zhou2020stochastic} (see Algorithm \ref{algorithm2}), we have the following bound on the errors of the estimated gradient $\vv_t$:
\begin{lemma}\label{lemma:vr}
Let $n_g^t$ denote the number of stochastic gradient samples used at iteration $t$, and fix $\varepsilon>0$. Suppose that at every \emph{checkpoint} $t=kS$ with $k\in\mathbb{N}$,
\[
n_g^t \;\ge\; \frac{2\sigma^{2}}{\varepsilon},
\]
and at every \emph{non-checkpoint} iteration,
\begin{equation}\label{eq:bounds on samples of grad and hessian}
n_g^t \;\ge\; \frac{4 \tilde{L}^{2} S \,\|x_{t}-x_{t-1}\|^2}{\varepsilon}.
\end{equation}
Then, under the $\tilde{L}$-average smoothness assumption \eqref{eq-L-avg-smooth}, the variance-reduced gradient estimator $\vv_t$ satisfies
\begin{equation}\label{eq:bounds on grad and hessian}
\mathbb{E}\big[\|\nabla F(x_t)-\vv_t\|^{2}\big]\;\le\;\varepsilon
\end{equation}
for all iterations $t$.
\end{lemma}
\begin{proof}
First the gradient estimator $\vv_{t}$ is such that $\vv_{t}-\nabla F(x_{t})=\sum_{k=\lfloor t/S\rfloor S}^{t}\uv_{k}$ where
\begin{align*}
    &\uv_{k}=
\1(k=\lfloor t/S\rfloor S)\cdot[\gv_{\mathcal{J}_{k}}(x_{k})-\nabla F(x_{k})]\nonumber\\
&+\1(k>\lfloor t/S\rfloor S)\cdot[\gv_{\mathcal{J}_{k}}(x_{k})-\gv_{\mathcal{J}_{k}}(x_{k-1})-\nabla F(x_{k})+\nabla F(x_{k-1})].
\end{align*}
We know that $\mathbb{E}[\uv_{k}|\mathcal{F}_{k}]=0$ where $\mathcal{F}_{k}=\sigma(x_{0},\ldots,x_{k}, \uv_0,\cdots,\uv_{k-1})$. Conditioned on $\mathcal{F}_{k}$ and for $k>\lfloor t/S\rfloor S$, we have
\begin{align}
    &\mathbb{E}[\|\uv_{k}\|^{2}|\mathcal{F}_{k}]=\mathbb{E}\Big\|\frac{1}{n^{k}_{g}}\sum_{i=1}^{n^{k}_{g}}\gv(x_{k},\xi_{i})-\gv(x_{k-1},\xi_{i})-\nabla F(x_{k})+\nabla F(x_{k-1})\Big\|^{2}\nonumber\\
    &\overset{(a)}{=}\frac{1}{n^{k}_{g}}\mathbb{E}\Big\|\gv(x_{k},\xi_{1})-\gv(x_{k-1},\xi_{1})-\nabla F(x_{k})+\nabla F(x_{k-1})\Big\|^{2}\nonumber\\
    &\overset{(b)}{\le} \frac{2}{n^{k}_{g}}\mathbb{E}\|\gv(x_{k},\xi_{1})-\gv(x_{k-1},\xi_{1})\|^{2}+\frac{2}{n^{k}_{g}}\mathbb{E}\|\nabla F(x_{k})-\nabla F(x_{k-1})\|^{2}\nonumber\\
    &\overset{(c)}{\le} \frac{4\tilde{L}^{2}}{n^{k}_{g}}\|x_{k}-x_{k-1}\|^{2}
\end{align}
 where (a) follows because $[\gv(x_{k},\xi_{i})-\gv(x_{k-1},\xi_{i})-\nabla F(x_{k})+\nabla F(x_{k-1})]$ are i.i.d conditioned on $\mathcal{F}_k$ for $1\le i\le n^{k}_{g}$, (b) from $(a+b)^{2}\le 2a^{2}+2b^{2}$ and (c) from the $\tilde{L}$-average smooothness \Cref{eq-L-avg-smooth}. For $k=\lfloor t/S\rfloor S$, $\mathbb{E}[\|\uv_{k}\|^{2}]\le \frac{\sigma^{2}}{n^{k}_{g}}$. 
\begin{align}\label{eq_0000045}
    &\mathbb{E}[\|\nabla F(x_{t})-\vv_{t}\|^{2}]=\mathbb{E}\left\|\sum_{k=\lfloor t/S\rfloor S}^{t}\uv_{k}\right\|^{2}\le \sum_{k=\lfloor t/S\rfloor S}^{t}\mathbb{E}[\mathbb{E}[\|\uv_{k}\|^{2}|\mathcal{F}_{k}]]\nonumber\\
    &\le \frac{\sigma^{2}}{n^{\lfloor t/S\rfloor S}_{g}}+\sum_{k=\lfloor t/S\rfloor S+1}^{t}{4\tilde{L}^{2}}\mathbb{E}\left[\frac{\|x_{k}-x_{k-1}\|^{2}}{n^{k}_{g}}\right]
\end{align}
where the first inequality comes from the fact that $\mathbb{E}[\uv_{k}^{T}\uv_{k'}]=\mathbb{E}[\uv_{k'}^{T}\mathbb{E}[\uv_{k}|\mathcal{F}_{k}]]=0$ for $k> k'$.\\
If we take $ n_g^k\ge  \frac{2\sigma^{2}}{\varepsilon}$ samples at checkpoints ($k=\lfloor t/S\rfloor S$) and and at the other iterations, we take $n_g^k\ge \frac{4 \tilde{L}^{2} S \|x_{k}-x_{k-1}\|^2}{\varepsilon}$ samples, we have
\[
\mathbb{E}[\|\nabla F(x_{t})-\vv_{t}\|^{2}]\le \varepsilon.
\]
\end{proof}
Now by using \eqref{eq_0000045} in the proof of the previous lemma and the batch-size in \eqref{eq_step_size_batch_size_SARAH}, we have
\begin{align}\label{eq_010146}
    &\mathbb{E}[\|\vv_{t}-\nabla F(x_{t})\|^{2}]\le \frac{\sigma^{2}}{(\lfloor t/S\rfloor S+1)^{\beta}}+\frac{4\tilde{L}^{2}\cdot[t-\lfloor t/S\rfloor S]}{(\lfloor t/S\rfloor S+1)^{\beta}}\cdot\max_{\lfloor t/S\rfloor S+1\le k\le t}\mathbb{E}[\|x_{k}-x_{k-1}\|^{2}].
\end{align}
We later use this bound to show to prove high probability stability of updates of SARAH near a connected component of local minima.
\subsection{Proof of stability of SARAH}\label{append_proof_high_prob_stab}
Fix any $\delta\in(0,1)$ and let $\mathcal M$ be an isolated, compact, connected set of local minimizers with common value
$l:=F(x^\star)$ for all $x^\star\in\mathcal M$. Define the neighborhoods
\begin{align}\label{eq_tubular_neighborhood}
    &\mathcal U_0:=\Bigl\{x:\ \operatorname{dist}(x,\mathcal M)<\tfrac R2,\ \,F(x)-l\le \tfrac{s}{2}\Bigr\},\nonumber\\
&\mathcal U:=\Bigl\{x:\ \operatorname{dist}(x,\mathcal M)<\tfrac R2\Bigr\},
\end{align}
and the event $\mathcal{E}_{n}(\mathcal{U}):=\{x_k\in\mathcal U\ \text{for all }k\le n\}$ and abbreviate it as \(\mathcal E_n\). We claim that
$\mathbb P(\mathcal{E}_n)\ge 1-\delta$ whenever $x_0\in\mathcal U_0$. 

Assume that $x_k\in\mathcal U$ and choose the step-size (and batching policy) so that the one–step move is small with high probability; i.e. $\mathbb{P}(C_{n})\ge 1-\delta/2$ where
\[
C_n:=\Bigl\{\ \|x_{k+1}-x_k\|\le \tfrac R4\ \text{for all }k\le n\ \Bigr\}.
\]
By the triangle inequality, on $C_n$ we have $x_{k+1}\in\mathbb B_2(\mathcal M;\,3R/4)$. We define the “forbidden region”
\[
\mathcal R:=\mathbb B_2(\mathcal M;\,3R/4)\setminus \mathbb B_2(\mathcal M;\,R/2),
\]
and then as $\mathcal{M}$ is a compact isolated set, the next lemma shows that $f(x')-l>0$ for all $x\in\mathcal{R}$.

\begin{lemma}\label{lem_D_2}
Let \[
s_{0}:=\inf\{\,F(x)-l:\ x\in\mathcal R\,\}>0.
\]
Then we have that $s_{0}>0$.
\end{lemma}
\begin{proof}
Suppose, by contradiction, that $s_{0}=0$. Then there exists a sequence $\{x^{k}\}\subset \mathcal R$ with
$F(x^{k})\to l$. Because $\mathcal R\subset \mathbb B_2(\mathcal M;\,3R/4)$ and $\mathcal M$ is compact,
the sequence $\{x^{k}\}$ is bounded:
\[
\|x^{k}\|\le \tfrac{3R}{4}+\sup_{y\in \mathcal M}\|y\|<\infty.
\]
Hence there is a subsequence $x^{k_i}\to x$. By continuity, $F(x)=l$. Moreover,
$\operatorname{dist}(x^{k},\mathcal M)\ge R/2$ for all $k$, so $\operatorname{dist}(x,\mathcal M)\ge R/2$.
On the other hand, $x\in\overline{\mathbb B_2(\mathcal M;\,3R/4)}\subset \mathbb B_2(\mathcal M;\,R)$, and as $\mathcal M$ is an isolated set, we have $F(y)>l$ for all
$y\in \mathbb B_2(\mathcal M;\,R)\setminus \mathcal M$. Thus $F(x)=l$ forces $x\in\mathcal M$, which
contradicts $\operatorname{dist}(x,\mathcal M)\ge R/2$. Therefore $s_{0}>0$.
\end{proof}

Next, we invoke the following lemma and verify its assumptions later to guarantee that the SARAH iterates remain in $\mathcal U$.
\begin{lemma}\label{lem_D_3}
Suppose $x,y\in\mathbb{R}^{d}$ such that (i) $\mathrm{dist}(x,\mathcal M)<\tfrac R2$, (ii) $\|x-y\|\le \tfrac R4$, and (iii) $F(y)-l<s_{0}$. Then $\mathrm{dist}(y, \mathcal{M})<\tfrac R2$.
\end{lemma}
\begin{proof}
Pick $x^\ast\in \mathcal{M}$ with $\|x-x^\ast\|<\frac R2$. By the triangle inequality,
\[
\|y-x^\ast\|\le \|y-x\|+\|x-x^\ast\|\le \tfrac R4+\tfrac R2=\tfrac{3R}{4},
\]
so $\operatorname{dist}(y,\mathcal{M})\le \tfrac{3R}{4}$. Suppose by a contradiction that $\operatorname{dist}(y,\mathcal{M})\ge \tfrac R2$, then
\[
y\in \mathcal{R}=\mathbb B_2(\mathcal M;\,3R/4)\setminus \mathbb B_2(\mathcal M;\,R/2).
\]
By the definition of $s_{0}$, this implies $f(y)-l\ge s_{0}$, contradicting assumption (ii). Hence $\operatorname{dist}(y,\mathcal{M})<\tfrac R2$, as claimed. 
\end{proof}
The plan is to exclude $x_{k+1}\in\mathcal R$ by showing that, with high probability, $F(x_{k+1})-l<s_{0}$ once $x_0\in\mathcal U_0$.\\
Let $D_n:=F(x_n)-l$ and define
\[
\mathsf A_k:=-\langle \nabla F(x_k),\,\vv_{k}-\nabla F(x_k)\rangle,
\qquad
\mathsf B_k:=\|\vv_{k}\|^2,
\]
where $\vv_{k}$ is the SARAH gradient estimator in \Cref{alg:SARAH}.

\begin{lemma}\label{lem_D_5}
Under $L_{2}$-local smoothness and local $\alpha$-P\L{} and for $\eta_{0}\le \tau/2$, on $\mathcal{E}_n$, we have
\begin{align*}
    \textbf{(i) $\alpha=2$:}\quad
D_{n+1}\mathbf 1_{\mathcal{E}_n}\ \le\
D_0\cdot\prod_{k=0}^{n}\!(1-\tau^{-1}\eta_k)+
\sum_{k=0}^{n}\!\Big(\prod_{j=k+1}^{n}(1-\tau^{-1}\eta_j)\Big)\,\eta_k\,\mathsf A_k\,\mathbf 1_{\mathcal{E}_k}
+\frac{L_{2}}{2}\sum_{k=0}^{n}\eta_k^2\,\mathsf B_k\,\mathbf 1_{\mathcal{E}_k},
\end{align*}
\begin{align*}
    \textbf{(ii) $\alpha\in[1,2)$:}\quad
D_{n+1}\mathbf 1_{\mathcal{E}_n}\ &\le\
D_0\cdot\prod_{k=0}^{n}\!(1-\tau^{-2/\alpha}\eta_k^{q})+\tilde{c}\sum_{k=0}^{n}\eta_k^{\frac{2q-2}{2-\alpha}}+\sum_{k=0}^{n}\!\Big(\prod_{j=k+1}^{n}(1-\tau^{-2/\alpha}\eta_j^{q})\Big)\,\eta_k\,\mathsf A_k\,\mathbf 1_{\mathcal{E}_k}\nonumber\\
&\qquad\qquad\qquad\qquad\qquad\qquad\qquad\qquad\qquad\qquad\qquad\qquad+\frac{L_{2}}{2}\sum_{k=0}^{n}\eta_k^2\,\mathsf B_k\,\mathbf 1_{\mathcal{E}_k},
\end{align*}   
for any $1\le q$. 
\end{lemma}
\begin{proof}
Since $F$ is $L_{2}$-smooth and the update rule is $x_{n+1}=x_n-\eta_n \vv_{n}$, we have
\begin{align}\label{eq_040}
&D_{n+1}
\le D_n - \eta_n\langle \nabla F(x_n), \vv_{n}\rangle + \frac{L_{2}\eta_{n}^2}{2}\|\vv_{n}\|^2\nonumber\\
&= D_n - \eta_{n}\|\nabla F(x_n)\|^2 - \eta_{n}\langle \nabla F(x_n), \vv_{n}-\nabla F(x_n)\rangle+ \frac{L_{2}\eta_{n}^2}{2}\|\vv_{n}\|^2\nonumber\\
&= D_n - \eta_{n}\|\nabla F(x_n)\|^2 + \eta_{n}\mathsf A_n + \frac{L_{2}\eta_{n}^2}{2}\mathsf B_k .
\end{align}
Multiplying both hand-sides of \eqref{eq_040} by $\mathbf{1}_{\mathcal{E}_n}$ and noting that $\mathbf{1}_{\mathcal{E}_{n+1}}\le \mathbf{1}_{\mathcal{E}_n}$, we obtain the recursive bounds below. We distinguish the two cases.

\smallskip\noindent
\emph{Case $\alpha=2$.}
On $\mathcal{E}_n$ the unified gradient domination gives $\|\nabla F(x_n)\|^2\ge \tau^{-1}\,(F(x_n)-l)=\tau^{-1} D_n$, hence
\begin{align}
    D_{n+1}\mathbf{1}_{\mathcal{E}_n}
&\;\le\; (1-\eta_{n} \tau^{-1})\,D_n\mathbf{1}_{\mathcal{E}_n} \;+\; \eta_{n}\mathsf A_n\mathbf{1}_{\mathcal{E}_n}\;+\;\frac{L\eta_{n}^2}{2}\,\mathsf B_n\mathbf{1}_{\mathcal{E}_n}.
\end{align}
Iterating this affine recursion and using $\mathbf{1}_{\mathcal{E}_{n}}\le \mathbf{1}_{\mathcal{E}_k}$ for $k\le n$ gives the unfolded product–sum representation stated in the lemma.

\smallskip\noindent
\emph{Case $\alpha\in[1,2)$.}
On $\mathcal{E}_n$, the local PŁ property yields that $\|\nabla F(x_n)\|\ge \tau^{-\frac{1}{\alpha}}\,D_n^{\frac{1}{\alpha}}$ and therefore \eqref{eq_040} becomes 
\begin{align}\label{eq_041}
\MoveEqLeft[4]D_{n+1}
\le D_n - \eta_{n}\|\nabla F(x_n)\|^2 + \eta_{n}\mathsf A_n + \frac{L\eta_{n}^2}{2}\mathsf B_k\nonumber\\
&\le D_n - \eta_{n}\tau^{-\frac{2}{\alpha}}D_{n}^{\frac{2}{\alpha}} + \eta_{n}\mathsf A_n + \frac{L\eta_{n}^2}{2}\mathsf B_k\nonumber\\
&= (1-\eta_{n}^{q}\tau^{-\frac{2}{\alpha}})D_{n}+\eta_{n}\tau^{-\frac{2}{\alpha}}(\eta_{n}^{q-1}D_{n}-D_{n}^{\frac{2}{\alpha}})+ \eta_{n}\mathsf A_n + \frac{L\eta_{n}^2}{2}\mathsf B_k\nonumber\\
&\overset{(a)}{\le} (1-\eta_{n}^{q}\tau^{-\frac{2}{\alpha}})D_{n}+\tilde{c}\eta_{n}^{\frac{2q-2}{2-\alpha}}+ \eta_{n}\mathsf A_n + \frac{L\eta_{n}^2}{2}\mathsf B_k.
\end{align}
In step (a), to linearize $D_n^{\,2/\alpha}$ and bound the term
$\eta_{n}^{\,q}\tau^{-2/\alpha} D_n - \eta_{n}\tau^{-2/\alpha} D_{n}^{\,2/\alpha}$,
we use
\[
\max_{x\in\R}\{\,c\,x - x^{2/\alpha}\,\}
= \Big(\tfrac{\alpha}{2}\Big)^{\!\frac{\alpha}{2-\alpha}}\,
\frac{2-\alpha}{\alpha}\; c^{\frac{2}{2-\alpha}},
\]
applied to $x=D_n$ and $c=\eta_n^{\,q-1}$. Hence, for
\[
\tilde c \;:=\; \Big(\tfrac{\alpha}{2}\Big)^{\!\frac{\alpha}{2-\alpha}}
\Big(1-\tfrac{\alpha}{2}\Big)\,\tau^{-2/\alpha},
\]
we obtain
\[
\eta_{n}^{\,q}\tau^{-2/\alpha} D_n - \eta_{n}\tau^{-2/\alpha} D_{n}^{\,2/\alpha}
\;\le\; \tilde c\, \eta_{n}^{\frac{2q-2}{\,2-\alpha\,}}.
\]
Multiplying \eqref{eq_041} by $\mathbf{1}_{\mathcal{E}_n}$ and iterating this affine recursion, and next using $\mathbf{1}_{\mathcal{E}_{n}}\le \mathbf{1}_{\mathcal{E}_k}$ for $k\le n$, we obtain the unfolded product–sum representation stated in the lemma.    
\end{proof}
We denote the aggregated noises with the natural weights:
\begin{align}\label{eq_def_tilde_A_k}
    \tilde{\mathsf A}_k:=\sum_{j=0}^{k}\Big(\prod_{i=j+1}^{k}\theta_i\Big)\eta_j\,\mathsf A_j\,\mathbf 1_{\mathcal{E}_j},\;
\tilde{\mathsf B}_k:=\frac{L}{2}\sum_{j=0}^{k}\eta_j^2\,\mathsf B_j\,\mathbf 1_{\mathcal{E}_j},
\end{align}
where $\theta_i=1-\tau^{-1}\eta_i$ for $\alpha=2$ and $\theta_i=1-\tau^{-2/\alpha}\eta_i^{q}$ for $\alpha\in[1,2)$. Let us define for $n\ge1$, $\mathsf{R}_n:=\tilde{\mathsf A}_n^{\,2}+\tilde{\mathsf B}_n$ and the ``small-error'' event
\[
E_n:=\big\{\ \mathsf{R}_k<s\ \text{for all }k\le n\ \big\}.
\]
\begin{lemma}\label{lem_d7}
Let us define
\[
 \widehat E_n:=E_n\cap C_n.
\]
Assume $x_0\in \mathcal U_0$ almost surely. Pick $s$ such that $2s+\sqrt{s}<s_{0}$. If $\alpha=2$, let $\eta_0>0$ be an arbitrary constant; if $\alpha\in[1,2)$ choose $2-\alpha/2\le q$ and $\eta_{0}$ small enough so that
\begin{align}
    \sum_{k=1}^\infty\eta_k^{\frac{2 q-2}{2-\alpha}}=\eta_{0}^{\frac{2q-2}{2-\alpha}}\sum_{n=1}^{\infty}\frac{1}{(n+1)^{\frac{2q-2}{2-\alpha}}}<\frac{s}{2\tilde c},
\end{align}
with $\tilde c=(2/\alpha)^{-\frac{\alpha}{2-\alpha}}\!\bigl(1-\tfrac{\alpha}{2}\bigr)\tau^{-2/\alpha}$.
Then the following holds for all $n\in\mathbb{N}$:
\[
\widehat E_n\subset \mathcal{E}_{n+1}.
\]
\end{lemma}
\begin{proof}
By definition, $\{\mathsf{R}_{k}<s\}_{k\le n+1}$ implies $\{\mathsf{R}_{k}<s\}_{k\le n}$, so $E_{n+1}\subset E_n$.
Similarly, $\{x_{k}\in\mathcal U\}_{k\le n+1}$ implies $\{x_{k}\in\mathcal U\}_{k\le n}$, hence $\mathcal{E}_{n+1}\subset\mathcal{E}_n$.
Since $C_{n+1}\subset C_n$, we also have $\widehat E_{n+1}=E_{n+1}\cap C_{n+1}\subset E_n\cap C_n=\widehat E_n$.

\smallskip
We prove $\widehat E_n\subset \mathcal{E}_{n+1}$ by induction. Because $x_{0}\in\mathcal U_0\subset\mathcal U$ a.s., the claim holds for $n=0$.
Assume $\widehat E_{n-1}\subset \mathcal{E}_n$ and pick $\omega\in \widehat E_n$. Then $\omega\in\mathcal{E}_n$ (induction hypothesis) and $\omega\in C_n$ (by definition). Applying \Cref{lem_D_3} with $x=x_n(\omega)$ and $y=x_{n+1}(\omega)$: on $\mathcal{E}_n(\omega)$,
\[
\inf_{x^\ast\in \mathcal{M}}\|x_n(\omega)-x^\ast\|<\tfrac R2
\quad\text{and}\quad
\|x_{n+1}(\omega)-x_n(\omega)\|\le\tfrac R4,
\]
so conditions (i) and (iii) of \Cref{lem_D_3} are satisfied. It remains to ensure that $f(x_{n+1}(\omega))-l<s$, where $s>0$ is as in \Cref{lem_D_2}. We distinguish two cases $\alpha=2$ and $\alpha\in[1,2)$ in the rest of the proof:

     \emph{Case $\alpha=2$.} From the recursion of \Cref{lem_D_5} (valid on $\mathcal{E}_n$),
\begin{align}\label{eq_000059}
    &D_{n+1}\mathbf 1_{\mathcal{E}_n}\ \le\
D_0\prod_{k=0}^{n}\!(1-\tau^{-1}\eta_k)+
\sum_{k=0}^{n}\!\Big(\prod_{j=k+1}^{n}(1-\tau^{-1}\eta_j)\Big)\,\eta_k\,\mathsf A_k\,\mathbf 1_{\mathcal{E}_k}
+\frac{L}{2}\sum_{k=0}^{n}\eta_k^2\,\mathsf B_k\,\mathbf 1_{\mathcal{E}_k}.
\end{align}
Since $D_0\le s/2$ (because $x_0\in\mathcal U_0$) and all products in \eqref{eq_000059} are less than or equal to $1$,
\[
D_{n+1}(\omega)\ \le\ \frac{s}{2} + |\tilde{\mathsf{A}}_n(\omega)| + \tilde{\mathsf{B}}_n(\omega)
\ \le\ \frac{s}{2} + \sqrt{\mathsf{R}_{n}(\omega)} + \mathsf{R}_{n}(\omega).
\]
As $\omega\in E_n$ gives $\mathsf{R}_{n}(\omega)<s$, we get
$D_{n+1}(\omega)\le3s/2+\sqrt{s}$. Choosing $s>0$ so that $2s+\sqrt{s}<s_{0}$ yields
$f(x_{n+1}(\omega))-l<s_{0}$.

\emph{Case $\alpha\in[1,2)$.} \Cref{lem_D_5} yields that
\begin{align}
    \MoveEqLeft[4]D_{n+1}\mathbf 1_{\mathcal{E}_n}\ \le\
D_0\prod_{k=0}^{n}\!(1-\tau^{-2/\alpha}\eta_k^{q})\ +\
\tilde{c}\sum_{k=0}^{n}\eta_k^{\frac{2q-\alpha}{2-\alpha}}\nonumber\\
&+\sum_{k=0}^{n}\!\Big(\prod_{j=k+1}^{n}(1-\tau^{-2/\alpha}\eta_j^{q})\Big)\,\eta_k\,\mathsf A_k\,\mathbf 1_{\mathcal{E}_k}
+\frac{L}{2}\sum_{k=0}^{n}\eta_k^2\,\mathsf B_k\,\mathbf 1_{\mathcal{E}_k}.
\end{align}
Using $D_0\le s/2$, the step–size condition
$\sum_{k=0}^n\eta_k^{\frac{2 q-2}{2-\alpha}}\le s/(2\tilde c)$, and $|\tilde{\mathsf{A}}_{n}|\le \sqrt{\mathsf{R}_{n}}$, $\tilde{\mathsf{B}}_{n}\le \mathsf{R}_{n}$,
we obtain $D_{n+1}(\omega)\le 2s+\sqrt{s}<s_{0}$. Thus \Cref{lem_D_5} applies and
$x_{n+1}(\omega)\in\mathcal U$, i.e., $\omega\in\mathcal{E}_{n+1}$.

\end{proof}
\begin{lemma}\label{lemm_13_append}
Under the same conditions in \Cref{lem_d7}, and the following condition on $\eta_{0}$
\begin{align}\label{eq_0000061}
    &\eta_{0}\le \delta(2s\mathsf{C}_{1}\cdot L_{2})^{-1}\times\left[L_{1}^{2}+\sigma^{2}+\tilde{L}^{2}(R^{2}+(\mathrm{diam}(\mathcal{M}))^{2})\right]^{-1}\times\left(\sum_{k=1}^{\infty}k^{-3\eta/2}(\lfloor k/S\rfloor S+1)^{-\beta/2}\right)^{-1},
\end{align}
where $\mathsf{C}_{1}$ is a positive constant independent of parameters of the problem, we have
\[
\mathbb P(E_n)\ge\;1-\frac{\delta}{2}.
\]
\end{lemma}
\begin{proof}
\noindent\textbf{(i) Recursion for the error accumulator and decomposition of the bad event.}
Let $\widetilde E_n:=E_{n-1}\setminus E_n=E_{n-1}\cap \{\mathsf{R}_{n}\ge s\}$ denote the event that the noise budget first crosses level $s$ at time $n$, and set $\widetilde{\mathsf{R}}_{n}:=\mathsf{R}_{n}\,\mathbf 1_{E_{n-1}}$. Since the events $\{\widetilde E_k\}_{k\ge 0}$ are disjoint and $E_n=\bigcap_{k=0}^{n}\widetilde E_k^{\,c}$, it suffices to control $\sum_{k=0}^{n}\mathbb{P}(\widetilde E_k)$ via the recursion for $\widetilde{\mathsf{R}}_n$.
Note that
\[
\widetilde{\mathsf{R}}_n= \mathsf{R}_n\,\mathbf 1_{E_{n-1}}= \widetilde{\mathsf{R}}_{n-1} - \mathsf{R}_{n-1}\mathbf 1_{\widetilde E_{n-1}} + (\mathsf{R}_n-\mathsf{R}_{n-1})\,\mathbf 1_{E_{n-1}}.
\]
Taking expectations gives
\begin{align}\label{eq_00046}
    \mathbb E[\widetilde{\mathsf{R}}_n]
= \mathbb E[\widetilde{\mathsf{R}}_{n-1}]
- \mathbb E\big[\mathsf{R}_{n-1}\mathbf 1_{\widetilde E_{n-1}}\big]
+ \mathbb E\big[(\mathsf{R}_{n}-\mathsf{R}_{n-1})\,\mathbf 1_{E_{n-1}}\big].
\end{align}
Using $\mathsf{R}_{n}-\mathsf{R}_{n-1}=(\tilde{\mathsf{A}}_{n}^2-\tilde{\mathsf{A}}_{n-1}^2)+(\tilde{\mathsf{B}}_{n}-\tilde{\mathsf{B}}_{n-1})$,
\(
\tilde{\mathsf{A}}_{n}=\tilde{\mathsf{A}}_{n-1}+\eta_n(1-\eta_n \tau^{-\frac{2}{\alpha}})\,\mathsf{A}_{n}\,\mathbf 1_{\mathcal{E}_n},
\) and $\tilde{\mathsf{B}}_{n}-\tilde{\mathsf{B}}_{n-1}=\frac{L_{2}}{2}\eta_n^2\|\vv_{n}\|^2\,\mathbf 1_{\mathcal{E}_n}$,
we get
\begin{align}\label{eq_0047}
&\mathsf{R}_{n}-\mathsf{R}_{n-1}= \eta_n^2(1-\eta_n \tau^{-\frac{2}{\alpha}})^2\,\mathsf{A}_{n}^2\,\mathbf 1_{\mathcal{E}_n}+ 2\eta_n(1-\eta_n \tau^{-\frac{2}{\alpha}})\,\mathsf{A}_{n}\,\mathbf 1_{\mathcal{E}_n}\tilde{\mathsf{A}}_{n-1}+ \frac{L_{2}}{2}\eta_n^2\|\vv_{n}\|^2\,\mathbf 1_{\mathcal{E}_n}.
\end{align}
By taking expectations, we use the following bounds for the terms in \eqref{eq_0047}: For the first term in the right hand-side of \eqref{eq_0047}, the Cauchy-Schwartz inequality for $\mathsf A_k:=-\langle \nabla F(x_k),\,\vv_{k}-\nabla F(x_k)\rangle$ yields
\begin{align}\label{eq_010153}
    \MoveEqLeft[4]\mathbb E[\mathsf{A}_{n}^2\,\mathbf 1_{\mathcal{E}_n}]\le \mathbb{E}[\|\nabla F(x_{n})\|^{2}\cdot\|\vv_{n}-\nabla F(x_{n})\|^{2}\mathbf 1_{\mathcal{E}_n}]\nonumber\\
    &\overset{(a)}{\le} L_{1}^2\cdot\mathbb{E}[\|\vv_{n}-\nabla F(x_{n})\|^{2}\mathbf 1_{\mathcal{E}_n}]\nonumber\\
    &\overset{(b)}{\le} \frac{L_{1}^{2}\sigma^{2}}{(\lfloor n/S\rfloor S+1)^{\beta}}+\frac{4L_{1}^{2}\tilde{L}^{2}\cdot[n-\lfloor n/S\rfloor S]}{(\lfloor n/S\rfloor S+1)^{\beta}}\cdot\max_{\lfloor n/S\rfloor S+1\le k\le n}\mathbb{E}[\|x_{k}-x_{k-1}\|^{2}]\nonumber\\
    &\overset{(c)}{\le} \frac{L_{1}^{2}\sigma^{2}}{(\lfloor n/S\rfloor S+1)^{\beta}}+\frac{4L_{1}^{2}\tilde{L}^{2}\cdot[n-\lfloor n/S\rfloor S]}{(\lfloor n/S\rfloor S+1)^{\beta}}\cdot(2R^{2}+2(\text{diam}(\mathcal{M}))^{2}),
\end{align}
where (a) follows from the local $L_{1}$-Lipschitz continuity of $F$, and (b) applies \eqref{eq_010146} to bound
$\mathbb{E}\!\left[\|\vv_{n}-\nabla F(x_{n})\|^{2}\right]$ from above. (c) follows from the compactness of $\mathcal M$ and the fact that $\,x_k\in\mathcal U\,$ for all $k=\lfloor n/S\rfloor S,\ldots,n$. The second term in \eqref{eq_0047} is not tightly controlled by a direct Cauchy--Schwarz bound. Instead, we exploit the SARAH recursion:

\smallskip\noindent\textbf{(ii) Sharp control of the cross term $\mathbb{E}[\mathsf{A}_{n}\tilde{\mathsf{A}}_{n-1}\boldsymbol{1}_{\mathcal{E}_{n}}]$.}
Let $e_t := \vv_t - \nabla F(x_t)$ and recall $\mathsf A_t = -\langle \nabla F(x_t), e_t\rangle$ and that
$\tilde{\mathsf A}_{n-1}$ is $\mathcal F_{n-1}$–measurable by construction. On inner steps, the SARAH error obeys
\begin{align}
    &e_t \;=\; e_{t-1} \;+\; \Delta_t,\nonumber\\
    &\mathbb E[\Delta_t\mid \mathcal F_{t-1}] = 0,\nonumber\\
 &\mathbb E\big[\|\Delta_t\|^2\mid \mathcal F_{t-1}\big] \;\le\; \frac{4\tilde L^{\,2}}{n_g^{\,t}}\;\|x_t-x_{t-1}\|^2.
\end{align}
At (mini-batch) restarts,
$\mathbb E[\|e_t\|^2\mid \mathcal F_{t-1}]\le \sigma^2/n_g^{\,t}$. Using  the law of total expectation and the fact that $\mathbb E[\Delta_n\mid \mathcal F_{n-1}]=0$, we get
\[
\mathbb E\big[\mathsf A_n\,\tilde{\mathsf A}_{n-1}\,\mathbf 1_{\mathcal{E}_n}\big]
\;=\;
\mathbb E\!\left[\,-\big\langle \nabla F(x_n),\,e_{n-1}\big\rangle\,\tilde{\mathsf A}_{n-1}\,\mathbf 1_{\mathcal{E}_{n-1}}\right].
\]
By Cauchy–Schwarz and local boundedness of the gradient on $\mathcal{E}_{n-1}$ (from \eqref{eq_local_lip}, $\|\nabla F(x)\|\le L_1$ on the tube),
\begin{align}\label{eq:cross-CS}
&\mathbb E\big[\mathsf A_n\,\tilde{\mathsf A}_{n-1}\,\mathbf 1_{\mathcal{E}_n}\big]\;\le\;
L_1\,
\Big(\mathbb E\big[\|e_{n-1}\|^2\,\mathbf 1_{\mathcal{E}_{n-1}}\big]\Big)^{1/2}
\Big(\mathbb E\big[\tilde{\mathsf A}_{n-1}^{\,2}\,\mathbf 1_{\mathcal{E}_{n-1}}\big]\Big)^{1/2}.
\end{align}
Let $t_0:=\lfloor n/S\rfloor S$ be the last restart index. Within the current epoch, the batch size is 
$n_g^{\,t}\ge (\lfloor t/S\rfloor S+1)^{\beta}$. Unrolling $e_t$ from $t$ to $t_0$ and taking expectations yields
\begin{align*}
    &\mathbb E\big[\|e_{n-1}\|^2\,\mathbf 1_{\mathcal{E}_{n-1}}\big]
\;\le\;
\frac{\sigma^2}{(\lfloor n/S\rfloor S+1)^\beta}\;+\;\frac{4\tilde L^{\,2}}{(\lfloor n/S\rfloor S+1)^\beta}
\sum_{j=t_0+1}^{n-1}\mathbb E\!\big[\|x_j-x_{j-1}\|^2\,\mathbf 1_{\mathcal{E}_{j-1}}\big].
\end{align*}
By using the identity $\|x_j-x_{j-1}\|^2=\eta_{j-1}^2\|\vv_{j-1}\|^2$ and the local boundedness
$\mathbb E[\|\vv_{j}\|^2\,\mathbf 1_{\mathcal{E}_j}]\le C_v$
for a constant $C_v$ depending only on the local neighborhood\footnote{Using $\|a+b\|^{2}\le 2\|a\|^{2}+2\|b\|^{2}$, we obtain
\(
\mathbb{E}\big[\|\vv_{j}\|^{2}\big]
\;\le\; 2\,\mathbb{E}\big[\|\vv_{j}-\nabla F(\xv_{j})\|^{2}\big]
      + 2\,\mathbb{E}\big[\|\nabla F(\xv_{j})\|^{2}\big]
\;\le\; 2\big(\sigma^{2}+4\tilde{L}^{2}R^{2}\big)+2L_{1}^{2},
\)
where the second inequality uses \eqref{eq_0000045} and the Lipschitz bounds for $F$.}, we obtain
\begin{align}
    \mathbb E\big[\|e_{n-1}\|^2\,\mathbf 1_{\mathcal{E}_{n-1}}\big]
\;&\le\;
\frac{\sigma^2}{(\lfloor n/S\rfloor S+1)^\beta}\;+\;
\frac{4\tilde L^{\,2}\,C_v}{(\lfloor n/S\rfloor S+1)^\beta}\,
\sum_{j=t_0+1}^{n-1}\eta_{j-1}^2.
\end{align}
Since the step-size schedule satisfies $\sum_{k\ge 0}\eta_k^2<\infty$ and is non-increasing, the tail
$\sum_{j=t_0+1}^{n-1}\eta_{j-1}^2$ is upper bounded by a constant $H\eta_{0}^{2}$ \emph{independent of $n$ and $S$}. Hence
\begin{equation}\label{eq:e-noS}
\mathbb E\big[\|e_{n-1}\|^2\,\mathbf 1_{\mathcal{E}_{n-1}}\big]
\;\le\;
\frac{\sigma^2 + 4\,\tilde L^{\,2}\,C_v\,H\eta_{0}^{2}}{(\lfloor n/S\rfloor S+1)^\beta}.
\end{equation}
To bound $\mathbb{E}[\tilde{\mathsf{A}}_{n-1}^{2}\boldsymbol{1}_{\mathcal{E}_{n}}]$, we use the definition of $\tilde{\mathsf{A}}_{n}$ in \eqref{eq_def_tilde_A_k} as follows:
\begin{align}
    \MoveEqLeft[4]\mathbb{E}[\tilde{\mathsf{A}}_{n-1}^{2}\boldsymbol{1}_{\mathcal{E}_{n}}]\le \sum_{j=0}^{n-1}\Bigg[\prod_{i=j+1}^{n-1}\Big(1-\tfrac{\eta_i}{\tau}\Big)^{\!2}\Bigg]\eta_j^{2}\mathbb{E}[\mathsf{A}_{j}^{2}\boldsymbol{1}_{\mathcal{E}_{n}}]\nonumber\\
    &\overset{(a)}{\le} L_{1}^{2} \sum_{j=0}^{n-1}\Bigg[\prod_{i=j+1}^{n-1}\Big(1-\tfrac{\eta_i}{\tau}\Big)^{\!2}\Bigg]\eta_j^{2}\times \Bigg\{\frac{\sigma^{2}}{(\lfloor j/S\rfloor S+1)^{\beta}}+\frac{4\tilde{L}^{2}\cdot[j-\lfloor j/S\rfloor S]}{(\lfloor j/S\rfloor S+1)^{\beta}}\cdot(2R^{2}+2(\text{diam}(\mathcal{M}))^{2})\Bigg\}\nonumber\\
    &\le L_{1}^{2}\cdot\left\{\sigma^{2}+4\tilde{L}^{2}\cdot(2R^{2}+2(\text{diam}(\mathcal{M}))^{2})\right\}\times\sum_{j=0}^{n-1}\Bigg[\prod_{i=j+1}^{n-1}\Big(1-\tfrac{\eta_i}{\tau}\Big)^{\!2}\Bigg]\eta_j^{2}\nonumber\\
    &\overset{(b)}{\le} L_{1}^{2}\cdot\left\{\sigma^{2}+4\tilde{L}^{2}\cdot(2R^{2}+2(\text{diam}(\mathcal{M}))^{2})\right\}\cdot \frac{\eta_{0}^{2}C}{n^{\eta}},\label{eq_010159}
\end{align}
where (a) comes from the bound on $\mathbb{E}[\mathsf{A}_{j}^{2}\boldsymbol{1}_{\mathcal{E}_{n}}]$ in  \eqref{eq_010153} and (b) follows from the bound on $S_n\;:=\;\sum_{j=0}^{n-1}\Bigg[\prod_{i=j+1}^{n-1}\Big(1-\tfrac{\eta_i}{\tau}\Big)^{\!2}\Bigg]\eta_j^{2}$ derived in \Cref{lem_A_n} with $\eta_{n}=\eta_{0}(n+1)^{-\eta}$ and some constant $C>0$.

Plugging \eqref{eq:e-noS} and \eqref{eq_010159} into
\eqref{eq:cross-CS} gives
\begin{align*}
    &\mathbb E\big[\mathsf A_n\,\tilde{\mathsf A}_{n-1}\,\mathbf 1_{\mathcal{E}_n}\big]
\;\le\;
L_1^2\,
\sqrt{\frac{\sigma^2 + 4\tilde L^{\,2}\,C_v\,H\eta_{0}^{2}}{(\lfloor n/S\rfloor S+1)^\beta}}\times\sqrt{\Big\{\sigma^2 + c_0\,\tilde L^{\,2}\big(R^2+\mathrm{diam}(\mathcal M)^2\big)\Big\}\cdot\frac{\eta_0^{2}\,C}{n^{\eta}}}.
\end{align*}
Finally, for the last term in the right hand-side of \eqref{eq_0047}, by using \eqref{eq_010153}, we get
\begin{align}\label{eq_000050}
    &\mathbb{E}[\|\vv_{n}\|^2\,\mathbf 1_{\mathcal{E}_n}]\nonumber\\
    &\le 2\mathbb{E}[\|\nabla F(x_{n})\|^{2}]+2\mathbb{E}[\|\vv_{n}-\nabla F(x_{n})\|^{2}]\nonumber\\
    &\le 2L_{1}^{2}+2\Bigg\{\frac{\sigma^{2}}{(\lfloor n/S\rfloor S+1)^{\beta}}+\frac{4\tilde{L}^{2}[n-\lfloor n/S\rfloor S]}{(\lfloor n/S\rfloor S+1)^{\beta}}\cdot(2R^{2}+2(\text{diam}(\mathcal{M}))^{2})\Bigg\},
\end{align}

Combining \eqref{eq_0047} with all the upper bounds computed above for each term on the right-hand side yields the following upper bound for $\mathbb{E}[\mathsf{R}_{n}]$.
\begin{align}\label{eq_0052}
    \mathbb{E}[\mathsf{R}_{n}]&\le \mathbb{E}[\mathsf{R}_{n-1}]+ \eta_n^2(1-\eta_n \tau^{-\frac{2}{\alpha}})^2\,\mathbb{E}[\mathsf{A}_{n}^2\,\mathbf 1_{\mathcal{E}_n}]+ 2\eta_n(1-\eta_n \tau^{-\frac{2}{\alpha}})\,\mathbb{E}[\mathsf{A}_{n}\,\mathbf 1_{\mathcal{E}_n}\tilde{\mathsf{A}}_{n-1}]+\frac{L_{2}}{2}\eta_n^2\mathbb{E}[\|\vv_{n}\|^2\,\mathbf 1_{\mathcal{E}_n}]\nonumber\\
&\le \mathbb{E}[\mathsf{R}_{n-1}]+ \eta_n^2(1-\eta_n \tau^{-\frac{2}{\alpha}})^2\times L_{1}^{2}\Bigg\{\frac{\sigma^{2}}{(\lfloor n/S\rfloor S+1)^{\beta}}+\frac{4\tilde{L}^{2}\cdot[n-\lfloor n/S\rfloor S]}{(\lfloor n/S\rfloor S+1)^{\beta}}\cdot(2R^{2}+2(\text{diam}(\mathcal{M}))^{2})\Bigg\}\nonumber\\
&
+2\eta_n(1-\eta_n \tau^{-\frac{2}{\alpha}})\cdot L_1^2\,
\sqrt{\frac{\sigma^2 + 4\tilde L^{\,2}\,C_v\,H\eta_{0}^{2}}{(\lfloor n/S\rfloor S+1)^\beta}}\times\sqrt{\Big\{\sigma^2 + c_0\,\tilde L^{\,2}\big(R^2+\mathrm{diam}(\mathcal M)^2\big)\Big\}\,\frac{\eta_0^{2}\,C}{n^{\eta}}} \nonumber\\
&+ L_{1}L_{2}\eta_n^2+\frac{L_{2}}{2}\eta_n^2\Bigg\{\frac{\sigma^{2}}{(\lfloor n/S\rfloor S+1)^{\beta}}+\frac{4\tilde{L}^{2}[n-\lfloor n/S\rfloor S]}{(\lfloor n/S\rfloor S+1)^{\beta}}\cdot(2R^{2}+2(\text{diam}(\mathcal{M}))^{2})\Bigg\}\nonumber\\
&\le \mathbb{E}[\mathsf{R}_{n-1}]+\eta_{0}\mathsf{C}_{1}\cdot n^{-3\eta/2}(\lfloor n/S\rfloor S+1)^{-\beta/2}\times\Bigg[L_{2}L_{1}^{2}+L_{2}\sigma^{2}+L_{2}\tilde{L}^{2}(R^{2}+(\text{diam}(\mathcal{M}))^{2})\Bigg],
\end{align}
where in the last step we used $\eta_{n}=\eta_{0}/(n+1)^{\eta}$ and we grouped each term decaying with $n$ in the dominant term on the right hand side of \eqref{eq_0052}, multiplied by constant $\mathsf{C}_{1}>0$ independent of the problem parameters.

\smallskip\noindent\textbf{(iii) Final summation and tail bound.}
Now by plugging the bound in \eqref{eq_0052} for $\mathbb{E}[\mathsf{R}_{n}-\mathsf{R}_{n-1}]$ into \eqref{eq_00046}, we get
\begin{align}\label{eq_0053}
\MoveEqLeft[2]\mathbb E[\widetilde{\mathsf{R}}_n]
\le \mathbb E[\widetilde{\mathsf{R}}_{n-1}]
- \mathbb E\big[\mathsf{R}_{n-1}\mathbf 1_{\widetilde E_{n-1}}\big]+ \eta_{0}\mathsf{C}_{1}\cdot L_{2}\Big[L_{1}^{2}+\sigma^{2}+\tilde{L}^{2}(R^{2}+(\text{diam}(\mathcal{M}))^{2})\Big].    
\end{align}
Recall that $\widetilde E_n:=E_{n-1}\setminus E_n=\{\mathsf{R}_{n-1}<s,\ \mathsf{R}_n\ge s\}$. By using the fact that $\mathsf{R}_{n-1}\ge s$ on $\widetilde E_{n-1}$ and Markov's inequality,
\begin{align}
    \mathbb P(\widetilde E_{n-1})
&=\mathbb P\big(E_{n-1}\cap\{\mathsf{R}_n\ge s\}\big)
=\mathbb E\!\left[\mathbf 1_{E_{n-1}}\mathbf 1_{\{\mathsf{R}_n>s\}}\right]\nonumber\\
&\le \frac{1}{s}\,\mathbb E\!\left[\mathbf 1_{E_{n-1}}\mathsf{R}_n\right]
=\frac{\mathbb E[\widetilde{\mathsf{R}}_n]}{s}.
\end{align}
On the other hand, \eqref{eq_0053} yields the recursive bound
\begin{align}
s\,&\mathbb P(\widetilde E_{n-1})
\;\le\; \mathbb E[\widetilde{\mathsf{R}}_{n-1}]-\mathbb E[\widetilde{\mathsf{R}}_n]+ \eta_{0}\mathsf{C}_{1}\cdot L_{2}\left[L_{1}^{2}+\sigma^{2}+\tilde{L}^{2}(R^{2}+(\text{diam}(\mathcal{M}))^{2})\right]\times n^{-3\eta/2}(\lfloor n/S\rfloor S+1)^{-\beta/2}.
\end{align}
Summing this inequality from $k=1$ to $n+1$ and using $\widetilde {\mathsf{R}}_0=0$ and $\widetilde{\mathsf{R}}_n\ge 0$ gives
\begin{align}
    &s\sum_{k=1}^{n+1} \mathbb P(\widetilde E_{k-1})\le \eta_{0}\mathsf{C}_{1}\cdot L_{2}\left[L_{1}^{2}+\sigma^{2}+\tilde{L}^{2}(R^{2}+(\text{diam}(\mathcal{M}))^{2})\right]\times\sum_{k=1}^{n+1}k^{-3\eta/2}(\lfloor k/S\rfloor S+1)^{-\beta/2}.
\end{align}
By the step-size condition \eqref{eq_0000061}, we get
\[
\sum_{k=0}^{n}\mathbb P(\widetilde E_{k}) <\ \frac{\delta}{2}.
\]
Since the events $\{\widetilde E_k\}_{k\ge 0}$ are disjoint and
$E_n=\bigcap_{k=0}^{n}\widetilde E_k^{\,c}$, we obtain
\[
\mathbb P(E_n)
\;=\;1-\mathbb P\!\left(\bigcup_{k=0}^{n}\widetilde E_k\right)
\;=\;1-\sum_{k=0}^{n}\mathbb P(\widetilde E_k)
\;\ge\;1-\frac{\delta}{2}.
\]
\end{proof}
\begin{lemma}\label{lem:D8}
Suppose $\eta_0$ is small enough that
\begin{align}
    &\eta_{0}^{2}\cdot\frac{4\mathsf{C}_{2}\cdot [L_{1}^{2}+\sigma^{2}+\tilde{L}^{2}(R^{2}+(\mathrm{diam}(\mathcal{M}))^{2})]}{R^2}\times\sum_{k=0}^\infty \frac{1}{(k+1)^{2\eta}} \;\le\; \frac{\delta}{2}\,.
\end{align}
Then $\mathbb P(C_n)\ge 1-\delta/2$, where
\(
C_n:=\{\|x_{k+1}-x_k\|\le R/2,\ \forall\,k\le n\}.
\)
\end{lemma}

\begin{proof}
By \Cref{lem_d7}, we have $\mathbb P(E_n)\ge 1-\delta/2$. By \eqref{eq_000050} in the proof of \Cref{lem_d7}, we have
\begin{align}
    \MoveEqLeft[2]\mathbb E[\|\vv_{n}\|^2\,\mathbf 1_{\mathcal{E}_n}]\le2L_{1}^{2}+2\Bigg\{\frac{\sigma^{2}}{(\lfloor n/S\rfloor S+1)^{\beta}}+\frac{4\tilde{L}^{2}[n-\lfloor n/S\rfloor S]}{(\lfloor n/S\rfloor S+1)^{\beta}}\cdot(2R^{2}+2(\text{diam}(\mathcal{M}))^{2})\Bigg\}\nonumber\\
    &\le \mathsf{C}_{2}\cdot [L_{1}^{2}+\sigma^{2}+\tilde{L}^{2}(R^{2}+(\text{diam}(\mathcal{M}))^{2})],
\end{align}
where $\mathsf{C}_{2}$ is the constant independent of the parameters of the problem.
Moreover, by the additional step-size assumption and Markov’s inequality,
\begin{align}
\MoveEqLeft[4]\mathbb P(C_n)
= \mathbb P\Big(\forall\,k\le n:\ \|x_{k+1}-x_k\|\le \tfrac R2\Big)\nonumber\\
&\;\ge\; 1-\sum_{k=0}^n \mathbb P\!\left(\|x_{k+1}-x_k\|> \tfrac R2\right) \\
&= 1-\sum_{k=0}^n \mathbb P\!\left(\|\vv_{k}\|> \tfrac R{2\eta_k}\right)
\;\ge\; 1-\sum_{k=0}^n \frac{\mathbb E[\|\vv_{k}\|^2]}{(R/(2\eta_k))^2}
\nonumber\\
&\ge 1-\frac{4\mathsf{C}_{2}\cdot [L_{1}^{2}+\sigma^{2}+\tilde{L}^{2}(R^{2}+(\text{diam}(\mathcal{M}))^{2})]}{R^2}\cdot\sum_{k=1}^n \eta_k^2\nonumber\\
&\ge 1-\frac{\delta}{2}.
\end{align}

\end{proof}

From \cref{lemm_13_append,lem:D8}, we have
\begin{align}
    &\mathbb P(E_n)\ge 1-\delta/2,\quad \mathbb P(C_n)\ge 1-\delta/2\nonumber\\
    &\Longrightarrow\quad \mathbb P(E_n\cap C_n)\ge 1-\delta,
\end{align}
which implies the following lemma:
\begin{lemma}\label{lem_15_append}
Under the conditions in \Cref{lem_d7,lemm_13_append,lem:D8}, we have $\mathbb{P}[\hat{E}_{n}]\ge 1-\delta$.    
\end{lemma}
On $E_n\cap C_n$, \Cref{lem_d7} gives $D_{k+1}<s$ for all $k\le n$, hence $x_{k+1}\notin\mathcal R$; since $x_{k+1}\in\mathbb B_2(\mathcal M;3R/4)$ by $C_n$, we must have $x_{k+1}\in\mathbb B_2(\mathcal M;R/2)=\mathcal U$. Therefore
\[
E_n\cap C_n\ \subset\ \mathcal{E}_{n+1}.
\]
Starting from $x_0\in\mathcal U_0$ (so $D_0\le s/2$), an induction yields
\(
\mathbb P(\mathcal{E}_n)\ge 1-\delta
\)
for all $n$.

\subsection{Proof of optimal local convergence rate of SARAH}\label{append:G3}
Up to this point, we have shown that the SARAH iterates remain within a neighborhood of the connected component $\mathcal M$ of local minimizers with high probability; concretely, $\mathbb P(\mathcal{E}_n)\ge 1-\delta$ for all $n\ge 1$. We now turn to deriving the optimal local convergence rate of SARAH by exploiting the dynamics induced by the local $\alpha$-P\L\ property.

\begin{lemma}\label{lem:smooth-step}
Let $f:\mathbb{R}^d\to\mathbb{R}$ be $L$-smooth. For any $x_t\in\mathbb{R}^d$, step size $\eta>0$, and vector $\vv_{t}\in\mathbb{R}^d$, define $x_{t+1}=x_t-\eta \vv_{t}$. Then
\begin{align}
    &f(x_{t+1})
\;\le\;
f(x_t)
-\frac{\eta}{2}\bigl\|\nabla f(x_t)\bigr\|^2
\nonumber\\
&-\Bigl(\frac{1}{2\eta}-\frac{L}{2}\Bigr)\bigl\|x_{t+1}-x_t\bigr\|^2+\frac{\eta}{2}\bigl\|\vv_{t}-\nabla f(x_t)\bigr\|^2 .
\end{align}
\end{lemma}

\begin{proof}
By $L$-smoothness, for any $y$,
\[
f(y)\le f(x_t)+\langle \nabla f(x_t),y-x_t\rangle+\frac{L}{2}\|y-x_t\|^2 .
\]
Take $y=x_{t+1}=x_t-\eta \vv_{t}$ and write
\begin{align*}
    \langle \nabla f(x_t),x_{t+1}-x_t\rangle
&= \langle \nabla f(x_t)-\vv_{t},\,x_{t+1}-x_t\rangle\nonumber\\
&\qquad\qquad+ \langle \vv_{t},\,x_{t+1}-x_t\rangle .
\end{align*}
Let $\bar x_{t+1}:=x_t-\eta \nabla f(x_t)$. Since $x_{t+1}-x_t=-\eta \vv_{t}$,
\begin{align}
    &\langle \vv_{t}, x_{t+1}-x_t\rangle=-\eta\|\vv_{t}\|^2,\nonumber\\
&\langle \nabla f(x_t)-\vv_{t}, x_{t+1}-x_t\rangle
= -\frac{1}{\eta}\langle x_{t+1}-\bar x_{t+1},\,x_{t+1}-x_t\rangle .
\end{align}
Using the three-point identity
\[
\langle \uv-\vv,\,\uv-\wv\rangle=\tfrac12\bigl(\|\uv-\wv\|^2-\|\vv-\wv\|^2+\|\uv-\vv\|^2\bigr)
\]
with $\uv=x_{t+1}$, $\vv=\bar x_{t+1}$, $\wv=x_t$, we obtain
\begin{align*}
    &-\frac{1}{\eta}\langle x_{t+1}-\bar x_{t+1},\,x_{t+1}-x_t\rangle=\nonumber\\
& -\frac{1}{2\eta}\|x_{t+1}-x_t\|^2
+\frac{\eta}{2}\|\nabla f(x_t)\|^2
-\frac{\eta}{2}\|\vv_{t}-\nabla f(x_t)\|^2 .
\end{align*}
Substituting these relations into the smoothness bound and grouping terms gives
\begin{align*}
    &f(x_{t+1})\le
f(x_t)
-\frac{\eta}{2}\|\nabla f(x_t)\|^2\nonumber\\
&-\Bigl(\frac{1}{2\eta}-\frac{L}{2}\Bigr)\|x_{t+1}-x_t\|^2
+\frac{\eta}{2}\|\vv_{t}-\nabla f(x_t)\|^2 ,
\end{align*}
as claimed.
\end{proof}
\begin{lemma}\label{lemma_recursion}
    For every $F\in\mathcal{F}_{\alpha}$, given $\{x_{t}\}_{t=1}^{T}\in \mathcal{N}(\mathcal{M})$, $l=F(x^{*})$ for all $x^{*}
    \in \mathcal{M}$, we have the following recursion inequality for $\delta_{t}=\mathbb{E}[F(x_{t})-l]$:
    \begin{align}
        \delta_{t+1}\le \delta_{t}-\frac{\eta_{t}}{2\tau^{\frac{2}{\alpha}}}\delta_{t}^{\frac{2}{\alpha}}+\eta_{t}\mathbb{E}[\|\vv_{t}-\nabla F(x_{t})\|^{2}].
    \end{align}
\end{lemma}
\begin{proof}
By applying \Cref{lem:smooth-step} to $F$ and $\eta_{t}\le \frac{1}{2L_{2}}$, we have
\begin{align}
    F(x_{t+1})\le F(x_t)
-\frac{\eta_{t}}{2}\|\nabla F(x_t)\|^2
+\frac{\eta_{t}}{2}\|\vv_{t}-\nabla F(x_t)\|^2.
\end{align}
\newline
Using \Cref{assum_PL_alpha}, we have
\begin{align}\label{D.34}
    &F(x_{t+1})\le F(x_{t})-\frac{\eta_{t}}{2\tau^{\frac{2}{\alpha}}}(F(x_{t})-l)^{\frac{2}{\alpha}}+\frac{\eta_{t}}{2}\|\vv_{t}-\nabla F(x_{t})\|^{2}.
\end{align}
Let us define $\delta_{t}:=\mathbb{E}[F(x_{t})]-l$. By taking expectation of both sides of \eqref{D.34} and using Jensen's inequality ($\mathbb{E}[x^{2/\alpha}]\ge (\mathbb{E}[x])^{2/\alpha}$ for $\alpha\in[1,2]$), we have
\begin{align}\label{D.341}
    \delta_{t+1}\le \delta_{t}-\frac{\eta_{t}}{2\tau^{\frac{2}{\alpha}}}\delta_{t}^{\frac{2}{\alpha}}+\frac{\eta_{t}}{2}\mathbb{E}[\|\vv_{t}-\nabla F(x_{t})\|^{2}].
\end{align}    
\end{proof}
\begin{lemma}\label{lemma:rec_eq}
    If $\mathbb{E}[\|\vv_{t}-\nabla F(x_{t})\|^{2}]=C_{0}t^{-\beta}$, $\eta_{t}=\eta_{0}t^{-\eta}$, and $\eta<1$, then
    \begin{align}
    \delta_{t}=\mathcal{O}\left(t^{-\min\left\{\frac{\alpha\beta}{2},\frac{\alpha(1-\eta)}{2-\alpha}\right\}}\right).
    \end{align}
\end{lemma}
\begin{proof}
Set $\gamma:=\min\left\{\frac{\alpha\beta}{2},\frac{\alpha(1-\eta)}{2-\alpha}\right\}$ and define $B_{k}:=(k+1)^{\gamma}\delta_{k}$ for $k\ge0$. We show that $B_{k}=\mathcal{O}(1)$ for $k\ge 1$, which implies $\delta_k=\mathcal{O}((k+1)^{-\gamma})$.
\begin{align}
    \MoveEqLeft[1]B_{k+1}\le (k+2)^{\gamma}\delta_{k}+(k+2)^{\gamma}\frac{\eta_{k}}{2}\mathbb{E}[\|\vv_{k}-\nabla F(x_{k})\|^{2}]-(k+2)^{\gamma}\frac{\eta_{k}}{2\tau^{\frac{2}{\alpha}}}\delta_{k}^{\frac{2}{\alpha}}.\nonumber\\
    &=\left(\frac{k+2}{k+1}\right)^{\gamma}\Bigg[B_{k}+C_{0}\eta_{0}(k+1)^{\gamma}(k+1)^{-\eta-\beta}-(k+1)^{\gamma-\frac{2}{\alpha}\gamma-\eta}\cdot\frac{\eta_{0}B_{k}^{\frac{2}{\alpha}}}{2\tau^{\frac{2}{\alpha}}}\Bigg]\label{D17}\\
        &=B_{k}+\left[\left(1+\frac{1}{k+1}\right)^{\gamma}-1\right]B_{k}+\left(\frac{k+2}{k+1}\right)^{\gamma}\Bigg[C_{0}\eta_{0}(k+1)^{\gamma-\eta-\beta}-(k+1)^{\gamma-\frac{2}{\alpha}\gamma-\eta}\cdot\frac{\eta_{0}B_{k}^{\frac{2}{\alpha}}}{2\tau^{\frac{2}{\alpha}}}\Bigg]
    \end{align}
    where \eqref{D17} is from $\eta_{k}=\eta_{0}{(k+1)^{-\eta}}$.
    \newline
Note that for any $k\in \mathbb{N}\cup\{0\}$ we have
\begin{align}\label{eq_0001}
(k+2)^{\gamma}-(k+1)^{\gamma}=(k+1)^{\gamma}[(1+(k+1)^{-1})^{\gamma}-1]\le c_{\gamma}(k+1)^{\gamma-1}
\end{align}
where $c_{\gamma}=\gamma 2^{\gamma-1}$ and the last inequality follows from 
\begin{align}\label{eq_0002}
    (1+a)^{\gamma}-1=\int_{1}^{1+a}\gamma x^{\gamma-1}dx\le \gamma\cdot(1+a-1)\cdot\left(1+a\right)^{\gamma-1}\le \gamma 2^{\gamma-1}a
\end{align}
for $a=(k+1)^{-1}$. Hence
    \begin{align}
        &B_{k+1}-B_{k}\le c_{\gamma}(k+1)^{-1}B_{k}+2^{\gamma}\left[C_{0}\eta_{0}(k+1)^{\gamma-\eta-\beta}-(k+1)^{\gamma-\frac{2}{\alpha}\gamma-\eta}\cdot\frac{\eta_{0}B_{k}^{\frac{2}{\alpha}}}{2\tau^{\frac{2}{\alpha}}}\right]\nonumber\\
        &=(k+1)^{\gamma-\frac{2}{\alpha}\gamma-\eta}\Bigg(c_{\gamma}(k+1)^{-(\gamma-\frac{2}{\alpha}\gamma-\eta+1)}B_{k}+2^{\gamma}\Bigg[C_{0}\eta_{0}(k+1)^{-(-\frac{2}{\alpha}\gamma+\beta)}-\frac{\eta_{0}B_{k}^{\frac{2}{\alpha}}}{2\tau^{\frac{2}{\alpha}}}\Bigg]\Bigg)\nonumber\\
        &\le (k+1)^{-1}\left(c_{\gamma}B_{k}+2^{\gamma}\left[C_{0}\eta_{0}-\frac{\eta_{0}B_{k}^{\frac{2}{\alpha}}}{2\tau^{\frac{2}{\alpha}}}\right]\right),\label{D20}
    \end{align}
	where \eqref{D20} uses that $\gamma\le \alpha\beta/2$ and $\gamma\le \alpha(1-\eta)/(2-\alpha)$, i.e.,
	$-{2\gamma}/{\alpha}+\beta\ge 0$ and $\gamma-{2\gamma}/{\alpha}-\eta+1\ge0$. To upper bound \eqref{D20}, we use the following lemma.
\begin{lemma}\label{lemm_max_poly}
Let $F(B):=A_{0}B-A_{1}B^{{2}/{\alpha}}+A_{2}$ where $A_{0}>0$, $A_{1}>0$, $A_{2}\ge0$, and $1\le\alpha<2$. Then for $B\ge \max\{{A_{2}}/{A_{0}},\left({2A_{0}}/{A_{1}}\right)^{{\alpha}/(2-\alpha)}\}$, $F(B)\le 0$ and for all $B\ge0$, we have $F(B)\le A_{2}+\left({\alpha}/{2}\right)^{{\alpha}/(2-\alpha)}\cdot (2-\alpha)/{2}\cdot{A_{0}^{{2}/(2-\alpha)}}{A_{1}^{-{\alpha}/(2-\alpha)}}$.
\end{lemma}
Define
\[
A_{0}:=c_{\gamma},\qquad
A_{1}:=2^{\gamma-1}\eta_{0}\tau^{-{2}/{\alpha}},\qquad
A_{2}:=2^{\gamma}C_{0}\eta_{0},
\]
and set $M:=\max\left\{{A_{2}}/{A_{0}},\left({2A_{0}}/{A_{1}}\right)^{{\alpha}/(2-\alpha)}\right\}$ and
$M':=A_{2}+\left({\alpha}/{2}\right)^{{\alpha}/(2-\alpha)}\cdot (2-\alpha)/2\cdot{A_{0}^{{2}/(2-\alpha)}}{A_{1}^{-{\alpha}/(2-\alpha)}}$.
Then \eqref{D20} implies
\begin{align}\label{D0021}
    B_{k+1}\le B_{k}+(A_{0}B_{k}-A_{1}B_{k}^{\frac{2}{\alpha}}+A_{2})/(k+1).
\end{align}
We show that $B_{t}\le \max\{B_{0},M\}+M'/t$ for $t\ge 1$ by induction, which concludes the proof. For the base case, $B_{1}\le B_{0}+M'$ by \eqref{D0021} and \Cref{lemm_max_poly}. For the induction step, assume that $B_{k}\le \max\{B_{0},M\}+M'/k$. If $B_{k}\le M$, then by \eqref{D0021} and \Cref{lemm_max_poly}, we have $B_{k+1}\le B_k+M'/(k+1)\le M+M'/(k+1)$. If $B_{k}\ge M$, then $(A_{0}B_{k}-A_{1}B_{k}^{\frac{2}{\alpha}}+A_{2})\le 0$ by \Cref{lemm_max_poly}, and \eqref{D0021} gives $B_{k+1}\le B_{k}\le \max\{B_{0},M\}+M'/k$.
\end{proof}

\begin{proof}\label{proof_lemm_max_poly}
    For $B\ge \max\{{A_{2}}/{A_{0}},\left({2A_{0}}/{A_{1}}\right)^{{\alpha}/(2-\alpha)}\}$, we have
\[
F(B)= A_{0}B(1-A_{1}A_{0}^{-1}B^{{2}/{\alpha}-1})+A_{2}\le -A_{0}B+A_{2}\le 0.
\]
Note that $\max_{B\ge 0} F(B)$ is attained at $B_{*}\ge0$ where $F'(B_{*})=A_{0}-(2/\alpha)\cdot A_{1} B_{*}^{2/\alpha-1}=0$. This implies $B_{*}=(\alpha A_{0}/(2A_{1}))^{\alpha/(2-\alpha)}$. Consequently, 
\[
F(B)\le \max_{B\ge 0} F(B)= \left(\frac{\alpha}{2}\right)^{{\alpha}/(2-\alpha)}\cdot \frac{2-\alpha}{2}\cdot\frac{A_{0}^{{2}/(2-\alpha)}}{A_{1}^{{\alpha}/(2-\alpha)}}+A_{2}.
\]
\end{proof}

\textbf{\Cref{thm:cvg_SARAH}.}\textit{ Fix $\delta\in(0,1)$ and let $\mathcal M$ be a connected component of local minima of $F$ with level $l=F(y)$ for all $y\in\mathcal M$. 
Assume $F\in\mathcal F_\alpha$ and choose $R>0$ so that $\mathbb B_{2}(\mathcal M;R)\subseteq\mathcal N(\mathcal M)$ (cf.\ \Cref{assum_PL_alpha}). 
Let the oracle be batch-smooth, $O\in\mathsf O^{\tilde L}_{\sigma}$, and run \Cref{alg:SARAH} from $x_0$ with step sizes $\eta_t=\eta_0 (t+1)^{-\frac{\alpha}{2}-(2-\alpha)\mathsf{x}}$ and batch-sizes $n_g^t=(t+1)^{1-2\mathsf{x}}$, for an arbitrary small value $\mathsf{x}>0$.
Let $\{x_t\}_{t\ge1}$ be the iterates and set $\hat{x}:=x_T$. Then the following holds: 
\begin{itemize}
    \item There exists $s>0$ such that, with \(
\mathcal U:=\bigl\{x:\mathrm{dist}(x,\mathcal M)\le R/2\bigr\},\)
\(\mathcal U_0:=\bigl\{x:\mathrm{dist}(x,\mathcal M)<R/2,\ F(x)-l\le s/2\bigr\},\) and $\eta_{0}=\Theta(\sqrt{s}+\delta+R\sqrt{\delta})$, if $x_0\in\mathcal U_0$, the event \(\mathcal{E}_T(\mathcal{U}):=\{\ x_t\in\mathcal U\ \text{for all }t=1,\dots,T\ \}\) occurs with probability at least $1-\delta$.
\item Let $N:=\mathbb E\!\left[\sum_{t=1}^{T} n_g^t\mid \mathcal{E}_T(\mathcal{U})\right]$ denote the expected total number of oracle queries used up to iteration $T$ given $\mathcal{E}_T(\mathcal{U})$. 
Then $\mathbb E\!\left[F(\hat{x})-l\mid \mathcal{E}_T(\mathcal{U})\right]\le \varepsilon$, with \(N\ =\ \mathcal O\big(\varepsilon^{-2/\alpha}\big)\).
\end{itemize}}
\begin{proof}[Proof of \Cref{thm:cvg_SARAH}]
According to $\gamma=\min\{\frac{\alpha\beta}{2},\frac{\alpha(1-\eta)}{2-\alpha}\}$ in \Cref{lemma:rec_eq}, for $1\le\alpha<2$, we set  if $\eta=\frac{\alpha}{2}+(2-\alpha)\mathsf{x}$, $\beta=1-2\mathsf{x}$, and then $\gamma=\frac{\alpha}{2}-\alpha\mathsf{x}$. For example, when $\alpha=1$, we set $\eta=\frac{1}{2}+\mathsf{x}$, $\beta=1-2\mathsf{x}$, and then $\gamma=\frac{1}{2}-\mathsf{x}$. The first part of the claim is the result of \Cref{lem_15_append}. We first check the conditions in \Cref{lem_d7,lemm_13_append,lem:D8} for the choices of $\eta$ and $\beta$ in the following:
Note that for in \Cref{lem_d7}, the auxiliary variable $q=3-\alpha$, we get
\[
\sum_{n=1}^{\infty}\frac{1}{(n+1)^{\frac{2q-2}{2-\alpha}}}=\sum_{n=1}^{\infty}\frac{1}{(n+1)^{2}}\le\frac{\pi^{2}}{6}<\infty.
\]
Then the condition in \Cref{lem_d7}, reduces to
\begin{align}\label{eq_cond_eta}
    \eta_{0}<\mathsf{UB}_{1}:=\left(\frac{3s}{\pi^{2}\tilde c}\right)^{1/2}.
\end{align}

In \Cref{lemm_13_append}, since $3\eta/2+\beta/2> 5/4$, we have
\begin{align}
     \sum_{k=1}^{\infty}k^{-3\eta/2}(\lfloor k/S\rfloor S+1)^{-\beta/2}\le \sum_{t=1}^{\infty}\frac{1}{t^{\frac{5}{4}}}=\zeta(5/4)<\infty.
\end{align}
where $\zeta(\cdot)$ is the Riemann's zeta function.
Then the condition in \Cref{lemm_13_append} is restated as follows:
\begin{align}
    \eta_{0}&\le \mathsf{UB}_{2}:=\delta(2s\mathsf{C}_{1}\cdot L_{2})^{-1}\times\left[L_{1}^{2}+\sigma^{2}+\tilde{L}^{2}(R^{2}+(\text{diam}(\mathcal{M}))^{2})\right]^{-1}\cdot (\zeta(5/4))^{-1},
\end{align}
In \Cref{lem:D8}, since $\eta=\frac{\alpha}{2}+(2-\alpha)\mathsf{x}>\frac{1}{2}$,
\[
\sum_{k=0}^\infty \frac{1}{(k+1)^{2\eta}}=\zeta(\alpha+2(2-\alpha)\mathsf{x})<\infty.
\]
Then the condition in \Cref{lem:D8} can be restated as follows:
\begin{align}
    \eta_{0}\;\le\; \mathsf{UB}_{3}:=\sqrt{R^{2}\delta}\cdot8^{-\frac{1}{2}}(\mathsf{C}_{2}\cdot\zeta(2\eta))^{-\frac{1}{2}}\times[L_{1}^{2}+\sigma^{2}+\tilde{L}^{2}(R^{2}+(\mathrm{diam}(\mathcal{M}))^{2})]^{-\frac{1}{2}}.
\end{align}
In conclusion, we get
\begin{align}
    \eta_{0}\le \min\Bigg\{\mathsf{UB}_{1},\mathsf{UB}_{2},\mathsf{UB}_{3}\Bigg\}.
\end{align}

\smallskip
From Lemma \ref{lemma_recursion}, we have for $\delta_{t}=\mathbb{E}[(F(x_{t})-l)\boldsymbol{1}_{\mathcal{E}_{t}}]$:
    \begin{align}
        \delta_{t+1}\le \delta_{t}-\frac{\eta_{t}}{2\tau^{\frac{2}{\alpha}}}\delta_{t}^{\frac{2}{\alpha}}+\eta_{t}\mathbb{E}[\|\vv_{t}-\nabla F(x_{t})\|^{2}].
    \end{align}

Now, suppose for $(k-1)S<t\leq kS$, $k\geq 1$, we set $\varepsilon$ to $(kS)^{-{\beta}}$ in the $n_g^t$'s. Then, from Lemma \ref{lemma:vr}, for all $t\geq 0,$ 
\begin{equation}
    \delta_{t+1}\leq \delta_{t}-\frac{\eta_{t}}{2\tau^{\frac{2}{\alpha}}}\delta_{t}^{\frac{2}{\alpha}}+\eta_{t}\frac{C_g}{(\lceil t/S\rceil S)^{\beta}}, 
\end{equation}
 and  based on \Cref{lemma:rec_eq}, if $\mathbb{E}[\|\vv_{t}-\nabla F(x_{t})\|^{2}]=C_{0}t^{-\beta}$, $\eta_{t}=\eta_{0}(t+1)^{-\eta}$, and $\eta<1$, then
 \begin{align}
     \delta_t=\mathcal{O}(t^{-\gamma}),
 \end{align}
where $\gamma=\min\{\frac{\alpha\beta}{2},\frac{\alpha(1-\eta)}{2-\alpha}\}$.
 After $T=\mathcal{O}(\varepsilon^{-\frac{1}{\gamma}})$
 iterations, $\delta_{T}\le\varepsilon$.
 Furthermore, according to \Cref{lem:smooth-step}, we have: for $\eta_{t}\le \frac{1}{2L_{2}}$,
 \begin{align}
\MoveEqLeft[2]\|x_{t+1}-x_{t}\|^{2}\le4\eta_{t}[F(x_{t})-F(x_{t+1})]+2\eta_{t}^{2}\mathbb{E}[\|\vv_{t}-\nabla F(x_{t})\|^{2}]
 \end{align}
 and taking expectation and summing up from $t=(k-1)S+1$ to $t=kS$, we have:
\begin{align}\label{eq00123}
    \MoveEqLeft[4]\sum_{t=(k-1)S+1}^{kS}\mathbb{E}[\|x_{t+1}-x_{t}\|^{2}]\le4\sum_{t=(k-1)S+1}^{kS}\eta_{t}[F(x_{t})-F(x_{t+1})]+2\sum_{t=(k-1)S+1}^{kS}\eta_{t}^{2}\mathbb{E}[\|\vv_{t}-\nabla F(x_{t})\|^{2}]\nonumber\\
    &\le 4\sum_{t=(k-1)S+1}^{kS}\eta_{t}\mathbb{E}[F(x_{t})-F(x_{t+1})]+2\sum_{t=(k-1)S+1}^{kS}\eta^{2}_{t}\frac{1}{(kS)^{\beta}},
\end{align}
where we set $n_g^t$'s such that 
\begin{align*}
\mathbb{E}[\|\nabla F(x_t)-\gv_{t}\|^{2}]\le \frac{1}{(kS)^{\beta}},
\end{align*}
and $\eta_{t}=\eta_{0}(t+1)^{-\eta}$.\\
Hence, according to Lemma \ref{lemma:vr}, in order to obtain the errors
\[
\mathbb{E}[\|\nabla F(x_t)-\gv_{t}\|])\le \frac{1}{(kS)^{\beta}}.
\]
the average sample complexity of gradient $\mathbb{E}\left[\sum_{t=1}^T n_g^t\right]$ must be of the order of
\begin{align}\label{sample_g_VSCRN}
    &\sum_{\mod(t,S)=0}\frac{2\sigma^{2}}{t^{-\beta}}+\sum_{\mod(t,S)\neq 0} \frac{4\cdot \tilde{L}
^{2} S \mathbb{E}[\|x_{t+1}-x_{t}\|^2]}{(\lceil t/S\rceil S)^{-\beta}}.
\end{align}
It is enough to have the number of samples be greater than the following upper bound for \eqref{sample_g_VSCRN}:
 \begin{equation}\label{sample_gradient_VRSCRN}
 \begin{split}
           \MoveEqLeft[2]\mathbb{E}\left[\sum_{t=1}^T n_g^t\right]=\sum_{\mod(t,S)=0}\frac{2\sigma^{2}}{t^{-\beta}}+\sum_{\mod(t,S)\neq 0} \frac{4\cdot \tilde{L}^{2} S \mathbb{E}[\|x_{t}-x_{t+1}\|^2]}{(\lceil t/S\rceil S)^{-\beta}}\\
          &= \sum_{k=1}^{\lceil T/S\rceil}2\sigma^{2}(kS)^{\beta}+\sum_{k=1}^{\lceil T/S\rceil}\sum_{t=(k-1)S+1}^{kS}4\cdot \tilde{L}^{2} S (kS)^{\beta}\mathbb{E}[\|x_{t}-x_{t+1}\|^2]\nonumber\\
          &\overset{(a)}{\le} \sum_{k=1}^{\lceil T/S\rceil}2\sigma^{2}(kS)^{\beta}+ \sum_{k=1}^{\lceil T/S\rceil}4\cdot \tilde{L}^{2} S (kS)^{\beta}\times\sum_{t=(k-1)S+1}^{kS}\Bigg[4\eta_{t}\mathbb{E}[F(x_{t})-F(x_{t+1})+\eta_{t}^{2}(kS)^{-\beta}]\Bigg]\nonumber\\
          &\overset{(b)}{\le} \mathcal{O}(\sigma^{2}T^{\beta+1}S^{-1})+\sum_{k=1}^{\lceil T/S\rceil}4\cdot \tilde{L}^{2} S (kS)^{\beta}\times\Bigg[4\sum_{t=(k-1)S+1}^{kS}\mathcal{O}\left(\frac{\eta_{0}\max\{\delta_{0},M\}}{(t+1)^{\eta+\gamma+1}}\right)+(kS)^{-\beta}\sum_{t=(k-1)S+1}^{kS}\frac{\eta_{0}^{2}}{(t+1)^{2\eta}}\Bigg],
 \end{split}
 \end{equation}
 where in (a), we used \eqref{eq00123}, and in (b) we used 
 \begin{align*}
     \mathbb{E}[F(x_{t})-F(x_{t+1})]&=\delta_{t+1}-\delta_{t}\le \frac{B_{t+1}-B_{t}}{(t+1)^{\gamma}}\le \mathcal{O}(M'\cdot (t+1)^{-\gamma-1}).
 \end{align*}
 from \eqref{D0021} in the proof of \Cref{lemma:rec_eq} and $M'$ is defined in the line above \eqref{D0021} and $M'=\mathcal{O}((2-\alpha)\tau^{\frac{2}{2-\alpha}}(\alpha/2)^{\frac{\alpha}{2-\alpha}})$. Note that if $p\ge1$, then 
 \begin{align}
     \sum_{t=(k-1)S+1}^{kS}\frac{1}{(t+1)^{p}}=\Theta\!\big(S^{\,1-p} k^{-p}\big).
 \end{align}
 Then if $\eta\ge 1/2$, we get
 \begin{align}
     \MoveEqLeft[1]\mathbb{E}\left[\sum_{t=1}^T n_g^t\right]\le \mathcal{O}(\sigma^{2}T^{\beta+1}S^{-1})+ \nonumber\\
          &\sum_{k=1}^{\lceil T/S\rceil}4\cdot \tilde{L}^{2} S (kS)^{\beta}\times\Big[4\eta_{0}M'\Theta(S^{-\gamma-\eta}k^{-\gamma-\eta-1})+(kS)^{-\beta}\eta_{0}^{2}\cdot\Theta(S^{1-2\eta}k^{-2\eta})\Big]\nonumber\\
          &\le \mathcal{O}(\sigma^{2}T^{\beta+1}S^{-1})+ \nonumber\\
          &\sum_{k=1}^{\lceil T/S\rceil}\Big[16\tilde{L}^{2}\eta_{0}M'\Theta(S^{-\gamma-\eta+\beta+1}k^{-\gamma-\eta-1+\beta})+4\tilde{L}^{2}\eta_{0}^{2}\cdot\Theta(S^{2-2\eta}k^{-2\eta})\Big]\nonumber\\
          &\overset{(a)}{\le} \mathcal{O}\Bigg(\sigma^{2}T^{\beta+1}S^{-1}+\tilde{L}^{2}\eta_{0}M'S^{\beta+1-\eta-\gamma}+\tilde{L}^{2}\eta_{0}^{2}S^{2-2\eta}\Bigg),
 \end{align}
 where in (a), we used $\sum_{k=1}^{\lfloor T/S\rfloor}k^{-2\eta}=\mathcal{O}(1)$ and $\sum_{k=1}^{\lfloor T/S\rfloor}k^{-\gamma-\eta-1}=\mathcal{O}(1)$ for $\eta>1/2$.\\
 If $S=\lfloor \frac{T}{q}\rfloor$ where $q$ is a constant integer, the average sample complexity of gradient, would be 
 \[
\mathcal{O}\left(\sigma^{2}T^{\beta}+\tilde{L}^{2}\eta_{0}M'T^{\beta+1-\eta-\gamma}+\tilde{L}^{2}\eta_{0}^{2}T^{2-2\eta}\right)
 \]
 As $T\ge \mathcal{O}(\varepsilon^{-\frac{1}{\gamma}})$, the average sample complexity of gradient should be at least
 \[
\mathcal{O}\left(\sigma^{2}\varepsilon^{-\frac{\beta}{\gamma}}+\tilde{L}^{2}\eta_{0}M'\varepsilon^{-\frac{\beta+1-\eta-\gamma}{\gamma}}+\tilde{L}^{2}\eta_{0}^{2}\varepsilon^{-\frac{2-2\eta}{\gamma}}\right),
 \]
 in order to get $\delta_{T}\le \varepsilon$. Since $\gamma=\min\{\frac{\alpha\beta}{2},\frac{\alpha(1-\eta)}{2-\alpha}\}$, for $1\le\alpha<2$, we set  if $\eta=\frac{\alpha}{2}+(2-\alpha)\mathsf{x}$, $\beta=1-2\mathsf{x}$, and then $\gamma=\frac{\alpha}{2}-\alpha\mathsf{x}$. Then  the average sample complexity of gradient for the case $1<\alpha<2$ should be at least
\begin{align}\label{eq_103103}
    &\mathcal{O}\Bigg(\sigma^{2}\varepsilon^{-\frac{2}{\alpha}}+\tilde{L}^{2}\eta_{0}M'\varepsilon^{-\frac{2}{\alpha}+\frac{2\alpha-2}{\alpha}}+\tilde{L}^{2}\eta_{0}^{2}\varepsilon^{-\frac{2}{\alpha}(2-\alpha)}\Bigg)\nonumber\\
    &{=}\mathcal{O}\left([\sigma^{2}+\tilde{L}^{2}\eta_{0}M'+\tilde{L}^{2}\eta_{0}^{2}]\cdot \varepsilon^{-\frac{2}{\alpha}}\right)\nonumber\\
    &=\mathcal{O}\left([\sigma^{2}+\tilde{L}^{2}\eta_{0}(2-\alpha)\tau^{\frac{2}{2-\alpha}}\left(\frac{\alpha}{2}\right)^{\frac{\alpha}{2-\alpha}}+\tilde{L}^{2}\eta_{0}^{2}]\cdot \varepsilon^{-\frac{2}{\alpha}}\right),
\end{align}
where the last inequality comes from $M'=\mathcal{O}((2-\alpha)\tau^{\frac{2}{2-\alpha}}(\alpha/2)^{\frac{\alpha}{2-\alpha}})$.
 \end{proof}
\subsection{Supplementary lemmas for \Cref{sec:LB_under_local_PL}}

\begin{lemma}\label{lem_A_n}
Let
\[
S_n\;:=\;\sum_{j=1}^{n-1}\Bigg[\prod_{I=j+1}^{n-1}\Big(1-\frac{\eta_I}{\tau}\Big)^{\!2}\Bigg]\eta_j^{2},
\qquad
\eta_k=\frac{\eta_0}{(k+1)^{\eta}},
\]
with parameters $\eta\in(\tfrac12,1)$ and $\tau,\eta_0>0$ (empty products are $1$). Then there exist constants $c,C>0$ and $N\in\mathbb{N}$, depending only on $(\eta,\tau)$, such that for all $n\ge N$,
\[
\eta_0^{2}c\,n^{-\eta}\ \le\ S_n\ \le\ \eta_0^{2}C\,n^{-\eta}.
\]
In particular, $S_n=\Theta(n^{-\eta})$.
\end{lemma}

\begin{proof}
We use $(1-x)\le e^{-x}$ and $(1-x)^2\ge 1-2x$ for $x\in[0,1]$, and for nonnegative $\{y_i\}$ with $\sum_i y_i\le 1$,
$\prod_i(1-y_i)\ge 1-\sum_i y_i$.

\emph{Head–tail split.}
Pick $K$ so large that $\eta_I/\tau\le \tfrac12$ for all $I\ge K$.
Write
\[
S_n=\sum_{j=1}^{K-1}\Big(\cdots\Big)\eta_j^2 \;+\; \sum_{j=K}^{n-1}\Big(\cdots\Big)\eta_j^2
=:S_n^{\mathrm{head}}+S_n^{\mathrm{tail}} .
\]

For fixed $j\le K-1$,
\[
\prod_{I=j+1}^{n-1}\Big(1-\tfrac{\eta_I}{\tau}\Big)^{2}
\;\le\;\exp\!\Big(-\tfrac{2}{\tau}\sum_{I=K}^{n-1}\eta_I\Big)
\;\le\; \exp(-c\,n^{1-\eta}),
\]
since $\sum_{I=K}^{n-1}\eta_I\asymp n^{1-\eta}$ for $\eta\in(0,1)$. Hence
$S_n^{\mathrm{head}}=O(e^{-c\,n^{1-\eta}})=o(n^{-\eta})$.

For $j\ge K$, $(1-x)\le e^{-x}$ gives
\[
\prod_{I=j+1}^{n-1}\Big(1-\tfrac{\eta_I}{\tau}\Big)^{2}
\;\le\;\exp\!\Big(-\tfrac{2}{\tau}\sum_{I=j+1}^{n-1}\eta_I\Big).
\]
Let $r:=n-1-j\in\{0,1,\dots\}$. Monotonicity of $(\eta_k)$ yields
\[
\sum_{I=j+1}^{n-1}\eta_I \;\ge\; r\,\eta_{n-1}
\;=\; r\,\frac{\eta_0}{n^{\eta}},
\]
so with $\lambda:=\tfrac{2\eta_0}{\tau}$,
\[
\prod_{I=j+1}^{n-1}\Big(1-\tfrac{\eta_I}{\tau}\Big)^{2}
\ \le\ e^{-\lambda r/n^{\eta}}.
\]
Thus
\[
S_n^{\mathrm{tail}}
\;\le\; \sum_{r=0}^{n-1} e^{-\lambda r/n^{\eta}}\;\eta_{n-1-r}^{2}.
\]
Split the sum between $r\le n/2$ and $r>n/2$. If $0\le r\le n/2$, then
$\eta_{n-1-r}^{2}\le \eta_0^2/(n/2)^{2\eta}=2^{2\eta}\eta_0^2\,n^{-2\eta}$. Hence
\[
\sum_{r=0}^{\lfloor n/2\rfloor} e^{-\lambda r/n^{\eta}}\eta_{n-1-r}^{2}
\le 2^{2\eta}\eta_0^2\,n^{-2\eta}\sum_{r\ge0} e^{-\lambda r/n^{\eta}}
= 2^{2\eta}\eta_0^2\,n^{-2\eta}\,\frac{1}{1-e^{-\lambda/n^{\eta}}}.
\]
For large $n$, $1-e^{-x}\ge x/2$ when $x=\lambda/n^{\eta}\in(0,1]$, so this is
$\le \dfrac{2^{2\eta+1}\eta_0^2}{\lambda}\,n^{-\eta}= (2^{2\eta}\eta_0\tau)\,n^{-\eta}$.
If $r>n/2$, then $e^{-\lambda r/n^{\eta}}\le e^{-(\lambda/2)n^{1-\eta}}$ and
$\eta_{n-1-r}^{2}\le\eta_0^2$, so the total is
$O\!\big(n\,e^{-(\lambda/2)n^{1-\eta}}\big)=o(n^{-\eta})$.
Therefore $S_n^{\mathrm{tail}}\le C\,n^{-\eta}$ for some $C>0$,
and hence $S_n\le C\,n^{-\eta}$.

To give a lower bound, set
\[
\gamma:=\frac{\tau}{2^{\eta+3}\eta_0}\,,\qquad m:=\big\lfloor \gamma n^{\eta}\big\rfloor .
\]
For large $n$, $1\le m\le n/2$. For any $j\in\{n-m,\dots,n-1\}$ we have
\[
\sum_{I=j+1}^{n-1}\frac{2\eta_I}{\tau} \;\le\; \sum_{I=n-m}^{n-1}\frac{2\eta_I}{\tau}
\;\le\; \frac{2m}{\tau}\,\eta_{n-m}
\;\le\; \frac{2^{\eta+1}\eta_0\gamma}{\tau}
\;=\; \frac12,
\]
using $n-m+1\ge n/2$. Since $(1-x)^2\ge 1-2x$ and $\prod_i(1-y_i)\ge 1-\sum_i y_i$ when the sum is $\le1$,
\[
\prod_{I=j+1}^{n-1}\Big(1-\tfrac{\eta_I}{\tau}\Big)^{2}
\;\ge\; \prod_{I=n-m}^{n-1}\Big(1-\tfrac{2\eta_I}{\tau}\Big)
\;\ge\; 1-\sum_{I=n-m}^{n-1}\tfrac{2\eta_I}{\tau}
\;\ge\; \tfrac12 .
\]
Therefore
\[
S_n\ \ge\ \frac12\sum_{j=n-m}^{n-1}\eta_j^2
\ \ge\ \frac12\,m\,\eta_{n-1}^2
\ \ge\ \Big(\frac{\gamma\eta_0^2}{2}\Big) n^{-\eta}.
\]
Setting $c:=\gamma\eta_0^2/2$ gives the claim.
\end{proof}

\section{Proofs of \Cref{sec_Lower bound for stochastic convex}}\label{append:proofs_sec:4}
\textbf{\Cref{lower_bound_PL_convex_NBS}.}\textit{ For the family of domain sets $\mathbb{S}_{R}$, there exists a function $F$ in the function class $\mathcal{F}^{\mathcal{X}}_{\alpha,\tau,\varepsilon}$, an oracle $O$ in the family of oracles $\mathsf{O}^{G}$, and $\alpha\in(1,2]$ and $\varepsilon\le \min\{((\alpha-1)/{\alpha})^{\alpha}\tau,1\}$, such that the number of oracle queries for any stochastic first-order algorithm $\mathcal{A}\in\mathcal{A}_{m}$ required to output an estimate $\hat{x}$ satisfying $F(\hat{x})-F^{*}\le \varepsilon$ with probability at least $1-\delta$ is
\begin{align}
   \Omega\left(\frac{G^{2}\tau^{\frac{2}{\alpha}}\log\left(\frac{2\alpha R}{(\alpha-1)\varepsilon^{\frac{\alpha-1}{\alpha}}\tau^{\frac{1}{\alpha}}}\right)}{\varepsilon^{\frac{2}{\alpha}}}\right).
\end{align}}
\begin{proof}[Proof of Theorem \ref{lower_bound_PL_convex_NBS}]\label{proof_lower_bound_PL_convex_NBS}
We prove the lower bound by a reduction to the noisy binary search (NBS) problem. Herein, we consider the following: Assume that $N$ sorted elements $\{a_{1},\ldots, a_{N}\}$ are given and we want to insert a new element $u$ using the queries of the form  ``Is $u>a_{j}$?". The oracle answers this query correctly with probability $1/2+p$ for some fixed $p\in[0,1/2)$. Let $j^{*}$ be the unique index such that $a_{j^{*}}\le u<a_{j^{*}+1}$. It is well known (see \cite{feige1994computing,karp2007noisy}) that 
we need at least $\Omega\left(p^{-2}{\log N}\right)$ queries on average in order to identify $j^{*}$. 

\noindent
\textbf{Reduction scheme:} We will construct a stochastic optimization problem with the given parameters ($\bar{L},\tau,\alpha$), such that if there exists an algorithm that solves it (with a constant
probability) after $T$ first-order stochastic queries to the oracle $\mathsf{O}^{G}$, then it can be used to identify $j^{*}$ in NBS problem (with the same probability) using at most $2T$ queries.

First, at each iteration $t$, we define a random variable $Z_{t,j}\in\{-1,1\}$ for every $1\le j\le N$ as follows:
\begin{align}
    \mathbb{P}[Z_{t,j}=1]=\begin{cases}
        \frac{1}{2}+p\quad j> j^{*},\\
        \frac{1}{2}-p\quad j\le j^{*}.
    \end{cases}
\end{align}
$Z_{t,j}$ is the answer of the NBS oracle to query ``Is $u>a_{j}$?" at the iteration $t$.

\noindent
In the reduction scheme, we assume that function $F$ has a one-dimensional domain $\mathcal{X}$. The diameter of this domain is $\sup_{x,y\in \mathcal{X}}|x-y|=R$, and without loss of generality, we assume that $\mathcal{X}=[0,R]$. We first divide the interval $[0,R]$ into $N$ equal sub-intervals of length $R/N$ each, and consider the element $a_j$ as the smallest point in the $j$-th interval.

\textbf{NBS oracle:} At each iteration, NBS oracle is queried at a point $x\in \mathcal{X}$ and its response is $(Z_{t,j}, Z_{t,j+1})$, for $x\in[a_{j},a_{j+1})$.

\textbf{Stochastic first-order oracle:}
Using the noisy binary pairs $(Z_{t,j}, Z_{t,j+1})$ from NBS oracle queried at $x\in[a_{j},a_{j+1})$, the output of this oracle at point $x$ is constructed as follows:
\begin{align}\label{def_grad_est}
    \MoveEqLeft[4]f'(x,Z_{t,j},Z_{t,j+1})\nonumber\\
    &=\frac{G}{2}\left(1-g_{j}(x)\right)Z_{t,j}+\frac{G}{2}\left(1+g_{j}(x)\right)Z_{t,j+1},
\end{align}
where $G>0$ is some constant and 
\begin{align}\label{def_g(x)}
    g_{j}(x)=\frac{\left|x-\frac{R}{2N}-a_{j}\right|^{\frac{1}{\alpha-1}}\cdot\text{sgn}\left(x-\frac{R}{2N}-a_{j}\right)}{\left(\frac{R}{2N}\right)^{\frac{1}{\alpha-1}}}\1_{[a_{j},a_{j+1})}(x).
\end{align}
Note that $\mathbb{E}[f'(x,Z_{t,j},Z_{t,j+1})]=F'(x)$ and 
\[
|f'(x,Z_{t,j},Z_{t,j+1})|=\begin{cases}
    G\quad &\text{if }Z_{t,j}=Z_{t,j+1},\\
    G|g_{j}(x)|\quad &\text{if }Z_{t,j}\neq Z_{t,j+1}.
\end{cases}
\]
Hence, $|f'(x,Z_{t,j},Z_{t,j+1})|\le G$.
Taking expectation of $f'(x,Z_{t,j},Z_{t,j+1})$, we obtain
\begin{align}\label{def_grad}
    \MoveEqLeft[4]F'(x)=\mathbb{E}[f'(x,Z_{t,j},Z_{t,j+1})]=
        pG\,\1_{[a_{j^{*}+1}, R]}(x)\nonumber\\
        &
        -pG\,\1_{[0, a_{j^{*}})}(x)+
        pG  g_{j^{*}}(x)\1_{[a_{j^{*}},a_{j^{*}+1})}(x).
\end{align}
Integrating $F'(x)$ with respect to $x$, we get
\begin{align}\label{def_true function}
   F(x)
&= pG\,(x-a_{j^{*}+1})\,\1_{[a_{j^{*}+1},\,R]}(x)\nonumber\\
&+ pG\,(-x+a_{j^{*}})\,\1_{[0,\,a_{j^{*}})}(x)
\nonumber\\
&+ \Big[
  pG\,\frac{\alpha-1}{\alpha}\frac{\bigl|x-\frac{R}{2N}-a_{j^{*}}\bigr|^{\frac{\alpha}{\alpha-1}}}{\left(\frac{R}{2N}\right)^{\frac{1}{\alpha-1}}}
  \nonumber\\
  &\qquad\qquad\qquad\qquad- pG\,\frac{\alpha-1}{2\alpha}\frac{R}{N}
\Big]\,\1_{[a_{j^{*}},\,a_{j^{*}+1})}(x).
\end{align}
Note that by construction, $\min_{x\in \mathcal{X}}F(x)=-pG (\alpha-1){R}/(2\alpha N)$ and $a_{j^{*}}+{R}/(2N)=\argmin_{x\in \mathcal{X}}F(x)$. Moreover, function $F$ given by \eqref{def_true function} is convex and its domain is bounded ($\mathcal{X}=[0,R]$). 
In \Cref{lemma_alpha_PL}, we show that if 
\begin{align}\label{cond_1_cvx}
    \tau\ge \frac{\alpha-1}{\alpha}\frac{R}{2N}(pG)^{1-\alpha},
\end{align}
then $F$ satisfies the local $(\alpha,\tau,R/N)$-gradient-dominance (Assumption \ref{assum_local_PL_alpha}). 
In our reduction, we need to show that if the output of a stochastic first-order method $\hat x$ satisfies $F(\hat{x})-F^{*}\le \varepsilon$, then $j^{*}$ is identified (in other words, $\hat{x}\in[a_{j^{*}},a_{j^{*}+1})$). If  
\begin{align}\label{cond_3}
    pG \frac{\alpha-1}{2\alpha}\frac{R}{N}\ge2\varepsilon,
\end{align}
we get $F(x)-F^{*}>\varepsilon$ for every $x\notin [a_{j^{*}},a_{j^{*}+1})$. Indeed from the definition of the function \eqref{def_true function}, for every $x\notin [a_{j^{*}},a_{j^{*}+1})$, we have
\[
F(x)-F^{*}\ge pG \frac{\alpha-1}{2\alpha}\frac{R}{N}
\]
and if $pG (\alpha-1)/(2\alpha)R/N\ge2\varepsilon$, we get $F(x)-F^{*}>\varepsilon$. 

\noindent
We pick 
\begin{align}\label{eq_chsoen_p_N}
    p=\frac{2\varepsilon^{{1}/{\alpha}}}{G\tau^{1/{\alpha}}},\quad N=\frac{(\alpha-1)R}{(2\alpha)\varepsilon^{{(\alpha-1)}/{\alpha}}\tau^{{1}/{\alpha}}}.
\end{align} 
Subsequently, with these chosen values for $p$ and $N$, the inequalities \eqref{cond_1_cvx} and \eqref{cond_3} are met for every $\varepsilon\le 1$. For $\varepsilon\le ({(\alpha-1)}/{\alpha})^{\alpha}\tau$, we have: ${R}/{N}={2\alpha\varepsilon^{{(\alpha-1)}/{\alpha}}\tau^{{1}/{\alpha}}}{(\alpha-1)^{-1}}\ge \varepsilon$, and therefore, every local $(\alpha,\tau,R/N)$-gradient-dominated function is also a local $(\alpha,\tau,\varepsilon)$-gradient-dominated function. Consequently, $\mathcal{F}^{\mathcal{X}}_{\alpha,\tau,{R}/{N}}\subseteq \mathcal{F}^{\mathcal{X}}_{\alpha,\tau,\varepsilon}$, and as a result, $F\in \mathcal{F}^{\mathcal{X}}_{\alpha,\tau,\varepsilon}$.
\newline
Thus, for $(\mathcal{F}^{\mathcal{X}}_{\alpha,\tau,\varepsilon},\mathsf{O}^{G})$, any stochastic first-order algorithm that converges to an $\varepsilon$-minimizer can be used to identify $j^{*}$ in a NBS problem for appropriately chosen  $p$ and $N$ as in \eqref{eq_chsoen_p_N}. Therefore, the probability-based minimax oracle complexity $\Ts_{\varepsilon}(\mathcal{F}^{\mathcal{X}}_{\alpha,\tau,\varepsilon},\mathsf{O}^{G})$ can be lower bounded by $\Omega\left({p^{-2}}{\log N}\right)$. For every $\varepsilon\le \min\{({(\alpha-1)}/{\alpha})^{\alpha}\tau,1\}$,
\[
\Ts_{\varepsilon}(\mathcal{F}^{\mathcal{X}}_{\alpha,\tau,\varepsilon},\mathsf{O}^{G})=\Omega\left(\frac{G^{2}\tau^{\frac{2}{\alpha}}\log\left(\frac{(\alpha-1) R}{2\alpha\varepsilon^{\frac{\alpha-1}{\alpha}}\tau^{\frac{1}{\alpha}}}\right)}{\varepsilon^{\frac{2}{\alpha}}}\right).
\]
\end{proof}
\begin{lemma}\label{lemma_alpha_PL}
   Function $F$ defined in \eqref{def_true function} satisfies $(\alpha,\tau,R/N)$-gradient-dominance for $\tau\ge (\alpha-1){R}(pG)^{1-\alpha}/(2\alpha N)$.
\end{lemma}
\begin{proof}\label{proof_lemma_alpha_PL}
For $x\in [a_{j^{*}},a_{j^{*}+1})$, $F(x)-\min_{x\in \mathcal{X}}F(x)\le \tau|F'(x)|^{\alpha}$ is equivalent to have
\begin{align}
    &F(x)-\min_{x\in \mathcal{X}}F(x)=pG  \frac{\alpha-1}{\alpha}\frac{|x-\frac{R}{2N}-a_{j^{*}}|^{\frac{\alpha}{\alpha-1}}}{\left(\frac{R}{2N}\right)^{\frac{1}{\alpha-1}}}\nonumber\\
    &\qquad\qquad\le \tau\left|pG\frac{|x-\frac{R}{2N}-a_{j}|^{\frac{1}{\alpha-1}}\cdot\text{sgn}(x-\frac{R}{2N}-a_{j})}{\left(\frac{R}{2N}\right)^{\frac{1}{\alpha-1}}}\right|^{\alpha}.
\end{align}
If $\tau\ge (\alpha-1){R}(pG)^{1-\alpha}/(2\alpha N)$, we get $F(x)-\min_{x\in \mathcal{X}}F(x)\le \tau|F'(x)|^{\alpha}$.
\end{proof}
\begin{remark}\label{rmrk_lower_bound_kl_convex}
    In Appendix \ref{append_lower_bound_kl_convex}, we consider the $\phi$-Kurdyka-\L ojasiewicz (KL) inequality \citep{yang2018rsg} and recall the definition of function $\phi$ and $\phi$-KL inequality. For the class of convex functions satisfying the $\phi$-KL inequality with oracle $\mathsf{O}^{G}$ and domain sets $\mathbb{S}_{R}$, we derive the lower bound $\Omega\left(G^{2}(\phi'(\varepsilon))^{2}\log\left({R}/(2\phi(\varepsilon))\right)\right)$. In this setting, the upper bound $T=\mathcal{O}\left({G^{2}(\phi(\varepsilon))^{2}\log(1/\varepsilon)}/{\varepsilon^{2}}\right)$ from \cite[Corollary 14]{yang2018rsg} is larger than our lower bound by a multiplicative factor of $\mathcal{O}\left(\left({\phi(\varepsilon)}/{(\varepsilon\phi'(\varepsilon))}\right)^{2}\cdot {\log(1/\varepsilon)}/{\log(R/2\phi(\varepsilon))}\right)$. It is noteworthy that this factor becomes a constant for $\phi(s)=C\cdot s^{1-{1}/{\alpha}}$ for $\alpha>1$ and some constant $C>0$. It would be interesting to characterize the minimax oracle complexity of first-order methods for achieving a global-optimum point of a convex bounded domain function that satisfies $\phi$-KL inequality for  other choices of function $\phi$.  
    
\end{remark}
\subsection{Proof of Remark \ref{rmrk_lower_bound_kl_convex}}\label{append_lower_bound_kl_convex}

\begin{assumption}\label{assum_kl}
 Consider a continuous concave function $\phi:[0, \zeta) \to\mathbb{R}^{+}$ such that (i) $\phi(0) = 0$; (ii) $\phi$ is continuous on $(0, \zeta)$; (iii) and for all $s \in (0, \zeta)$, $\phi'(s)>0$. Function $f(x)$ satisfies the $\phi$-Kurdyka-\L ojasiewicz ($\phi$-KL) property at $\bar{x}$ if there exist $\zeta\in (0, \infty]$, a neighborhood $U_{\bar{x}}$ of $\bar{x}$ and for all $x\in U_{\bar{x}} \cap \{x : f(\bar{x}) < f(x) < f(\bar{x}) + \zeta\}$, the following inequality holds
 \begin{align}
     \phi'(f(x)-f(\bar{x}))\cdot \|\partial f(x)\|_{2} \ge 1,
 \end{align}
 where $\|\partial f(x)\|_{2} := \min_{\gv\in\partial f(x)} \|\gv\|_{2}$.
\end{assumption}
\textbf{Stochastic first-order oracle:}
Using the noisy binary pairs $(Z_{t,j}, Z_{t,j+1})$ from NBS oracle which is queried at $x\in[a_{j},a_{j+1})$, the output of this oracle at point $x$ is constructed as follows:
\begin{align}\label{def_grad_est_kl}
    f'(x,Z_{t,j},Z_{t,j+1})&=\frac{G}{2}\left(1-g_{j}(x)\right)Z_{t,j}\nonumber\\
    &\qquad\qquad+\frac{G}{2}\left(1+g_{j}(x)\right)Z_{t,j+1},
\end{align}
where $G$ is some constant and 
\begin{align}\label{def_g(x)_kl}
    \MoveEqLeft[1]g_{j}(x)=\nonumber\\
    &\frac{\psi'(|x-\frac{R}{2N}-a_{j}|)\cdot\text{sgn}(x-\frac{R}{2N}-a_{j})}{\psi'(\frac{R}{2N})}\cdot\1_{[a_{j},a_{j+1})}(x),
\end{align}
where $\psi\equiv \phi^{-1}$ and then $\psi:[0,\infty)\to [0,\infty)$ is a continuous convex function such that $\psi(0)=0$, $\psi'(x)>0$ for $x\in\mathbb{R}^{+}$. Note that
\[
|f'(x,Z_{t,j},Z_{t,j+1})|=\begin{cases}
    G\quad &\text{if }Z_{t,j}=Z_{t,j+1},\\
    G|g_{j}(x)|\quad &\text{if }Z_{t,j}\neq Z_{t,j+1}.
\end{cases}
\]
Hence, $|f'(x,Z_{t,j},Z_{t,j+1})|\le G$.
Taking expectation of $f'(x,Z_{t,j},Z_{t,j+1})$, we obtain
\begin{align}\label{def_grad_kl}
    \MoveEqLeft[4]F'(x)=\mathbb{E}[f'(x,Z_{t,j},Z_{t,j+1})]\nonumber\\
&= pG\,\1_{[a_{j^{*}+1},\,R]}(x)
\nonumber\\
&- pG\,\1_{[0,\,a_{j^{*}})}(x)
+ pG\,g_{j^{*}}(x)\,\1_{[a_{j^{*}},\,a_{j^{*}+1})}(x).
\end{align}
Integrating $F'(x)$, we have
\begin{align}\label{def_true function_kl}
    F(x)
&= pG\,(x-a_{j^{*}+1})\,\1_{[a_{j^{*}+1},\,R]}(x)\nonumber\\
&+ pG\,(-x+a_{j^{*}})\,\1_{[0,\,a_{j^{*}})}(x)\nonumber\\
&+ \left[
  pG\,\frac{\psi\!\left(\left|x-\frac{B}{2N}-a_{j^{*}}\right|\right)}{\psi'\!\left(\frac{R}{2N}\right)}-pG\,\frac{\psi\!\left(\frac{R}{2N}\right)}{\psi'\!\left(\frac{R}{2N}\right)}
\right] \1_{[a_{j^{*}},\,a_{j^{*}+1})}(x)
\end{align}
Note that by construction, $\min_{x\in \mathcal{X}}F(x)=pG {\psi({R}/{2N})}/{\psi'({R}/{2N})}$ and $a_{j^{*}}+{R}/{(2N)}=\argmin_{x\in \mathcal{X}}F(x)$. Function $F$ is convex and its domain is bounded ($\mathcal{X}=[0,R]$). From Lemma \ref{lemma_kl}, if 
\begin{align}\label{cond_1_cvxl}
    pG\ge\psi'(R/2N),
\end{align}
then $F$ satisfies $\phi$-KL property (Assumption \ref{assum_kl}) in the interval $U_{a_{j^{*}+R/2N}}=[a_{j^{*}},a_{j^{*}+1})$.
In the reduction, we need to show that if the output of a stochastic first-order method $\hat x$ satisfies $F(\hat{x})-F^{*}\le \varepsilon$, then $j^{*}$ is identified (more precisely, $\hat{x}\in[a_{j^{*}},a_{j^{*}+1})$). If  
\begin{align}\label{cond_33}
    pG\frac{\psi(R/2N)}{\psi'(R/2N)}\ge\varepsilon,
\end{align}
for every $x\notin [a_{j^{*}},a_{j^{*}+1})$, we get $F(x)-F^{*}\ge\varepsilon$. Indeed from the definition of the function \eqref{def_true function_kl}, for every $x\notin [a_{j^{*}},a_{j^{*}+1})$, we have
\[
F(x)-F^{*}\ge pG\frac{\psi(R/2N)}{\psi'(R/2N)},
\]
and if $pG{\psi(R/2N)}(\psi'(R/2N))^{-1}>\varepsilon$, we get $F(x)-F^{*}>\varepsilon$.

Let $p=(G\phi'(\varepsilon))^{-1}$ and $N={R}(2\phi(\varepsilon))^{-1}$. Then both conditions \eqref{cond_1_cvxl} and \eqref{cond_33} hold with equality. Therefore, the minimax oracle complexity in this case, can be lower bounded by $\Omega\left(p^{2}\log N\right)$ which is
\begin{align}
    \Omega\left(G^{2}(\phi'(\varepsilon))^{2}\log\left(\frac{R}{2\phi(\varepsilon)}\right)\right).
\end{align}

\begin{lemma}\label{lemma_kl}
    Function $F$ defined in \eqref{def_true function_kl} satisfies $\phi$-KL property when $pG\ge \psi'(B/2N)$.
\end{lemma}
\begin{proof}
By using $\phi\equiv\psi^{-1}$ and the condition $pG\ge \psi'(R/2N)$, we have
\begin{align}
    &\phi'(F(x)-\min_{x\in \mathcal{X}}F(x))=(\psi^{-1})'(F(x)-\min_{x\in \mathcal{X}}F(x))\nonumber\\
    &=(\psi^{-1})'\left( \psi(|x-\frac{R}{2N}-a_{j}|)\right)=\frac{1}{\psi'(|x-\frac{R}{2N}-a_{j}|)}\nonumber\\
    &\ge\frac{\psi'(\frac{R}{2N})}{pG\cdot \psi'(|x-\frac{R}{2N}-a_{j}|)}=\frac{1}{|F'(x)|}
\end{align}
\end{proof}

\section{Comparison between Theorem \ref{lower_bound_PL_convex_NBS} and  \citep{foster2019complexity}}\label{append_comparison_relate_work}

Regarding Theorem \ref{lower_bound_PL_convex_NBS}, we used a similar approach (reduction to NBS problem) as \citep{foster2019complexity}. \cite{foster2019complexity} uses the reduction to NBS problem in order to derive a complexity lower bound for stochastic first-order methods converging to the approximate first-order stationary point in expectation $\mathbb{E}[\|\nabla F(\hat{x})\|]\le \varepsilon$ over the class of convex smooth functions. The differences of our lower bound proof with \Cref{lower_bound_PL_convex_NBS} are:
\begin{itemize}[leftmargin=*]
    \item \cite{foster2019complexity} derived their lower bound to find the average first-order stationary point while we are using this approach to derive the lower bound to find the approximate minimizer in average, i.e.,  $\mathbb{E}[F(\hat{x})]-F^{*}\le \varepsilon$. For convex objective functions, the complexity of finding approximate stationary points is different from the complexity of finding approximate minimizers. For example, \citep{foster2019complexity} showed that while SGD is (worst-case) optimal for stochastic convex optimization for finding approximate minimizer, it appears to be far from optimal for finding near-stationary points (a version of SGD3 \citep{allen2018make} is optimal in this case).
    \item The gradient estimator in \citep{foster2019complexity} is 
\( 
f'(x,Z_{t,j},Z_{t,j+1})
= -2\varepsilon\,\1_{(-\infty,0)}(x)
+ 2\varepsilon\,\1_{[R,\infty)}(x)
+ \sum_{j=1}^{N-1} \1_{[a_j,a_{j+1})}(x)\big(h_j(x)Z_{t,j+1}+(1-h_j(x))Z_{t,j}\big).
\)
    where $h_{j}:=(x-a_{j})(R/N)^{-1}$. A naive approach to extend their construction to the case where the function satisfies local $(\alpha,\tau,\varepsilon)$-P\L\ property (Assumption \ref{assum_local_PL_alpha}) is the straightforward replacement of $h_{j}(x)$ with ${|x-a_{j}|^{{1}/(\alpha-1)}\text{sgn}(x-a_{j})}{\left({R}/{n}\right)^{-{1}/(\alpha-1)}}.$
    A drawback of this construction is that if the minimum of $f(x)$ is close to $a_{j^{*}}$ and when the approximate minimizer of the function is in $[a_{j^{*}-1},a_{j^{*}})$, then $[a_{j^{*}},a_{j^{*}+1})$ is not identified and the reduction to NBS problem does not work. 
    The solution is to use a version of $f'(x,Z_{t,j},Z_{t,j+1})$ in \eqref{def_grad_est} which has the following two properties: 1) the function satisfying local $(\alpha,\tau,\varepsilon)$-P\L\, 2) Finding the approximate minimizer of this function uniquely identifies the interval $[a_{j^{*}},a_{j^{*}+1})$.
\end{itemize}
\section{Experiment details: ReLU and Dictionary learning}\label{append_experiment}
This appendix details the problem and protocols used in our two empirical studies in \Cref{sec:experiment}:
(1) binary classification with a ReLU network, and
(2) dictionary learning.\\

Our theory shows that the optimal local oracle complexity takes the form
\[
\mathrm{poly}\!\big(\tau(\alpha)\big)\,\varepsilon^{-2/\alpha},
\]
where $\tau(\alpha)$ is the local $\alpha$–P{\L} constant around the target component $\mathcal M$ and
$\varepsilon$ is the target accuracy. Thus, larger $\alpha$ improves the $\varepsilon$–exponent
($2/\alpha$ decreases) but may increase $\tau(\alpha)$; the best choice balances these two effects.

Empirically, we estimate
\[
\tau(\alpha)\ :=\ \sup_{x\in \mathbb B_2(\hat x;r)}
\frac{F(x)-F(\hat x)}{\|\nabla F(x)\|^{\alpha}},
\]
and typically observe that $\tau(\alpha)$ is finite for all $\alpha\in[1,2]$ and fairly flat near $\alpha=1$. (i) For very small $\varepsilon$, the term $\varepsilon^{-2/\alpha}$ dominates, so taking $\alpha$ closer to $2$
is advantageous. (ii) For moderate $\varepsilon$, a flatter (often smaller) $\tau(\alpha)$ near $\alpha=1$ can yield the smaller overall bound. This explains why we care about $\alpha\in[1,2]$ and motivates tuning $\alpha$ to minimize the predicted cost.

Let $\hat{x}$ be the terminal iterate returned by full-batch gradient descent (FB-GD). In all reported runs we have $\|\nabla F(\hat{x})\|\le10^{-5}$, so $\hat{x}$ is an \emph{almost critical point}.
Fix a radius $r>0$ and consider the ball $\mathbb{B}_2(\hat{x};r) := \{x:\|x-\hat{x}\|_2\le r\}$ in parameter space.
For $\alpha\in[1,2]$, define
\[
\mathsf{R}_\alpha(x):=\frac{F(x)-F(\hat{x})}{\|\nabla F(x)\|^\alpha},
\qquad
\tau(\alpha)=\sup_{x\in\mathbb{B}_2(\hat{x};r)}\mathsf{R}_\alpha(x),
\]
	so that $F(x)-F(\hat{x})\le \tau(\alpha)\,\|\nabla F(x)\|^\alpha$ on $\mathbb{B}_2(\hat{x};r)$.
	Empirically, we approximate the supremum by sampling $m$ points $\{x_i\}_{i=1}^m\subset\mathbb{B}_2(\hat{x};r)$ (uniform in direction, radius drawn with density proportional to $r^{d-1}$) and compute
	\[
	\widehat{\tau}(\alpha)=\max_{1\le i\le m}\frac{F(x_i)-F(\hat{x})}{\|\nabla F(x_i)\|^\alpha}.
	\]
		In \Cref{fig:binary_append} we use $m=10^5$ probes for each radius $r\in\{0.5,0.05,0.005\}$; in \Cref{fig:realdata_relu_append} we use $(r,m)=(0.35,450)$ for MNIST and $(r,m)=(0.40,600)$ for WDBC.
		In the plots, we also report the baseline $10\widehat{\tau}(1)$ and highlight in green the set
		$\{\alpha\in[1,2]:\widehat{\tau}(\alpha)\le 10\widehat{\tau}(1)\}$, which visualizes the near-flat regime around $\alpha=1$.

\subsection{Binary classification with a one-hidden-layer ReLU network}
\label{app:relu-setup}

\paragraph{Data.}
Two-dimensional synthetic binary classification with four Gaussian blobs (means at $(\pm 1.5,\pm 1.5)$, isotropic covariance $0.35I$); labels form an XOR pattern.

	\paragraph{Model and loss.}
	A $1$-hidden-layer ReLU network with $m$ hidden units:
	$x\mapsto a=\mathrm{ReLU}(xW_1+b_1)\mapsto \ell=a^\top W_2+b_2\mapsto p=\sigma(\ell)$.
	Loss $F$ is the average binary cross-entropy, optionally with $\ell_2$ weight decay $\frac{\lambda}{2}(\|W_1\|_F^2+\|W_2\|_2^2)$.

\paragraph{Training.}
FB-GD with fixed step size; stop either at a gradient-norm threshold or a max-iteration budget; set $\hat{x}$ to the final iterate.

\paragraph{Neighborhood sampling and estimation.}
Form $\mathbb{B}_2(\hat{x};R)$ in \emph{parameter space} and draw $m$ probing parameters.
Compute gaps $\Delta_i:=F(x_i)-F(\hat{x})$ and gradient norms $g_i:=\|\nabla F(x_i)\|$.
Evaluate $\widehat{\tau}(\alpha)$ on a grid $\alpha\in[1,2]$, plot $\widehat{\tau}(\alpha)$ together with $10\widehat{\tau}(1)$, and highlight the set $\{\alpha:\widehat{\tau}(\alpha)\le 10\widehat{\tau}(1)\}$.\\

\Cref{fig:binary_append} reports the curves
\(
\widehat{\tau}(\alpha)
\)
	for radii \(r\in\{0.5,\,0.05,\,0.005\}\). By definition, \(\widehat{\tau}(\alpha)\) is a (sampled) supremum over the ball \(B_2(x^\star;r)\), hence it is \emph{monotone in \(r\)}. 
	The radius \(r\) also selects the \emph{local geometry} being used. Smaller \(r\) isolates the immediate neighborhood of \(x^\star\), where the landscape is closer to a single regime and may admit a \emph{smaller constant} and a \emph{wider near-flat regime} around $\alpha=1$.
	As \(r\downarrow 0\), the curve reflects the most local behavior around \(x^\star\) and the set $\{\alpha:\widehat{\tau}(\alpha)\le 10\widehat{\tau}(1)\}$ typically expands.

	\begin{figure}
	    \centering
	    \includegraphics[width=1\linewidth]{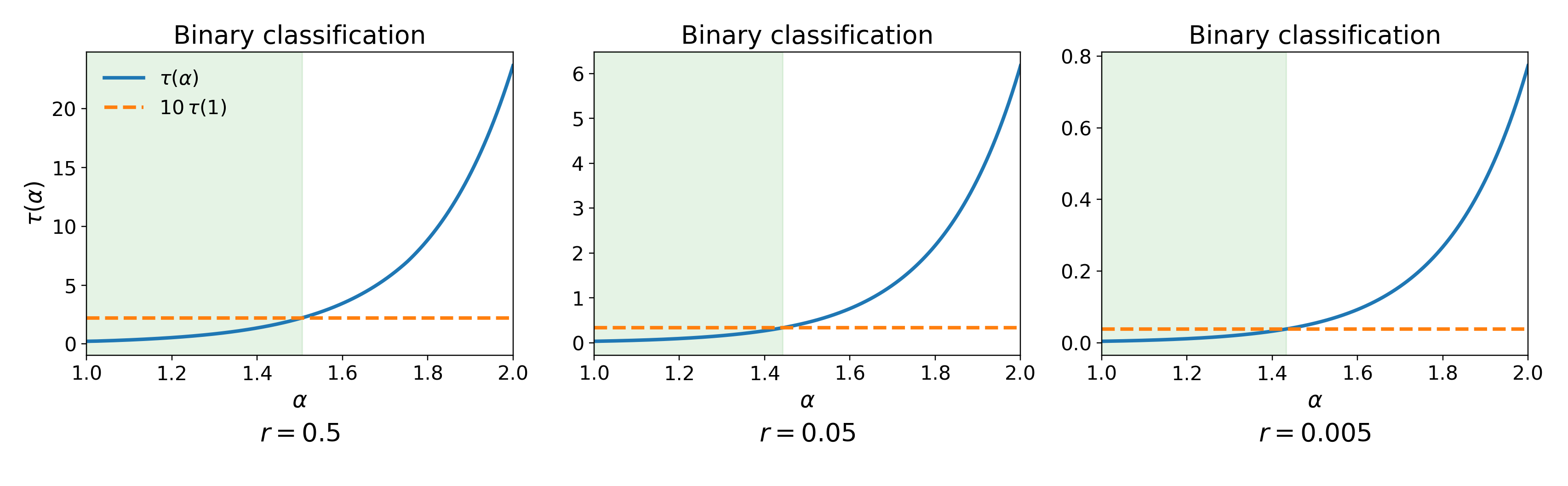}
	    \caption{\textbf{Synthetic ReLU classification (XOR).} Weight decay $\lambda=0$. Variation of $\tau(\alpha)$ over $\alpha\in[1,2]$ for three values of $r$. The dashed line is $10\tau(1)$ and the green region is $\{\alpha:\tau(\alpha)\le 10\tau(1)\}$. Ranges: $r=0.5$, $\tau(\alpha):[1,2]\to$ $[0.221,\,23.67]$; $r=0.05$, $\tau(\alpha):[1,2]\to$ $[0.0336,\,6.17]$; $r=0.005$, $\tau(\alpha):[1,2]\to$ $[0.00382,\,0.773]$.}
	    \label{fig:binary_append}
	\end{figure}

		\paragraph{Real-data binary classification.}
		We additionally repeat the same procedure on two standard real datasets: MNIST (digits $0$ vs.\ $1$) and the UCI Breast Cancer Wisconsin Diagnostic dataset (WDBC; malignant vs.\ benign).
		For both datasets, we fit the same one-hidden-layer ReLU network by full-batch gradient descent to obtain an almost critical point $\hat x$, then probe a small ball $\mathbb{B}_2(\hat x;r)$ in parameter space by uniform sampling and compute $\widehat{\tau}(\alpha)$ over $\alpha\in[1,2]$.
		\Cref{fig:realdata_relu_append} shows the resulting envelopes.

		\begin{figure}[h!]
		    \centering
		    \includegraphics[width=1\linewidth]{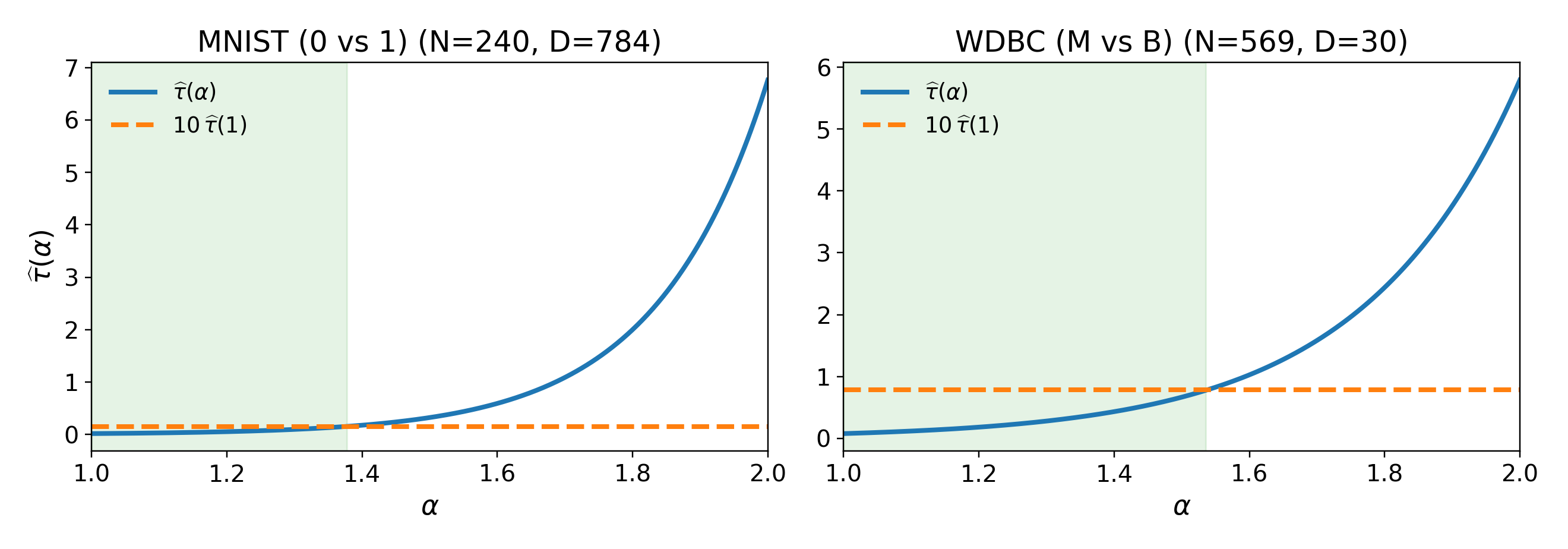}
		    \caption{\textbf{Real-data ReLU classification.} Weight decay $\lambda=10^{-4}$. Estimated envelopes $\alpha\mapsto \widehat{\tau}(\alpha)$ for MNIST ($0$ vs.\ $1$) and WDBC (malignant vs.\ benign). The dashed line is $10\widehat{\tau}(1)$ and the green region is $\{\alpha:\widehat{\tau}(\alpha)\le 10\widehat{\tau}(1)\}$. Ranges of $\widehat{\tau}(\alpha)$ over $\alpha\in[1,2]$: MNIST $[0.015,\,6.76]$, WDBC $[0.078,\,5.79]$.}
		    \label{fig:realdata_relu_append}
		\end{figure}

\subsection{Dictionary learning with smooth sparsity}
\label{app:dict-setup}

\paragraph{Data.}
Generate a ground-truth dictionary $D_0\in\mathbb{R}^{p\times k}$ with unit-norm columns; draw sparse codes $A_0\in\mathbb{R}^{k\times n}$ with $s$ nonzeros per column; form $X=D_0A_0+\xi$, $\xi\sim\mathcal{N}(0,\sigma^2I)$.

\paragraph{Objective.}
We minimize
\begin{align}
    F(D,A)&=\frac{1}{2n}\|DA-X\|_F^2
\;+\;\lambda_1\!\sum_{ij}\phi_\delta(A_{ij})+\;\frac{\lambda_2}{2}\|A\|_F^2
\;+\;\frac{\mu}{4}\sum_{j=1}^k\big(\|d_j\|_2^2-1\big)^2,
\end{align}
where $\phi_\delta$ is the pseudo-Huber penalty
$\phi_\delta(a)=\delta^2(\sqrt{1+(a/\delta)^2}-1)$
with derivative $\phi'_\delta(a)=a/\sqrt{1+(a/\delta)^2}$.
The last term softly enforces unit-norm atoms and removes the $D\!-\!A$ scaling ambiguity.

\paragraph{Training.}
FB-GD with backtracking (Armijo) on $(D,A)$ until a small gradient norm is reached; set $\hat{x}=(\hat{D},\hat{A})$.

		\paragraph{Neighborhood sampling and estimation.}
		Sample $\mathbb{B}_2(\hat{x};R)$ in the \emph{joint} $(D,A)$ parameter space by perturbing $\hat{D}$ and $\hat{A}$ together.
		In the main-body dictionary learning experiment (\Cref{fig:placeholder}, bottom) we use $R=0.1$ and $m=2500$ probe points.
		Compute $\Delta_i$ and $g_i$ as above and evaluate $\widehat{\tau}(\alpha)$ on $\alpha\in[1,2]$.



\section{Additional Experiments: Polynomial Matrix Factorization and Deep Linear Factorization}
\label{app:mf_experiments}

In both experiments we validate \Cref{assum_PL_alpha} (local $\alpha$–P\L{}) by estimating, on a bounded neighborhood of an almost–critical point,
\[
R_\alpha(\theta)\;:=\;\frac{F(\theta)-F(\theta^\star)}{\|\nabla F(\theta)\|^\alpha},
\qquad
	\tau(\alpha)\;:=\;\sup_{\theta\in B_2(\theta^\star;r)} R_\alpha(\theta),
	\quad \alpha\in[1,2],
	\]
	where $\theta^\star$ is the terminal iterate of full-batch gradient descent (FB-GD), hence an \emph{almost} critical point ($\|\nabla F(\theta^\star)\|$ small).
	In the plots below, we visualize the envelope $\alpha\mapsto\tau(\alpha)$ and highlight in green the range $\{\alpha:\tau(\alpha)\le 10\tau(1)\}$ to make the ``flat'' regime around $\alpha=1$ directly comparable across setups.
	\subsection{Polynomial Matrix Factorization}
\paragraph{Problem.}
Given $X\in\mathbb{R}^{p\times n}$ and a target rank $r$, factorize $X\approx U V$ with
\[
F(U,V)\;=\;\frac{1}{2n}\,\|UV - X\|_F^2\;+\;\frac{\lambda}{2}\bigl(\|U\|_F^2+\|V\|_F^2\bigr),
\quad U\in\mathbb{R}^{p\times r},\;V\in\mathbb{R}^{r\times n}.
\]
This is a \emph{polynomial} loss, thus fits our tame/definable setting.

	\paragraph{Synthetic setup.}
	We generate $X=U_0V_0+\xi$ with $U_0\in\mathbb{R}^{p\times r}$, $V_0\in\mathbb{R}^{r\times n}$ i.i.d.\ Gaussian and $\xi\sim\mathcal{N}(0,\sigma^2 I)$; optionally, we inject a small fraction of \emph{spikes} (large entries) to test robustness.\footnote{We use a spike fraction $s$ and a multiplier $\kappa$ relative to $\operatorname{std}(X)$; both are reported with the plots.}
	We run FB-GD with backtracking to obtain $\theta^\star=(U^\star,V^\star)$ and estimate $\widehat{\tau}(\alpha)$ by probing a neighborhood around $\theta^\star$ (radius $R=0.05$, $m=1500$ probes).
	\begin{figure}[h!]
	  \centering
	  \includegraphics[width=.6\linewidth]{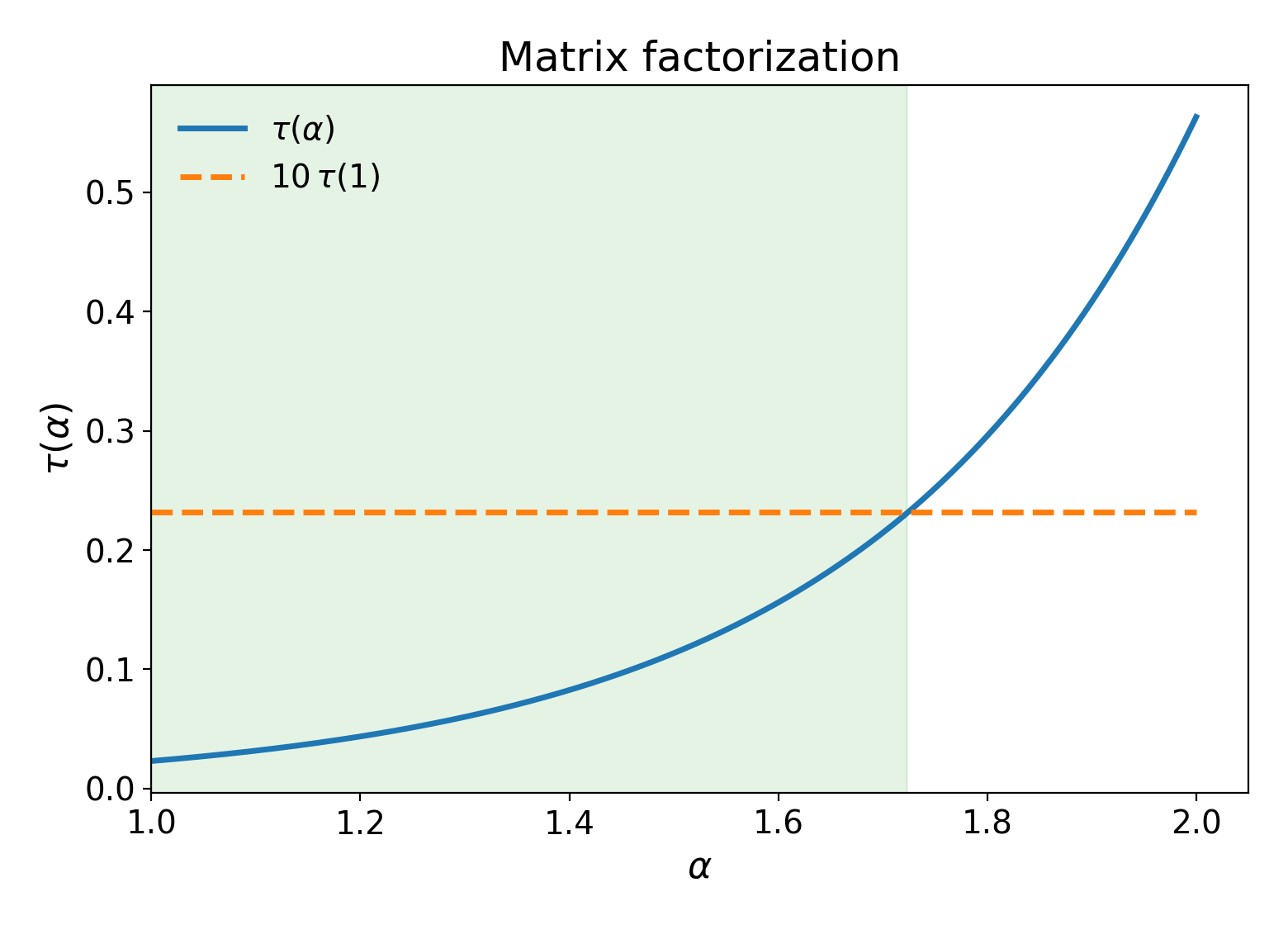}
	  \caption{\textbf{Synthetic polynomial MF.} $\lambda=10^{-3}$, $R=0.05$, $m=1500$. Variation of $\widehat{\tau}(\alpha)$ for $\alpha\in[1,2]$ (synthetic $X=U_0V_0+\xi$). Range of $\widehat{\tau}(\alpha)$ over $\alpha\in[1,2]$: $[0.023,\,0.563]$.}
	  \label{fig:synthetic_mf_append}
		\end{figure}
		\paragraph{Explanation.}
		In \Cref{fig:synthetic_mf_append} we plot the empirical envelope $\widehat{\tau}(\alpha)=\max_i \Delta_i/g_i^\alpha$ on a fixed-radius neighborhood around the terminal iterate $\theta^\star$.
		The curve is relatively flat for $\alpha$ close to $1$ and then increases as $\alpha$ approaches $2$, indicating that (at this radius) near-linear domination in $\|\nabla F\|$ provides a smaller constant than near-quadratic domination.

	\paragraph{Setup and estimation.}
	We fit by FB-GD with backtracking to obtain an almost critical point $\theta^\star=(U^\star,V^\star)$, then probe a ball $B_2(\theta^\star;R)$ by uniform sampling.
	For each probe $\theta_i$ we compute the optimality gap $\Delta_i:=F(\theta_i)-F(\theta^\star)$ and gradient norm $g_i:=\|\nabla F(\theta_i)\|$. For a grid $\alpha\in[1,2]$ we estimate
	\[
	\widehat{\tau}(\alpha)\;=\;\max_i \frac{\Delta_i}{g_i^\alpha}.
	\]
	We report the curve $\alpha\mapsto\widehat{\tau}(\alpha)$ and highlight $\{\alpha:\widehat{\tau}(\alpha)\le 10\widehat{\tau}(1)\}$ (for the real-data run below: $R=0.006$, $m=600$ probes in each subplot).
	\vspace{-0.2em}

	\begin{figure}[h!]
	  \centering
	  \includegraphics[width=1\linewidth]{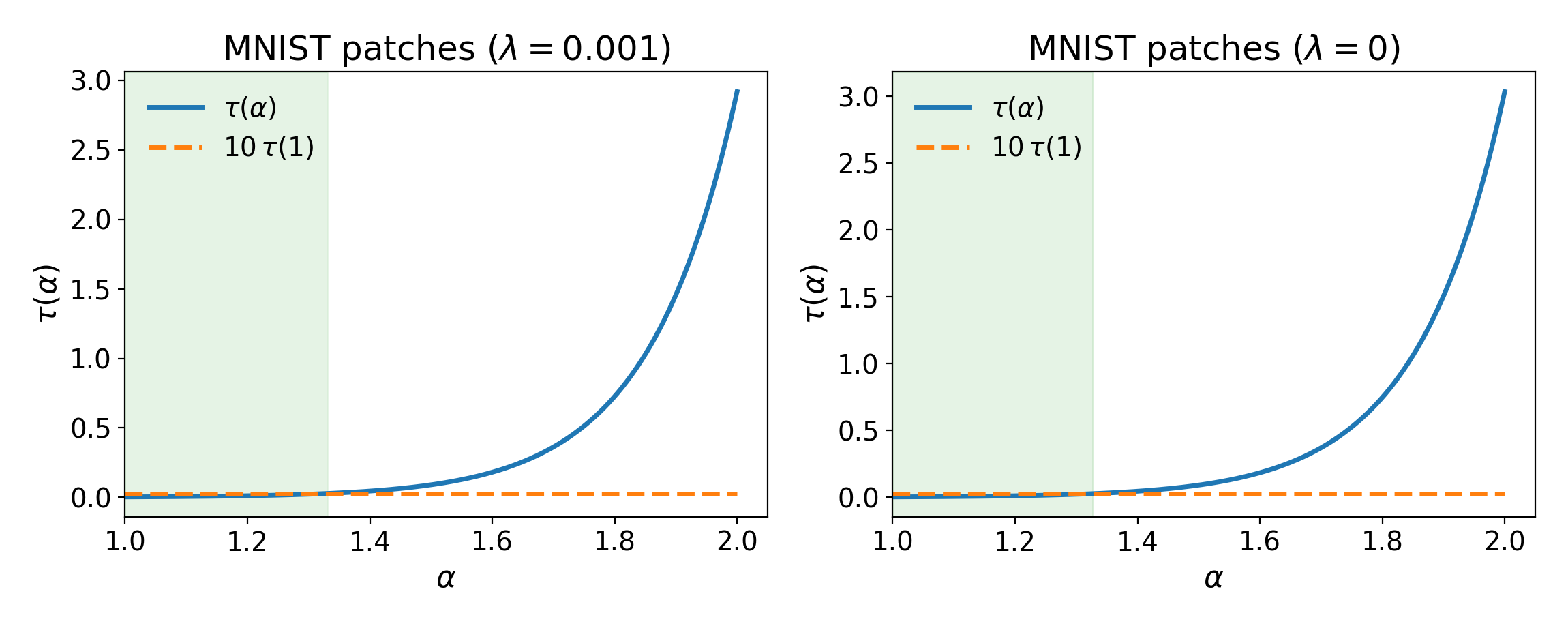}
	  \caption{\textbf{Real-data polynomial MF (MNIST patches).} $R=0.006$, $m=600$. Estimated envelopes $\alpha\mapsto \tau(\alpha)$ for two settings: regularized ($\lambda=10^{-3}$; range $[0.0028,\,2.919]$) and unregularized ($\lambda=0$; range $[0.0028,\,3.034]$). The green region is $\{\alpha:\tau(\alpha)\le 10\tau(1)\}$, mirroring the synthetic plots.}
	  \label{fig:realdata_mf_append}
		\end{figure}
		\paragraph{Explanation.}
		In \Cref{fig:realdata_mf_append} the two subplots compare how ridge regularization affects the local envelope around an almost-critical point fitted on MNIST patches.
		Both regularized and unregularized runs exhibit a broad ``flat'' regime near $\alpha=1$ (green band), while $\tau(\alpha)$ grows with $\alpha$ and is largest near $\alpha=2$.
		The similarity of the curves here suggests that, at this radius, the local PL-type behavior is driven primarily by the data-fit term rather than the ridge term.

\subsection{Deep Linear Matrix Factorization}
\paragraph{Problem.}
Approximate $X\in\mathbb{R}^{p\times n}$ by a product of $L$ linear factors
\[
X\approx W_0 W_1\cdots W_{L-1},
\qquad
F(W_0,\ldots,W_{L-1})\;=\;\frac{1}{2n}\,\bigl\|W_0\cdots W_{L-1}-X\bigr\|_F^2
+\frac{\lambda}{2}\sum_{\ell=0}^{L-1}\|W_\ell\|_F^2,
\]
with $W_0\in\mathbb{R}^{p\times r_1}$, $W_\ell\in\mathbb{R}^{r_\ell\times r_{\ell+1}}$, and $W_{L-1}\in\mathbb{R}^{r_{L-1}\times n}$. This loss is a \emph{polynomial of degree $2L$} in the parameters.

\paragraph{Setup.}
We train by FB-GD (with backtracking) to obtain $\theta^\star=(W_0^\star,\ldots,W_{L-1}^\star)$ and probe $B_2(\theta^\star;R)$ as above, producing pairs $(\Delta_i,g_i)$.
We compute the empirical envelope
\[
\widehat{\tau}(\alpha)\;=\;\max_i \frac{\Delta_i}{g_i^\alpha},\qquad \alpha\in[1,2].
\]

		\paragraph{Synthetic setup.}
		We also report the same envelope estimation for a synthetic instance (low-rank ground truth plus Gaussian noise), using $R=0.002$ and $m=2000$ probes:
		\begin{figure}[h!]
		  \centering
		  \includegraphics[width=.6\linewidth]{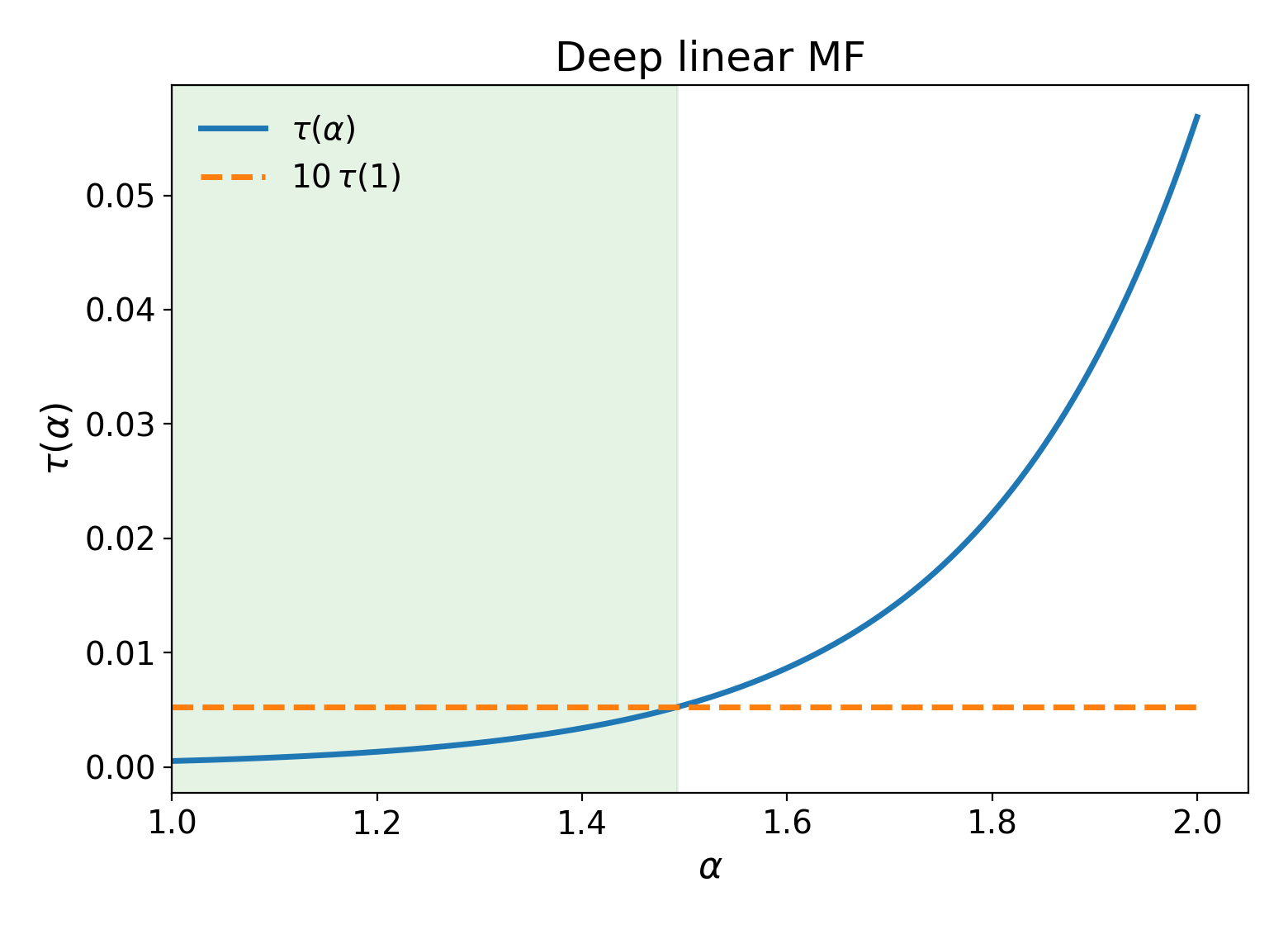}
		  \caption{\textbf{Synthetic deep linear MF.} $\lambda=10^{-3}$, $R=0.002$, $m=2000$, $L=4$. Variation of $\widehat{\tau}(\alpha)$ over $\alpha\in[1,2]$ (synthetic $X$). Range of $\widehat{\tau}(\alpha)$ over $\alpha\in[1,2]$: $[5.3\!\times\!10^{-4},\,0.0569]$.}
		  \label{fig:synthetic_deepmf_append}
		\end{figure}
		\paragraph{Explanation.}
		In \Cref{fig:synthetic_deepmf_append}, on a synthetic instance, the estimated envelope stays small for $\alpha$ close to $1$ and increases as $\alpha\uparrow 2$.
		This reflects the fact that, even for polynomial objectives, the best local domination exponent can be strictly less than $2$ on a fixed neighborhood, and the constant for $\alpha=2$ can be much larger than for $\alpha\approx 1$.

	\paragraph{Real-data deep linear MF.}
	We repeat the same procedure on a real dataset by extracting $8\times 8$ patches from MNIST training images and treating them as columns of $X$.
	We compare a regularized run ($\lambda>0$) to an unregularized run with $\lambda=0$, using $R=0.001$ and $m=700$ probes in each subplot.
	\begin{figure}[h!]
	  \centering
	  \includegraphics[width=1\linewidth]{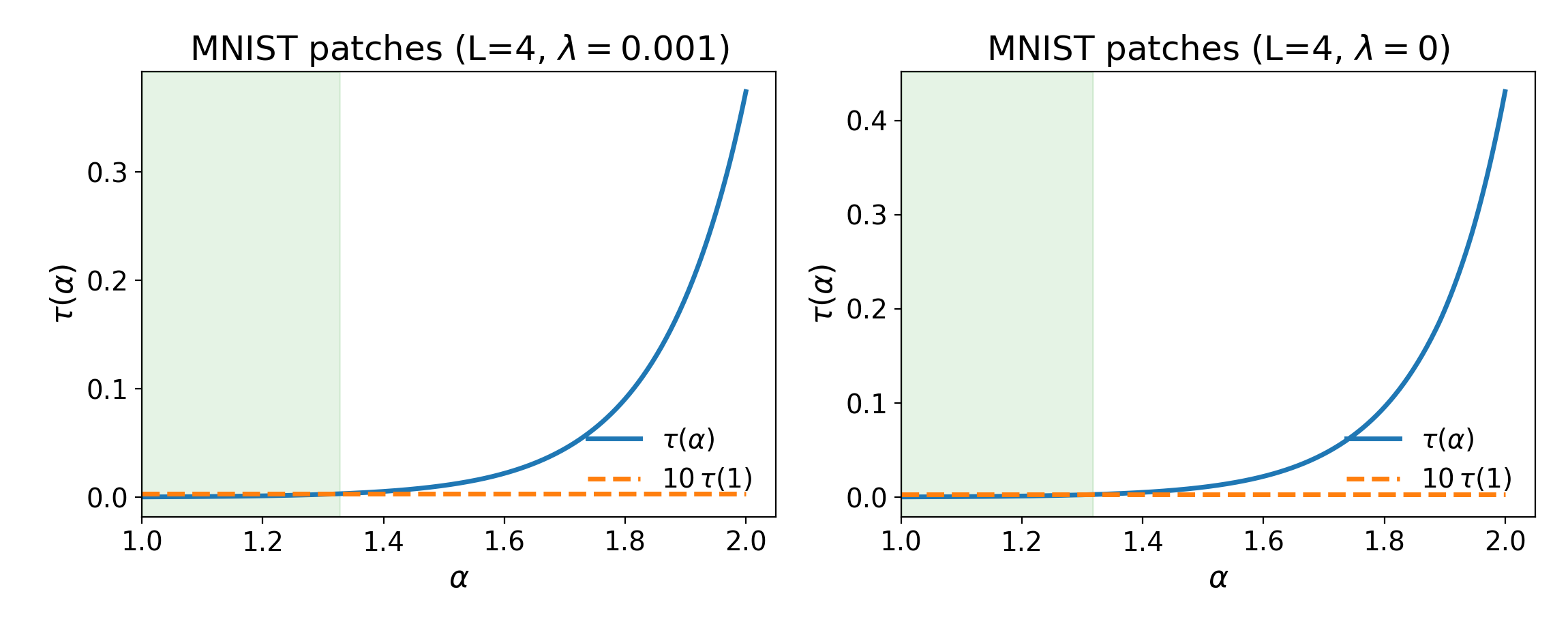}
	  \caption{\textbf{Real-data deep linear MF (MNIST patches).} $R=0.001$, $m=700$, $L=4$. Estimated envelopes $\alpha\mapsto \tau(\alpha)$ for a regularized run ($\lambda=10^{-3}$; range $[3.3\!\times\!10^{-4},\,0.374]$) and an unregularized run ($\lambda=0$; range $[2.9\!\times\!10^{-4},\,0.431]$). The green region is $\{\alpha:\tau(\alpha)\le 10\tau(1)\}$, matching the synthetic plot convention.}
	  \label{fig:realdata_deepmf_append}
		\end{figure}
		\paragraph{Explanation.}
		In \Cref{fig:realdata_deepmf_append} we repeat the envelope computation on MNIST patches.
		The green region highlights the near-$\alpha=1$ regime where $\tau(\alpha)\le 10\tau(1)$ and the curve is comparatively flat, while the envelope grows sharply as $\alpha$ approaches $2$.
		The $\lambda=0$ run has a slightly larger maximum envelope over $\alpha\in[1,2]$, consistent with regularization improving local conditioning in the deep parameterization.

	\paragraph{Degenerate critical point (unregularized).}
	For depth $L\ge 3$ and $\lambda=0$, the trivial point $\theta^\star=0$ is a stationary point of deep linear MF, and the local geometry can be much flatter than quadratic.
	We estimate the envelope around $\theta^\star=0$ on the same MNIST-patch matrix using $R=0.005$ and $m=8000$ probe points:
	\begin{figure}[h!]
	  \centering
	  \includegraphics[width=.6\linewidth]{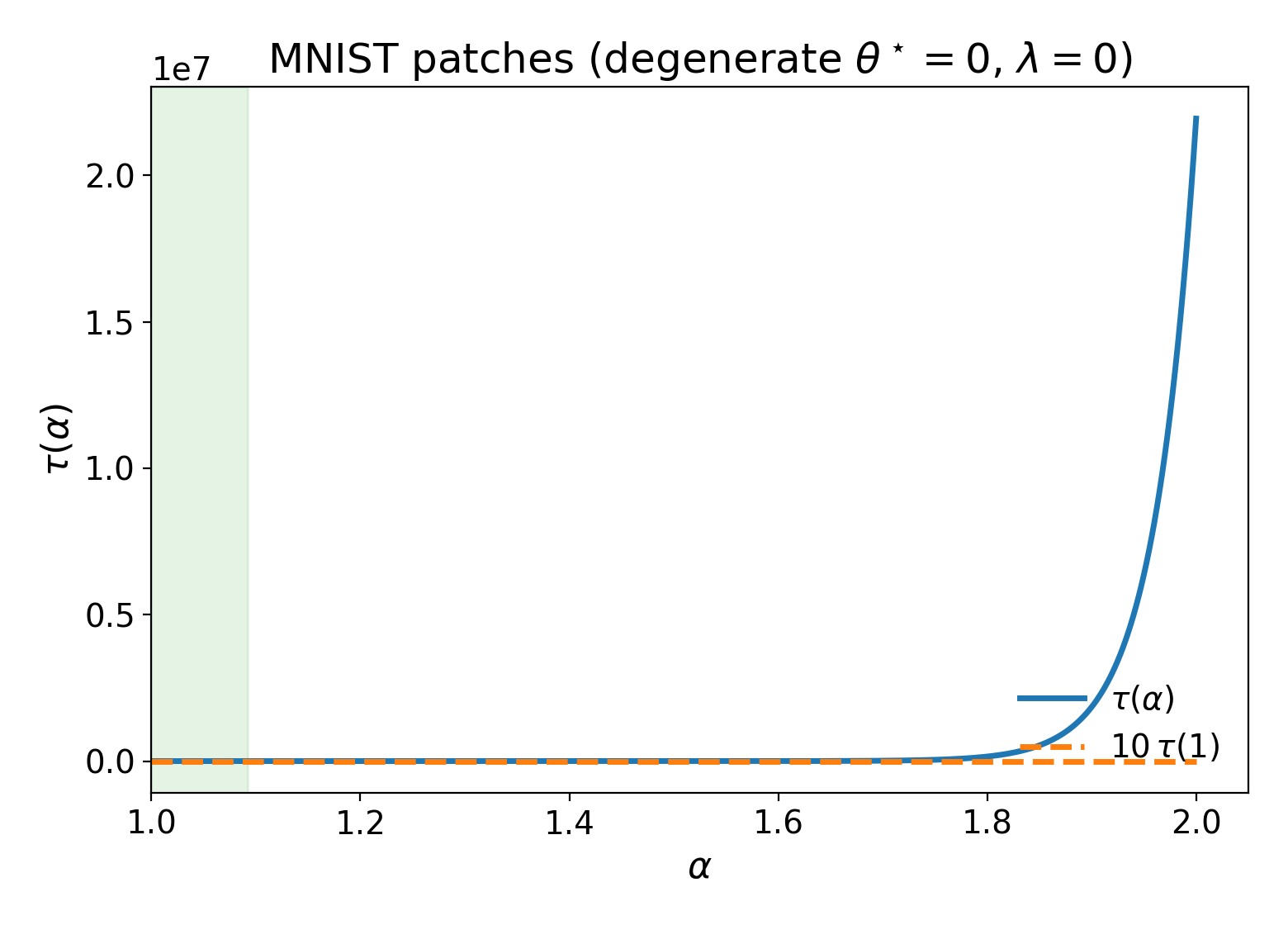}
	  \caption{\textbf{Real-data deep linear MF around the degenerate point $\theta^\star=0$ (MNIST patches).} $\lambda=0$, $R=0.005$, $m=8000$, $L=4$. Estimated envelope $\alpha\mapsto \tau(\alpha)$ around $\theta^\star=0$. Range of $\tau(\alpha)$ over $\alpha\in[1,2]$: $[3.9\!\times\!10^{-4},\,2.2\!\times\!10^{7}]$.}
	  \label{fig:realdata_deepmf_zero_crit_append}
		\end{figure}
		\paragraph{Explanation.}
		In \Cref{fig:realdata_deepmf_zero_crit_append} we intentionally probe around a \emph{degenerate} stationary point.
		The green region marks the small ``flat'' regime near $\alpha=1$ (where $\tau(\alpha)\le 10\tau(1)$), but the envelope then grows extremely rapidly, yielding a very large $\tau(2)$.
		This illustrates that, without regularization, deep linear MF can have regions where a quadratic ($\alpha=2$) PL-type inequality holds only with an enormous constant (or effectively fails to be informative), even though smaller exponents around $\alpha\simeq 1$ can still yield reasonable constants on the same neighborhood.

\end{document}

%% file: defns.tex
\newcommand{\ms}{\mathsf{m}}
\newcommand{\Ts}{\mathsf{T}}

\newcommand{\gv}{{\bf g}}

\newcommand{\uv}{{\bf u}}
\newcommand{\vv}{{\bf v}}
\newcommand{\wv}{{\bf w}}
\newcommand{\xv}{{\bf x}}

\DeclareMathOperator\E{E}

\DeclareMathOperator\R{R}

\def\textiid{i.i.d.\@\xspace}
\newcommand\iid{\ifmmode\text{ i.i.d. } \else \textiid \fi}

\newcommand{\beqs}{\begin{equation*}}
\newcommand{\eeqs}{\end{equation*}}
\newcommand{\beq}{\begin{equation}}
\newcommand{\eeq}{\end{equation}}